\documentclass[12pt,a4paper]{amsart}
\usepackage{etoolbox}
\usepackage{amssymb}
\usepackage{mathtools}
\usepackage[a4paper,left=2.25cm,right=2.25cm,top=3cm,bottom=3cm]{geometry}
\usepackage{tikz}
\usetikzlibrary{calc,decorations.pathmorphing,decorations.markings,decorations.pathreplacing,matrix,fit}
\usepackage{hyperref}
\usepackage{euler}
\allowdisplaybreaks[1]
\renewcommand\ss{\scriptstyle}
\newcommand\sss{\scriptscriptstyle}

\newcommand\stc{\tikz[scale=0.3]{\draw (0,0) -- (1,1) -- (0,1) -- cycle;}}
\newcommand\FPL{\mathrm{FPL}}
\newcommand\sh{\sigma}
\newcommand\subsarg[2]{{[#2]\choose #1}}
\newcommand\subs{\subsarg{n}{N}}
\newcommand\LPN{LP(N)}
\newcommand\Ker{\mathop\mathrm{Ker}\nolimits}
\renewcommand\Im{\mathop\mathrm{Im}\nolimits}
\newcommand\defn[1]{{\bf #1}}
\theoremstyle{plain}
\newtheorem{thm}{Theorem}
\newtheorem*{thm*}{Theorem}
\newtheorem{conj}{Conjecture}
\newtheorem{prop}{Proposition}
\newtheorem{lem}{Lemma}
\newtheorem{cor}{Corollary}
\newtheorem{rmk}{Remark}
\makeatletter
\newcommand{\gettikzxy}[3]{
  \tikz@scan@one@point\pgfutil@firstofone#1\relax
\pgfmathsetmacro{#2}{\the\pgf@x/\linkpatternunit}
\pgfmathsetmacro{#3}{\the\pgf@y/\linkpatternunit}
}
\tikzset{label anchor/.code={%
    \let\tikz@auto@anchor=\pgfutil@empty
    \def\tikz@anchor{#1}
  },
  label anchor/.default=center
}
\makeatother
\tikzset{arrow/.style={postaction={decorate,thick,decoration={markings,mark = at position #1 with {\arrow{>}}}}},arrow/.default=0.5}
\tikzset{invarrow/.style={postaction={decorate,thick,decoration={markings,mark = at position #1 with {\arrow{<}}}}},invarrow/.default=0.5}
\newdimen\linkpatternunit%
\newcount\linkpatternsize%
\newcount\lpsize
\newif\iflinkpatterninverted
\newif\iflinkpatterntikzstarted
\newif\iflinkpatternboxed
\newif\iflinkpatternaxis
\newif\iflinkpatternstraightlines
\newif\iflinkpatternnumbered
\newif\iflinkpatternalias
\newif\iflinkpatternnode
\newif\iflinkpatterncentered
\newcount\linkpatternfused
\pgfkeys{/linkpattern/.cd,centered/.is if=linkpatterncentered,inverted/.is if=linkpatterninverted,numbered/.is if=linkpatternnumbered,tikzstarted/.is if=linkpatterntikzstarted,straight lines/.is if=linkpatternstraightlines,boxed/.is if=linkpatternboxed,axis/.is if=linkpatternaxis,vertexcolor/.store in=\linkpatternvertexcolor,edgecolor/.store in=\linkpatternedgecolor,boxcolor/.code={\linkpatternboxedtrue\def\linkpatternboxcolor{#1}},tikzoptions/.style={every linkpattern/.append style={#1}},size/.code={\linkpatternsize=#1},numbering0/.code={\def\lpnumbering{#1}\def\linkpatternnumbering{#1}},numbering/.style={numbered,numbering0={#1}},unit/.code={\linkpatternunit=#1},height/.store in=\linkpatternheight,shape/.store in=\linkpatternshape,looseness/.store in=\linkpatternlooseness,squareness/.store in=\linkpatternsquareness,extra space/.store in=\linkpatternextraspace,width/.store in=\linkpatternwidth,alias/.is if=linkpatternalias,pos/.store in=\linkpatternpos,labeloptions/.style={labeloptionslist/.append style={#1}},labeloptionslist/.style={inner sep=2pt,font=\scriptsize,execute at begin node=$,execute at end node=$,label anchor={#1+180}},nodeon/.is if=linkpatternnode,node/.style={nodeon,nodeoptionslist/.append style={#1}},nodeoptionslist/.style={},
pipedream/.style={shape=pipedream,looseness=0,straight lines,numbering0=tangle},
tangle/.style={shape=tangle,numbering0=tangle},
every linkpattern/.style={x=\linkpatternunit,y=\linkpatternunit},
fused/.code={\linkpatternfused=#1},
vertex/.style={circle,thin,solid,draw=black,fill=\linkpatternvertexcolor,inner sep=1.5pt,draw opacity=1,transform shape},
edge/.style={very thick,solid,draw=\linkpatternedgecolor,draw opacity=1}}
\linkpatterncenteredfalse%
\linkpatterninvertedfalse%
\linkpatternnumberedfalse%
\linkpatterntikzstartedfalse%
\linkpatternboxedfalse%
\linkpatternaxistrue%
\linkpatternaliastrue%
\linkpatternstraightlinesfalse%
\def\linkpatternlooseness{0.2}
\def\linkpatternsquareness{0.35}
\def\linkpatternvertexcolor{red}%
\def\linkpatternedgecolor{blue}%
\def\linkpatternboxcolor{none}%
\def\linkpatternheight{0}
\def\linkpatternwidth{0}
\def\linkpatternshape{default}
\def\linkpatternnumbering{default}
\def\linkpatternpos{(0,0)}
\def\linkpatternextraspace{0}
\linkpatternnodefalse
\def\firstchar#1#2\empty{#1}%
\def\linkpatterndo#1#2{
\edef\param{\csname linkpattern#2\endcsname}
\edef\firstcharparam{\expandafter\firstchar\param\empty}
\expandafter\ifcat\firstcharparam a
\expandafter\ifx\csname linkpattern#1\param\endcsname\relax
\csname linkpattern#1unknown\endcsname
\else
\csname linkpattern#1\csname linkpattern#2\endcsname\endcsname
\fi
\else
\csname linkpattern#1unknown\endcsname
\fi
}%
\def\linkpatterncoordtangle{\ifnum\x>\lphalfsize\pgfmathparse{\lpsize+1-\x}\xdef\lpcoordx{\pgfmathresult}\xdef\lpcoordy{\lpheight}\xdef\lpangle{270}\else\xdef\lpcoordx{\x}\xdef\lpcoordy{-\lpheight}\xdef\lpangle{90}\fi}
\def\linkpatterncoordpipedream{\ifnum\x>\lphalfsize\pgfmathparse{\lpsize+1-\x-0.5}\xdef\lpcoordx{\pgfmathresult}\xdef\lpcoordy{0}\xdef\lpangle{270}\else\pgfmathparse{0.5-\x}\xdef\lpcoordy{\pgfmathresult}\xdef\lpcoordx{0}\xdef\lpangle{0}\fi}
\def\linkpatterncoordrectangle{
\ifnum\x>\lptqsize
\pgfmathparse{\lpsize+1-\x-0.5}\xdef\lpcoordx{\pgfmathresult}\xdef\lpcoordy{0}\xdef\lpangle{270}
\else\ifnum\x>\lphalfsize
\pgfmathparse{\x-\lptqsize-0.5}\xdef\lpcoordy{\pgfmathresult}\xdef\lpcoordx{\linkpatternwidth}\xdef\lpangle{180}
\else\ifnum\x>\linkpatternheight
\pgfmathparse{\x-\linkpatternheight-0.5}\xdef\lpcoordx{\pgfmathresult}\xdef\lpcoordy{-\linkpatternheight}\xdef\lpangle{90}
\else
\pgfmathparse{0.5-\x}\xdef\lpcoordy{\pgfmathresult}\xdef\lpcoordx{0}\xdef\lpangle{0}
\fi\fi\fi
}%
\def\linkpatternsetsizeunknown{
\global\lpsize=\linkpatternsize
\if\linkpatternheight0
\xdef\maxsep{0}
\foreach \x/\xx in \mylist%
{%
\edef\tempx{\withoutprime{\x}}
\edef\tempxx{\withoutprime{\xx}}
\pgfmathparse{max(\maxsep,abs(\tempx-\tempxx))}
\xdef\maxsep{\pgfmathresult}
}%
\pgfmathparse{0.25+0.8*\linkpatternsquareness*\maxsep}
\xdef\lpheight{\pgfmathresult}
\else
\xdef\lpheight{\linkpatternheight}
\fi
}
\newcount\tempsize
\def\linkpatternrightmostunknown{
\global\lpsize=0
\global\tempsize=0
\foreach\x/\labx in \linkpatternnumbering
{
\edef\tempx{\withoutprime{\x}}
\ifnum\lpsize<\tempx\global\lpsize=\tempx\fi
\global\advance\tempsize by 1
}
\ifnum\tempsize>\lpsize\global\lpsize=\tempsize\fi
}%
\def\linkpatternrightmostdefault{
\global\lpsize=0
\global\tempsize=0
\foreach \x/\y in \mylist
{
\edef\tempx{\withoutprime{\x}}
\ifnum\lpsize<\tempx\global\lpsize=\tempx\fi
\ifx\x\y
\global\advance\tempsize by 1
\else
\edef\tempy{\withoutprime{\y}}
\ifnum\lpsize<\tempy\global\lpsize=\tempy\fi%
\global\advance\tempsize by 2
\fi
}
\ifnum\tempsize>\lpsize\global\lpsize=\tempsize\fi
}%
\def\linkpatternrightmosttangle{
\global\lpsize=0
\global\tempsize=0
\foreach \x/\y in \mylist
{
\edef\tempx{\withoutprime{\x}}
\ifnum\lpsize<\tempx\global\lpsize=\tempx\fi
\ifx\x\y
\global\advance\tempsize by 1
\else
\edef\tempy{\withoutprime{\y}}
\ifnum\lpsize<\tempy\global\lpsize=\tempy\fi%
\global\advance\tempsize by 2
\fi
}
\global\advance\lpsize by\lpsize
\ifnum\tempsize>\lpsize\global\lpsize=\tempsize\fi
}%

\newcommand\linkpattern[2][]{
{
\pgfkeys{/linkpattern/.cd,#1}
\edef\mylist{#2}
\def\primetest##1'{}%
\def\hasaprime##1{\expandafter\primetest##1''}
\def\internalwithoutprime##1'{##1}%
\def\withoutprime##1{\if\hasaprime##1 %
\expandafter\internalwithoutprime##1\else ##1\fi}%
\iflinkpatternnumbered%
\iflinkpatterninverted
\tikzset{/linkpattern/lbl/.style n args={3}{label={[/linkpattern/labeloptionslist=-##1,##3] ##1:##2}}}%
\else%
\tikzset{/linkpattern/lbl/.style n args={3}{label={[/linkpattern/labeloptionslist=##1,##3] ##1:##2}}}%
\fi%
\else%
\tikzset{/linkpattern/lbl/.style={}}%
\fi%
\tikzifinpicture{\linkpatterntikzstartedtrue%
\begin{scope}[shift=\linkpatternpos,/linkpattern/every linkpattern]
}{%
\linkpatterntikzstartedfalse%
\iflinkpatterncentered
\begin{tikzpicture}[baseline=(current  bounding  box.center),/linkpattern/every linkpattern]%
\else
\begin{tikzpicture}[baseline=0,/linkpattern/every linkpattern]%
\fi
}%
\begin{scope}[local bounding box=link pattern box]
\iflinkpatterninverted%
\begin{scope}[yscale=-1]%
\fi%
\linkpatterndo{setsize}{shape}
\ifnum\lpsize=0
\linkpatterndo{rightmost}{numbering}
\fi
\pgfmathtruncatemacro{\lphalfsize}{\lpsize/2}
\linkpatterndo{numbering}{numbering}
\iflinkpatternboxed
\linkpatterndo{drawbox}{shape}
\else
\iflinkpatternaxis
\linkpatterndo{drawaxis}{shape}
\fi
\fi
\foreach\xx/\xlab/\opt in \lpnumbering
{
\ifx\xlab\opt\def\opt{}\fi
\if\hasaprime\xx %
\pgfmathtruncatemacro{\xx}{\lpsize+1-\withoutprime{\xx}}
\fi
\ifnum\linkpatternfused>1
\pgfmathsetmacro{\x}{0.4*(0.5+\linkpatternfused*(0.5+floor((\xx-1)/\linkpatternfused)))+0.6*\xx}
\else
\def\x{\xx}
\fi
\linkpatterndo{coord}{shape}
\iflinkpatternalias\def\xlabb{\xlab}\else\def\xlabb{\xx}\fi
\path (\lpcoordx,\lpcoordy) coordinate[/linkpattern/vertex,/linkpattern/lbl={\lpangle+180}{\xlab}{\opt},alias=v\xlabb] (v\xx) ++(\lpangle:\linkpatternunit) coordinate[alias=vv\xlabb] (vv\xx); 
}
\foreach \a/\b/\c in \mylist
{
\if\hasaprime\a %
\pgfmathtruncatemacro{\a}{\lpsize+1-\withoutprime{\a}}
\fi
\ifx\b\c\def\c{}\fi
\draw[/linkpattern/edge]
\ifx\a\b
(v\a)
\c
--
++(0,\lpheight);
\else
\pgfextra{
\if\hasaprime\b %
\pgfmathtruncatemacro{\b}{\lpsize+1-\withoutprime{\b}}
\fi
\gettikzxy{(v\a)}{\ax}{\ay}
\gettikzxy{(v\b)}{\bx}{\by}
\gettikzxy{(vv\a)}{\axx}{\ayy}
\gettikzxy{(vv\b)}{\bxx}{\byy}
\pgfmathsetmacro{\dist}{sqrt((\ax-\bx)*(\ax-\bx)+(\ay-\by)*(\ay-\by))}
\pgfmathsetmacro{\abx}{(\axx-\ax)*\dist*\linkpatternsquareness+(\bx-\ax)*\linkpatternlooseness)}
\pgfmathsetmacro{\aby}{(\ayy-\ay)*\dist*\linkpatternsquareness+(\by-\ay)*\linkpatternlooseness)}
\pgfmathsetmacro{\bax}{(\bxx-\bx)*\dist*\linkpatternsquareness+(\ax-\bx)*\linkpatternlooseness)}
\pgfmathsetmacro{\bay}{(\byy-\by)*\dist*\linkpatternsquareness+(\ay-\by)*\linkpatternlooseness)}
}
(v\a)
\c
\iflinkpatternstraightlines
\pgfextra{
\pgfmathsetmacro{\t}{((\ax-\bx)*\bay-(\ay-\by)*\bax)/(\aby*\bax-\abx*\bay)}
\pgfmathsetmacro{\abx}{\t*\abx}
\pgfmathsetmacro{\aby}{\t*\aby}
}
[rounded corners=0.2\linkpatternunit] -- ++(\abx,\aby) -- (v\b);
\else
.. controls ++(\abx,\aby) and ++(\bax,\bay) .. 
\fi
(v\b);
\fi
}
\end{scope}
\iflinkpatternnode
\node[fit=(link pattern box),/linkpattern/nodeoptionslist] {};
\fi
\iflinkpatterninverted
\end{scope}
\fi
\iflinkpatterntikzstarted
\end{scope}
\else%
\end{tikzpicture}%
\fi%
}}%
\newcommand\tanglelinkpattern[3][]{%
{
\pgfkeys{/linkpattern/.cd,#1}
\iflinkpatterninverted
\begin{tikzpicture}[/linkpattern/every linkpattern,baseline=\linkpatternunit]%
\else
\begin{tikzpicture}[/linkpattern/every linkpattern,baseline=-\linkpatternunit]%
\fi
\linkpattern[#1,tikzstarted,numbered=false]{#3}
\pgfmathtruncatemacro{\lptempsize}{2*\linkpatternsize}
\iflinkpatterninverted
\begin{scope}[yshift=0.5*\linkpatternunit]
\else
\begin{scope}[yshift=-0.5*\linkpatternunit]
\fi
\linkpattern[tangle,#1,tikzstarted,size=\lptempsize,
numbering=halftangle,
height=0.5]{#2}
\end{scope}
\end{tikzpicture}%
}}
\newcommand\diag[4][]{%
\pgfkeys{/linkpattern/.cd,#1}
\iflinkpatterntikzstarted\else%
\begin{tikzpicture}[scale=0.5]
\fi%
\iflinkpatterninverted%
\begin{scope}[yscale=-1]%
\fi%
\draw (0,0) grid (#2,#3);
\edef\mylist{#4}
\foreach\y/\x/\z in \mylist
{
\ifx\x\z
\draw[decorate,decoration={zigzag,
amplitude=1pt,segment length=5pt}]
(\x-0.5,#3) -- (\x-0.5,\y-0.5) node[circle,fill=black,inner sep=2pt] {} -- (#2,\y-0.5);
\else
\node at (\x-0.5,\y-0.5) {$\z$};
\fi
}
\iflinkpatterninverted
\end{scope}
\fi
\iflinkpatterntikzstarted\else%
\end{tikzpicture}%
\fi%
}
\makeatletter
\tikzset{circle split part fill/.style  args={#1,#2}{%
 alias=tmp@name,
  postaction={%
    insert path={
     \pgfextra{%
     \pgfpointdiff{\pgfpointanchor{\pgf@node@name}{center}}%
                  {\pgfpointanchor{\pgf@node@name}{east}}%
     \pgfmathsetmacro\insiderad{\pgf@x}
      \fill[#1] (\pgf@node@name.base) ([xshift=-\pgflinewidth]\pgf@node@name.east) arc
                          (0:180:\insiderad-\pgflinewidth)--cycle;
      \fill[#2] (\pgf@node@name.base) ([xshift=\pgflinewidth]\pgf@node@name.west)  arc
                           (180:360:\insiderad-\pgflinewidth)--cycle;                    }}}}}  
 \makeatother
\tikzset{bdot/.style={circle,circle split,draw,circle split part fill={black,white},thin,inner sep=1pt}}%
\tikzset{wdot/.style={circle,circle split,draw,circle split part fill={white,black},thin,inner sep=1pt}}%
\newcommand\circlelinkpattern[2][]{
{
\pgfkeys{/linkpattern/.cd,#1}
\iflinkpatterntikzstarted\else%
\begin{tikzpicture}[/linkpattern/every linkpattern]%
\fi%
\iflinkpatterninverted%
\begin{scope}[yscale=-1]%
\fi%
\global\lpsize=\linkpatternsize
\edef\mylist{#2}
\foreach \x/\y in \mylist
{
\ifnum\x>\lpsize\global\lpsize=\x\fi
\ifnum\y>\lpsize\global\lpsize=\y\fi
}
\iflinkpatternaxis
\draw (0,0) circle (1);
\fi
\foreach\x in {1,...,\lpsize}
{
\pgfmathparse{(0.3*floor((\x-1)/\linkpatternfused)+0.7*((\x-0.5)/\linkpatternfused-0.5))*\linkpatternfused*360/\lpsize}
\coordinate[/linkpattern/vertex] (v\x) at (\pgfmathresult:1);
}
\foreach \x/\y/\z in \mylist
{
\ifx\y\z%
\draw[/linkpattern/edge] (v\x) .. controls ($0.5*(v\x)$) and  ($0.5*(v\y)$) .. (v\y);
\else
\draw[/linkpattern/edge] \z (v\x) .. controls ($0.5*(v\x)$) and  ($0.5*(v\y)$) .. (v\y);
\fi
}
\iflinkpatternnumbered%
\pgfmathparse{\lpsize/\linkpatternfused}
\global\lpsize=\pgfmathresult
\def\linkpatternnumbering{1,...,\lpsize}
\newdimen\angle
\foreach\x/\xx/\opt in \linkpatternnumbering
{
  \pgfmathsetmacro{\angle}{360/\lpsize*(\x-1)}
\ifx\xx\opt%
  \node[outer sep=1pt,anchor=180+\angle] at (\angle:1) {$\scriptstyle\xx$}; 
\else
  \node[outer sep=1pt,anchor=180+\angle,\opt] at (\angle:1) {$\scriptstyle\xx$}; 
\fi
}
\fi%
\iflinkpatterninverted%
\end{scope}
\fi%
\iflinkpatterntikzstarted\else%
\end{tikzpicture}%
\fi%
}}%
\newdimen{\loopcellsize}
\tikzset{bgplaq/.style={draw=black,fill=\linkpatternboxcolor}}
\def\plaqwest{}
\def\plaqeast{}
\def\plaqnorth{}
\def\plaqsouth{}
\pgfkeys{/linkpattern/.cd,west/.store in=\plaqwest,east/.store in=\plaqeast,north/.store in=\plaqnorth,south/.store in=\plaqsouth}
\def\plaqname{plaq}
\newcommand\plaq[2][]{
\node[bgplaq,rectangle,draw,use as bounding box,minimum size=\loopcellsize,transform shape] (\plaqname) {};
\pgfkeys{/linkpattern/.cd,#1}
\ifx#2\empty\else
\begin{scope}[x=\loopcellsize,y=\loopcellsize]
\csname plaq#2\endcsname
\end{scope}\fi
}
\tikzset{loop/.code={\def\plaqname{loop-\the\pgfmatrixcurrentrow-\the\pgfmatrixcurrentcolumn}},loop/.append style={matrix,row sep={\loopcellsize,between origins},column sep={\loopcellsize,between origins}}}
\newdimen{\cellsize}
\newcommand\bigboxes{\setlength{\cellsize}{18pt}\def\boxformat{}}
\newcommand\medboxes{\setlength{\cellsize}{14.22pt}\def\boxformat{}}
\newcommand\smallboxes{\setlength{\cellsize}{10pt}\def\boxformat{\scriptstyle}}

\medboxes
\tikzset{tableaubox/.style={draw=black,thin,sharp corners,solid,minimum size=\cellsize,inner sep=0pt}}
\tikzset{tableau/.style={matrix,name=tab,matrix anchor=tab-1-1.south west,inner sep=1pt,matrix of math nodes,cells={anchor=center,draw=black,thin,solid,arrows=-},nodes={tableaubox,execute at begin node=\boxformat},nodes in empty cells,row sep={\cellsize,between origins},column sep={\cellsize,between origins}}}
\newcommand\missingcell{|[draw=none]|}
\makeatletter
\newcommand\cellextra[1]{#1\expandafter\tikz@lib@matrix@start@cell}
\makeatother
\newcommand\hcell{\cellextra{\draw (-0.5*\cellsize,0.5*\cellsize) --++(\cellsize,0);}\missingcell}
\newcommand\vcell{\cellextra{\draw (-0.5*\cellsize,-0.5*\cellsize) --++(0,\cellsize);}\missingcell}
\newcommand\vhcell{\cellextra{\draw (-0.5*\cellsize,-0.5*\cellsize) --++(0,\cellsize) --++(\cellsize,0);}\missingcell}
\def\activate#1{\begingroup
  \lccode`\~=`#1%
  \lowercase{\endgroup \let~#1}%
  \catcode`#1=13\relax}
\activate &
\newcommand\tableau[1]{\tikz[baseline=0]
\node[tableau]{#1};}

\def\linkpatternboxcolor{pink!20!white}
\tikzset{bgplaq/.style={draw=black,dotted,fill=\linkpatternboxcolor}}
\tikzset{mypath/.style={ultra thick,red!50!black}}
\tikzset{emptypath/.style={double,thin,gray}}
\def\lozXY(#1,#2){
\begin{scope}[x={(0.866cm,-0.5cm)},y={(0.866cm,0.5cm)}]
\draw[fill=cyan!30!white] (#1,#2) -- ++(1,0) -- ++(0,1) -- ++(-1,0) -- cycle;
\end{scope}
}
\def\lozX(#1,#2){
\begin{scope}[x={(0cm,-1cm)},y={(0.866cm,0.5cm)},shift={(\a*0.866cm+0.866cm,\a*0.5cm+0.5cm)}]
\draw[fill=pink] (#1,#2) -- ++(1,0) -- ++(0,1) -- ++(-1,0) -- cycle;
\end{scope}
}
\def\lozY(#1,#2){
\begin{scope}[x={(0.866cm,-0.5cm)},y={(0cm,1cm)},shift={(\a*0.866cm+0.866cm,-\a*0.5cm-0.5cm)}]
\draw[fill=orange!50!white] (#1,#2) -- ++(1,0) -- ++(0,1) -- ++(-1,0) -- cycle;
\end{scope}
}
\def\nilpX(#1,#2){
\begin{scope}[x={(0cm,-1cm)},y={(0.5cm,0.5cm)},shift={(\a*0.5cm+0.5cm,\a*0.5cm+0.5cm)}]
\draw[mypath] (#1,#2) ++(0.5,0) -- ++(0,1);
\end{scope}
}
\def\nilpY(#1,#2){
\begin{scope}[x={(0.5cm,-0.5cm)},y={(0cm,1cm)},shift={(\a*0.5cm+0.5cm,-\a*0.5cm-0.5cm)}]
\draw[mypath] (#1,#2) ++(0,0.5) -- ++(1,0);
\end{scope}
}
\def\dimXY(#1,#2){
\begin{scope}[x={(0.866cm,-0.5cm)},y={(0.866cm,0.5cm)}]
\draw[mypath] (#1,#2) ++(0.333,0.333) -- ++(0.333,0.333);
\draw[emptypath] (#1,#2) ++(0.333,0.333) -- ++(-0.333,0.166);
\draw[emptypath] (#1,#2) ++(0.333,0.333) -- ++(0.166,-0.333);
\draw[emptypath] (#1,#2) ++(0.666,0.666) -- ++(0.333,-0.166);
\draw[emptypath] (#1,#2) ++(0.666,0.666) -- ++(-0.166,0.333);
\end{scope}
}
\def\dimX(#1,#2){
\begin{scope}[x={(0cm,-1cm)},y={(0.866cm,0.5cm)},shift={(\a*0.866cm+0.866cm,\a*0.5cm+0.5cm)}]
\draw[mypath] (#1,#2) ++(0.333,0.666) -- ++(0.333,-0.333);
\draw[emptypath] (#1,#2) ++(0.333,0.666) -- ++(0.166,0.333);
\draw[emptypath] (#1,#2) ++(0.333,0.666) -- ++(-0.333,-0.166);
\draw[emptypath] (#1,#2) ++(0.666,0.333) -- ++(0.333,0.166);
\draw[emptypath] (#1,#2) ++(0.666,0.333) -- ++(-0.166,-0.333);
\end{scope}
}
\def\dimY(#1,#2){
\begin{scope}[x={(0.866cm,-0.5cm)},y={(0cm,1cm)},shift={(\a*0.866cm+0.866cm,-\a*0.5cm-0.5cm)}]
\draw[mypath] (#1,#2) ++(0.333,0.666) -- ++(0.333,-0.333);
\draw[emptypath] (#1,#2) ++(0.333,0.666) -- ++(0.166,0.333);
\draw[emptypath] (#1,#2) ++(0.333,0.666) -- ++(-0.333,-0.166);
\draw[emptypath] (#1,#2) ++(0.666,0.333) -- ++(0.333,0.166);
\draw[emptypath] (#1,#2) ++(0.666,0.333) -- ++(-0.166,-0.333);
\end{scope}
}
\newcommand{\loz}[3]{%
\begin{scope}[scale=0.75,xshift=-\a*0.866cm-\b*0.433cm-\c*0.433cm-1.732cm,yshift=\b*0.25cm-\c*0.25cm]
\foreach\co in #1 { \expandafter\lozXY\co }
\foreach\co in #2 { \expandafter\lozX\co }
\foreach\co in #3 { \expandafter\lozY\co }
\end{scope}
}
\newcommand{\nilp}[2]{%
\begin{scope}[scale=0.9,xshift=-\a*0.5cm-\b*0.25cm-\c*0.25cm-1cm,yshift=\b*0.25cm-\c*0.25cm]
\begin{scope}
\clip (\a*0.5+1,\a*0.5) -- ++(\c*0.5,\c*0.5) -- ++(\b*0.5,-\b*0.5) -- ++(0,-\a) -- ++(-\c*0.5,-\c*0.5) -- ++(-\b*0.5,\b*0.5) -- cycle;
\draw[emptypath,xshift=1.5cm,rotate=-45,scale=0.707] (0,0) grid (\a+\b-1,\a+\c-1);
\end{scope}
\foreach\co in #1 { \expandafter\nilpX\co }
\foreach\co in #2 { \expandafter\nilpY\co }
\begin{scope}[shift={(\a*0.5cm+1cm,\a*0.5cm+0.5cm)}]
\foreach\i in {1,...,\a}
{
\node[/linkpattern/vertex] (a) at (0,-\i) {};
\node[/linkpattern/vertex] (b) at (\b*0.5+\c*0.5,-\i-\b*0.5+\c*0.5) {};
}
\end{scope}
\end{scope}
}
\newcommand{\dimer}[3]{%
\begin{scope}[scale=0.75,xshift=-\a*0.866cm-\b*0.433cm-\c*0.433cm-1.732cm,yshift=\b*0.25cm-\c*0.25cm]
\foreach\co in #1 { \expandafter\dimXY\co }
\foreach\co in #2 { \expandafter\dimX\co }
\foreach\co in #3 { \expandafter\dimY\co }
\end{scope}
}
\newcommand\tensor\otimes
\newcommand\Tensor\bigotimes
\newcommand\calO{{\mathcal O}}
\newcommand\onto{\twoheadrightarrow}
\newcommand\wt[1]{\widetilde{#1}}
\newcommand\into\hookrightarrow
\newcommand\infrom\hookleftarrow
\newcommand\from\longleftarrow
\newcommand\junk[1]{}
\newcommand\CC{{\mathbb C}}
\newcommand\NN{{\mathbb N}}
\newcommand\Spec{\mathop{\tt Spec}}

\newcommand\init{\mathtt{in}}
\newcommand\ZZ{{\mathbb Z}}
\newcommand\iso{\cong}

\title[Grassmann-Grassmann conormal varieties, integrability, and plane partitions\ \ ]
{Grassmann-Grassmann conormal varieties,\\ integrability, and plane partitions}
\author{Allen Knutson}
\address{Allen Knutson, Cornell, Ithaca, NY 14853, USA.}
\email{allenk@math.cornell.edu}
\author{Paul Zinn-Justin}
\address{Paul Zinn-Justin, School of Mathematics and Statistics, The University of Melbourne, 
Parkville, Victoria 3010, Australia.}
\email{pzinn@unimelb.edu.au}
\thanks{PZJ was supported by 
ERC grant ``LIC'' 278124
and ARC grant DP140102201.
Computerized checks of the results of this paper have been performed
with the help of Macaulay 2 \cite{M2}.}
\date{\today}

\long\def\rem#1{{\bf [#1]}}

\begin{document}

\begin{abstract}
We give a conjectural formula for sheaves supported on 
(irreducible) conormal varieties inside
the cotangent bundle of the Grassmannian, such that their equivariant $K$-class
is given by the partition function of an integrable loop model, and furthermore
their $K$-theoretic pushforward to a point is a solution of the level $1$ quantum
Knizhnik--Zamolodchikov equation. We prove these results in the case that
the Lagrangian is smooth (hence is the conormal bundle
to a subGrassmannian). To compute the pushforward to a point,
or equivalently to the affinization, we simultaneously
degenerate the Lagrangian and sheaf (over the affinization); 
the sheaf degenerates to a direct sum of cyclic modules over the
geometric components, which are in bijection with plane partitions, giving a
geometric interpretation 
to the Razumov--Stroganov correspondence satisfied by the loop model.
\end{abstract}

\maketitle

\tableofcontents

\section{Introduction}
In \cite{MO-qg} solutions to the rational Yang--Baxter equation (YBE) were
constructed using cohomology classes living on ``symplectic resolutions'',
in particular on the cotangent bundles of Grassmannians
(the main symplectic resolutions considered here).
Each class has a geometric origin, as the (usually reducible) singular support 
of a certain $\mathcal D$-module on the Grassmannian.  
In this paper we begin the study of a geometric origin for the
corresponding trigonometric YBE, constructing sheaves on these
cotangent bundles whose equivariant $K$-classes (once pushed to a point) 
we conjecture to satisfy the (trigonometric) quantum
Knizhnik--Zamolodchikov equation.

For each Schubert variety $X^r \subseteq Gr(n,N)$, we define a sheaf $\sh_r$
supported on its conormal variety $CX^r \subseteq T^* Gr(n,N)$
(definitions appearing in \S\ref{sec:gen}).
We give a conjectural formula for its equivariant $K$-theory class
as a rectangular-domain partition function.
The sheaf cohomology groups of $\sh_r$ 
are modules over the affinization $\mu(CX^r)$;
conjecturally, all cohomology vanishes unless the affinization map is 
birational, not dropping dimension. 
As the paper's title indicates, in this paper we focus attention on 
(and prove the conjectures in)
the case that $X^r$ is a subGrassmannian (the only time $X^r$ is smooth);
the affinization $\mu(CX^r)$ of $CX^r$ is then an $A_3$ quiver cycle. 

The Stanley--Reisner ring of a simplicial complex has a basis given 
by monomials, and a ``shelling'' of the simplicial complex gives a partitioning
of the monomials into orthants, allowing one to count those monomials
without inclusion-exclusion.
We do something closely analogous with a degeneration of 
$(\mu(CX^r),\mu_* \sh_r)$, giving a partitioning into cones of
a $\CC$-basis of the module $\mu_* \sh_r = H^0(CX^r;\ \sh_r)$.
Geometrically, this degenerates the base $\mu(CX^r)$ of the sheaf 
$\mu_* \sh_r$ to a highly reducible
scheme with very simple components, each bearing one summand of
the degenerate sheaf (though the shelling statement is stronger).

One key difference between the results of \cite{MO-qg,RTV-K} and ours 
is that we work directly with sheaves on $T^*Gr(n,N)$, 
not just $K$-classes thereof,
which is a sort of positivity statement; in the subGrassmannian case
we get a similar positivity
on their pushforwards from the vanishing of their higher cohomology 
(proposition \ref{prop:noHi}).
Another difference is that we work with (sheaves supported on)  
individual conormal varieties, whereas \cite{MO-qg,RTV-K} 
work with the stable basis (supported on unions).
Their basis is positive upper triangular w.r.t.\ ours,
the change of basis being given by maximal parabolic Kazhdan--Lusztig 
polynomials (the same change of basis which relates the corresponding integrable
models: the six-vertex model for the stable basis
and the Temperley--Lieb loop model for ours).

\junk{announce main result of future work as conjecture:
that there are sheaves [positivity] in $T^\ast Gr$
which are supported on conormals of Schubert [which we can build
explicitly geometrically] (1) whose $K$-class is given by a rectangular
finite domain partition function [variation of our unfinished
finite-domain paper: $K$-theory without crossings! should incidentally be also obtained by degeneration in the pair-of-$n\times N$-matrices 
picture] and (2) whose $K$-theoretic pushforward to a point [presumably, just character of global sections since
no higher sheaf cohomology] is given by [Laurent] polynomial
level $1$ solution of $q$KZ equation in the link pattern basis [really, only
in $Gr(n,N)$ with $N=2n$ -- should check one day what happens when $n\ne N/2$]}

\bigskip

\subsection{Various combinatorial gadgets}\label{ssec:combo}
We shall use interchangeably three sets in bijection:
\begin{itemize}
\item {\em Subsets}\/ of $[N] := \{1,\ldots,N\}$ of cardinality $n$;
\item {\em Young diagrams}\/ fitting inside an $(N-n)\times n$ rectangle;
\item {\em Link patterns}\/ of size $N$ with at most $\min(n,N-n)$ chords, that is, pairings of $N$ vertices on a line 
(some possibly being left unpaired), in such a way that there are at most $\min(n,N-n)$ pairings, drawn as chords,
and that the chords, as well as half-infinite lines coming out of unpaired vertices, may be drawn in a half-plane without any crossings.
\end{itemize}
We call the set of any of these objects $\subs$, and now describe the bijections.

Given a Young diagram, we associate to it a subset as follows: if we number from $1$ to $N$ 
the boundary edges of the Young diagram from bottom left to top right, then the subset $r=\{r_1,\ldots,r_n\}$
consists of all steps to the right. We always order increasingly elements
of the subset, that is we always have $r_i<r_{i+1}$, $i=1,\ldots,n-1$.
We shall denote by $\bar r$
its complement in $\{1,\ldots,N\}$, which therefore consists of all steps up.
\junk{PZJ: careful that these are not the parts of the partition
  associated to $r$; there's a shift. AK: I doubt this warning is necessary}

Given a link pattern with at most $\min(n,N-n)$ chords, then the subset consists of all ``closings'' of chords (vertices paired to another vertex to the left, where the numbering is from left to right), completed to cardinality $n$ by including unpaired vertices starting from the left.

Finally, we also provide the bijection from Young diagrams to link patterns, since it will be useful later: rotate the Young diagram
45 degrees clockwise, then fill the boxes of its {\em complement}\/ with the picture \tikz[scale=0.75,rotate=45,baseline=-3pt]{\plaq{b}};
the connectivity of the lines emerging at the top reproduces the link pattern.

On an example,
\newcommand\dottedcell{\cellextra{\draw[dotted] (-0.5*\cellsize,-0.5*\cellsize) --++(\cellsize,0);
\draw[dotted] (0.5*\cellsize,-0.5*\cellsize) --++(0,\cellsize);}\missingcell}
\begin{multline}\label{eq:bij}
r=\{1,4,6,7,10\}\in \subsarg{5}{11}:
\\
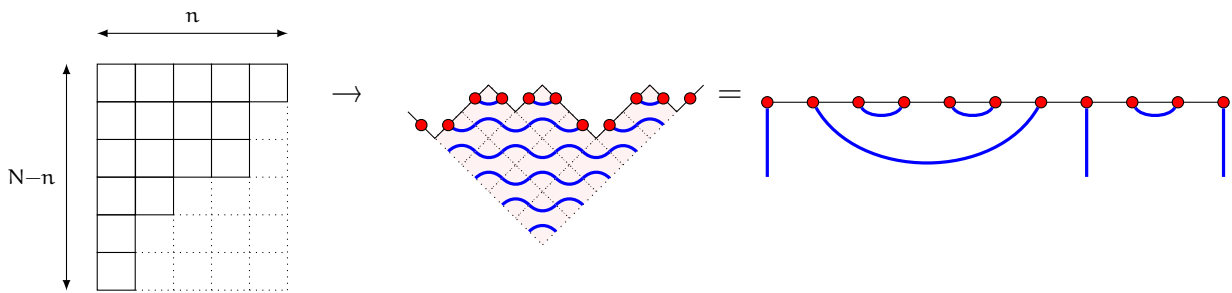
\begin{tikzpicture}[baseline=0]
\node[tableau]{
&&&&\\
&&&&\dottedcell\\
&&&&\dottedcell\\
&&\dottedcell&\dottedcell&\dottedcell\\
&\dottedcell&\dottedcell&\dottedcell&\dottedcell\\
&\dottedcell&\dottedcell&\dottedcell&\dottedcell\\
};
\draw[latex-latex] ([xshift=-0.4cm]tab-1-1.north west) -- node[left] {$\ss N-n$} ([xshift=-0.4cm]tab-6-1.south west);
\draw[latex-latex] ([yshift=0.4cm]tab-1-1.north west) -- node[above] {$\ss n$} ([yshift=0.4cm]tab-1-5.north east);
\end{tikzpicture}
\quad\rightarrow\quad
\setlength{\loopcellsize}{0.5cm}
\begin{tikzpicture}[baseline=0]
\node[rotate=45,inner sep=0pt] { \tikz{ 
\node[loop,outer sep=0pt]
{
\node[rectangle,draw=none,use as bounding box,minimum size=\loopcellsize] (loop-\the\pgfmatrixcurrentrow-\the\pgfmatrixcurrentcolumn) {};&&&&&\\
\plaq{b}&\plaq{b}&&&\\
\plaq{b}&\plaq{b}&\plaq{b}&&&\\
\plaq{b}&\plaq{b}&\plaq{b}&&&\\
\plaq{b}&\plaq{b}&\plaq{b}&\plaq{b}&\plaq{b}&\\
};
\begin{scope}[x=\loopcellsize,y=\loopcellsize]
\draw (loop-1-1.north west) -- node[/linkpattern/vertex] {} ++(0,-1) -- node[/linkpattern/vertex] {} ++(1,0) -- node[/linkpattern/vertex] {} ++(1,0) -- node[/linkpattern/vertex] {} ++(0,-1) -- node[/linkpattern/vertex] {} ++(1,0) -- node[/linkpattern/vertex] {} ++(0,-1) -- node[/linkpattern/vertex] {}  ++(0,-1) -- node[/linkpattern/vertex] {} ++(1,0) -- node[/linkpattern/vertex] {} ++(1,0) -- node[/linkpattern/vertex] {} ++(0,-1) -- node[/linkpattern/vertex] {} ++(1,0);
\end{scope}
}};
\end{tikzpicture}
=\ 
\linkpattern[inverted]{1/1,2/7,3/4,5/6,8/8,9/10,11/11}
\end{multline}

\junk{it'd be nice to indicate
  in the third picture the difference between ``unmatched going West''
  and ``unmatched going East''; without that extra data we can't determine $n$,
  only $N$. P: true but that would require extra explanations and it would just make things more confusing, I believe. 
  A: Well, without, we don't have a bijection. P: we do have a bijection for fixed $n$, not taking the union over $n$}

Also, denote by $|r|$ the number of boxes of $r$;
equivalently, $|r|=n(N-n)-\sum_{i=1}^n (r_i-i)$.

We endow $\subs$ with the following order relation:
$\subseteq$ denotes inclusion of Young diagrams, or equivalently
of the corresponding Schubert varieties. 
An example of the poset structure is shown in Fig.~\ref{fig:poset}.
\begin{figure}
\begin{tikzpicture}
\smallboxes
\begin{scope}[local bounding box=f34,shift={(0,0)}]
\linkpattern[tikzstarted,inverted] {1/4,2/3}
\node[tableau] at (3,0) {\vhcell 3&\hcell 4\\\vcell\\};
\end{scope}
\node[fit=(f34)] (y34) {};
\begin{scope}[local bounding box=f24,shift={(0,2)}]
\linkpattern[tikzstarted,inverted] {1/2,3/4}
\node[tableau] at (3,0) {&\hcell 4\\\vcell 2\\};
\end{scope}
\node[fit=(f24)] (y24) {};
\begin{scope}[local bounding box=f23,shift={(-2.5,4)}]
\linkpattern[tikzstarted,inverted] {1/2,3/3,4/4}
\node[tableau] at (3,0) {&\\\vcell 2&\missingcell 3\\};
\end{scope}
\node[fit=(f23)] (y23) {};
\begin{scope}[local bounding box=f14,shift={(2.5,4)}]
\linkpattern[tikzstarted,inverted] {1/1,2/2,3/4}
\node[tableau] at (3,0) {&\hcell\\\\};
\end{scope}
\node[tableau] at (5.5,4) {&\hcell 4\\\\\missingcell 1\\};
\node[fit=(f14)] (y14) {};
\begin{scope}[local bounding box=f13,shift={(0,6)}]
\linkpattern[tikzstarted,inverted] {2/3,1/1,4/4}
\node[tableau] at (3,0) {&\\&\missingcell 3\\\missingcell 1\\};
\end{scope}
\node[fit=(f13)] (y13) {};
\begin{scope}[local bounding box=f12,shift={(0,8)}]
\linkpattern[tikzstarted,inverted,height=0.5] {1/1,2/2,3/3,4/4}
\node[tableau] at (3,0) {&\\&\\\missingcell 1&\missingcell 2\\};
\end{scope}
\node[fit=(f12)] (y12) {};
\draw[->] (y34) -- (y24);
\draw[->] (y24) -- (y23);
\draw[->] (y24) -- (y14);
\draw[->] (y23) -- (y13);
\draw[->] (y14) -- (y13);
\draw[->] (y13) -- (y12);
\end{tikzpicture}
\caption{The poset structure for $\subsarg{2}{4}$.}\label{fig:poset}
\end{figure}

When this order is reformulated
in terms of subsets, it corresponds to pointwise greater or equal;
we shall therefore denote the opposite order $\le$, hoping this does not
create any confusion:
\[
r\subseteq s
\ \Leftrightarrow\ 
s\le r
\ \Leftrightarrow\ 
s_i\le r_i,\ i=1,\ldots,n
\]

A \defn{Completely Packed Loop configuration} (or CPL, in short) is an assignment of the two possible plaquettes
\tikz[baseline=0]{\plaq{a}} and \tikz[baseline=0]{\plaq{b}} to the faces of a $n\times N$ square grid, e.g., for
$N=4$, $n=2$, one CPL is
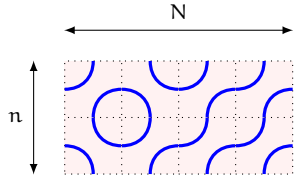
\begin{center}
\begin{tikzpicture}
\node[loop] { 
\plaq{a} & \plaq{b} & \plaq{a} & \plaq{a}
\\
\plaq{b} & \plaq{a} & \plaq{a} & \plaq{a}
\\
};
\draw[latex-latex] ([xshift=-0.4cm]loop-1-1.north west) -- node[left] {$\ss n$} ([xshift=-0.4cm]loop-2-1.south west);
\draw[latex-latex] ([yshift=0.4cm]loop-1-1.north west) -- node[above] {$\ss N$} ([yshift=0.4cm]loop-1-4.north east);
\end{tikzpicture}
\end{center}

We say that a CPL has top-connectivity given by $r\in\subs$ iff when following the paths made by the (blue) lines, the connectivity of the external edge midpoints obeys the following rules:
\begin{itemize}
\item Denoting (l,r,b,t) for a midpoint situated on the left, right,
  bottom, or top sides, the allowed connectivities are (b,l), (b,r),
  (b,t), (l,t), (t,t).
\item The connectivity of the $N$ midpoints across the top edge (ignoring
  connectivities outside the top side, i.e.,\ declaring a midpoint
  connected to the left or bottom to be unpaired) 
  reproduces the link pattern $r$.
\end{itemize}
Note that these conditions imply that if $r$ has $k$ chords, then a CPL with 
top-connectivity $r$
has $k$ pairings (t,t), $n-k$ (t,l) and $N-n-k$ (b,t).

\junk{link patterns should be upside down. and yes, if we went back
to the brauer model, that means link patterns should be read
clockwise, not counterclockwise... oops}

A CPL may have paths which close onto themselves; we call them \defn{loops}, and their number is denoted $|\text{loops}|$.

For example, the CPL above has top-connectivity $\linkpattern[inverted]{1/1,2/3,4/4}$, and $|\text{loops}|=1$.

\subsection{The geometric setup}\label{sec:gen}
Given two integers $n$ and $N$ such that $0\le n\le N$, 
consider the Grassmannian $Gr(n,N)=\{ V\leq \mathbb C^N: \dim V=n\}$ 
and its cotangent bundle $T^\ast Gr(n,N)$. $GL(N)$ acts on each of these,
as do its Borel subgroups $B_\pm$ and diagonal matrices $T_0 = B_+\cap B_-$. 
An additional circle $\mathbb C^\times$ acts on $T^\ast Gr(n,N)$ by scaling 
of the cotangent spaces.  Let $T:=T_0 \times \CC^\times$, 
with representation rings 
$K_{T_0} = \ZZ[z_1^\pm,\ldots,z_n^\pm], K_{\CC^\times} = \ZZ[t^\pm]$.

The $T_0$-fixed points in $Gr(n,N)$ 
(or equivalently, $T$-fixed points
in $T^\ast Gr(n,N)$, viewing $Gr(n,N)$ as its zero section)
are coordinate subspaces, thereby labeled by subsets
$r\in\subs$. Their $B_-$-orbits $X^r_o$ are called \defn{Schubert cells},
with conormal bundles denoted
$$ CX^r_\circ := \{(x,\vec v)\ :\ x \in X^r_\circ, \ \vec v \in T^*_x Gr(n,N),
\ \vec v \perp T_x (X^r_\circ) \} \quad \subseteq T^* Gr(n,N). $$
Their closures $X^r:=\overline{X^r_o}$ and $CX^r:=\overline{CX^r_o}$
we call \defn{Schubert varieties}
and \defn{conormal Schubert varieties}, respectively.

We shall consider certain $T$-equivariant coherent sheaves
on $T^\ast Gr(n,N)$, and their classes in the equivariant $K$-theory
ring $K_T(T^\ast Gr(n,N)) \iso K_T(Gr(n,N)) 
\iso K_{T_0}(Gr(n,N)) \tensor \ZZ[t^\pm]$. 
For $r \in {[N]\choose n}$, define the \defn{restriction map}
\begin{align*}
\vert_r:\ 
\ZZ[y_1^\pm,\ldots,y_n^\pm,z_1^\pm,\ldots,z_N^\pm] &\to K_{T_0} \\
f &\mapsto f\vert_r := f(z_{r_1},\ldots,z_{r_n},z_1,\ldots,z_N)
\end{align*}
in which case
$$ K_{T_0}(Gr(n,N)) \quad\iso \quad
\ZZ[y_1^\pm,\ldots,y_n^\pm,z_1^\pm,\ldots,z_N^\pm]^{S_n} \bigg/
\bigcap_{r \in {[N]\choose n}} \ker\left(\vert_r\right)  $$
where the $\mathcal S_n$ permutes the $y$ Laurent variables, which are 
the Chern roots of the tautological $n$-plane bundle on the Grassmannian.
Thus to describe a class, it suffices to give a Laurent polynomial 
and check its symmetry in the $y$s.

\begin{conj}\label{conj:ourconj1}
  Assume $N$ even. 
  There exists a $T$-equivariant coherent rank $1$ sheaf $\sh_r$ 
  supported on $CX^r$, defined in \S\ref{ssec:beyond},
  whose class in $K_T(T^\ast Gr(n,N))$ is represented by
\begin{equation}\label{eq:ourconj1}
[\sh_r]=m_r\ 
\prod_{i,j=1}^n a(y_i/y_j)^{-1}
\sum_{\substack{\text{CPLs with}\\\text{top-connectivity $r$}}}
\tau^{|\text{loops}|}
\prod_{i=1}^n \prod_{j=1}^N \begin{cases} a(y_i/z_j)&\tikz[baseline=0]{\plaq{a}}
\\[4mm] b(y_i/z_j)&\tikz[baseline=0]{\plaq{b}}
\end{cases}
\end{equation}
where the product is over rows $i$ and columns $j$ of the grid (the choice of $a$ or $b$ depending
on the type of plaquette at $(i,j)$),
the various plaquette weights are given by
\begin{align*}
a(x)&=t^{-1/2}x^{1/2}-t^{1/2} x^{-1/2}
\\
b(x)&=x^{-1/2}-x^{1/2}
\\
\tau&=t^{-1/2}+t^{1/2},
\end{align*}
and $m_r$ is a monomial (with unit coefficient) in the $z_i^{1/2}$ and $t^{1/2}$.
\end{conj}
\junk{the denominator is due to the different embedding than in the matrix schubert case -- moment map. here of course we never talk about matrix schubert varieties}

The r.h.s.\ of \eqref{eq:ourconj1} may not seem well-defined in
$K_T(T^\ast Gr(n,N))$ due to the presence of a denominator, but we
show in Remark~\ref{rmk:welldef} that it is. The symmetry in the $y$s
will also be proven, in Lemma~\ref{lem:sym0}. Lastly, note that $m_r$
can be absorbed into $[\sh_r]$, up to issues with square roots
of the $z_i$, by tensoring
$\sh_r$ with a trivial line bundle; one could even dispense with 
the ``$N$ even'' hypothesis by allowing square roots of the Chern roots
$y_i$ as well.

In the sections to come we won't be concerned with the sheaves $\sh_r$
so often as their sheaf cohomology groups, which we conjecture to all
vanish unless the map from $CX^r$ to its affinization doesn't drop dimension.
In \S \ref{sec:fiber} this geometric condition will be
shown equivalent to a Dyck path condition on $r$, and the vanishing
conjecture established for $H^0$ when $X^r$ is Gorenstein. 

\subsection{Polynomial solution of the level 1 quantum Knizhnik--Zamolodchikov equation}\label{sec:qKZ}
We now restrict to the case $N=2n$. This allows for the possibility to have full link patterns, that is, link patterns for which
every vertex is paired. We denote the set of full link patterns by $\LPN$;
from the point of view of Young diagrams, it is exactly the subset of Young diagrams which are inside
the ``staircase'' diagram $\stc:=\{2i,\ 1\le i\le n\}$. Its cardinality is the Catalan number $c_n=\frac{(2n)!}{n!(n+1)!}$.

\junk{$t^{1/2}=-q^{-1}$ compared to standard integrable conventions, except my papers with PDF have $q\to q^{-1}$}

We first recall the following
\begin{thm}[\cite{Kasa-wheel,Pas-RS,artic41}]\label{thm:wheel}
The space of polynomials in $N$ variables $z_1,\ldots,z_N$ of degree at most
$n(n-1)$
satisfying the {\em wheel condition}
\[
\left\{ P\in \mathbb C(t^{1/2})[z_1,\ldots,z_N]:\  
P(\ldots,z,\ldots,tz,\ldots,t^2z,\ldots)=0
\right\}
\]
is of dimension $c_n$ over $\mathbb C(t^{1/2})$. 
It has a basis indexed by link patterns
$(\Psi_r)_{r\in \LPN}$ given by the dual basis condition
\[
\Psi_r
\left( z_i=\begin{cases}
t^{-1/2}& i\in \bar s\\
t^{1/2}& i\in s
\end{cases}, i=1,\ldots,N\right)
= \delta_{r,s} \tau^{|r|},
\qquad r,s\in \LPN
\]
where $\tau=t^{1/2}+t^{-1/2}$.
\end{thm}
The $\Psi_r$ are homogeneous of degree $n(n-1)$.
They have remarkable properties, many of which follow from the 
(level 1) quantum Knizhnik--Zamolodchikov ($q$KZ) equation.
Given a vector $\Psi$ with entries $\Psi_r$ in a basis indexed by $\LPN$,
the level 1 $q$KZ system is:
\begin{align*}
\Psi(z_1,\ldots,z_{i+1},z_i,\ldots,z_N)&=\check R_i(z_i/z_{i+1}) \Psi(z_1,\ldots,z_i,z_{i+1},\ldots,z_N),
\qquad i=1,\ldots,N-1
\\
\Psi(z_2,\ldots,z_N,t^3 z_1) &= (-t^{1/2})^{3(n-1)} \rho \Psi(z_1,\ldots,z_N)
\end{align*}
where $\check R_i(z)=\frac{z-t+t^{1/2}(1-z)e_i}{1-t\,z}$,
$e_i$ is the Temperley--Lieb operator \tikz[scale=0.75,rotate=45,baseline=-3pt]{\plaq{b}} acting on link patterns by connecting $i$ and $i+1$
(with a weight of $\tau$ if they were already connected),\footnote{Note that if we similarly identify
the identity with \tikz[scale=0.75,rotate=45,baseline=-3pt]{\plaq{a}}, then $\check R_i$ is nothing but the combination of plaquettes occurring
in Conj.~\ref{conj:ourconj1}, up to normalization. See also \S\ref{sec:CPL}.}
and
$\rho$ is the rotation operator that shifts cyclically to the right link patterns.
We shall not make use of this system of equations in the present work and refer to \cite{artic34,hdr} for details.

We claim the following
\begin{conj}\label{conj:ourconj2}
The pushforward of $\sh_r$ to a point in localized 
$T$-equivariant $K$-theory is equal, up to normalization, to $\Psi_r$:
\[
\pi_*[\sh_r]=
\begin{cases}
0& r\not\subseteq \stc
\\
(1-t)^{n(n-1)} \tilde m_r
\ 
\prod_{1\le i<j\le N} 
(1-t\, z_i/z_j)^{-1}
\ 
\Psi_r
&r\subseteq\stc
\end{cases}
\]
\end{conj}
Here $\tilde m_r$ is another monomial in the $z_i^{1/2}$, related
to the previous one by $\tilde m_r=\prod_{i=1}^N z_i^{n/2+1-i}\, m_r$.

\junk{
\[
\tilde m_r=
\frac{t^{\frac{|r|}{2}} wt(\sh_r)}{\prod_{i=1}^n z_{\bar r_i}^{i-1} z_{r_i}^{i-1}}
\]
$wt(\sh_r)=\prod_i z_{r_i}^{-a}$  in the Gorenstein case
}

We will prove these conjectures in a forthcoming paper. 

In fact, it would not be difficult to show that
Conj~\ref{conj:ourconj1} implies Conj.~\ref{conj:ourconj2}
(the first equation of the $q$KZ system is naturally satisfied
by $\pi_* [\sh_r]$ up to normalization of the $R$-matrix,
and together with the initial condition of $\pi_* [\sh_\varnothing]$,
it determines them uniquely).

\junk{PZJ: or we can put it in this paper, it's not too hard.
  AK: I'm in favor of just saying ``$1\implies 2$, but we won't
  take the space to show this in this paper'' if we don't include it.
PZJ: ok rewrite as you see fit}

The factor $\prod_{1\le i<j\le N} (1-t\, z_i/z_j)$
can be naturally interpreted
in terms of the weights of the space of strict upper triangular matrices
$\text{Mat}_<(N)$; see \S \ref{sec:reform2} for details.

The limit $t\to 1$ corresponds on the integrable side to the ``rational'' 
limit from the quantum KZ equation to the difference KZ equation;
on the geometric side, to the limit from $K$-theory to cohomology, 
where some results similar to Conjecture~\ref{conj:ourconj2} are known
\cite{artic34,artic56,artic64}.

\subsection{Fully Packed Loops and the Razumov--Stroganov correspondence}
A \defn{Fully Packed Loop configuration}\/ (in short, \defn{FPL}) is an assignment of two possible states (empty, occupied) to the edges of a $n\times n$ square grid such that every vertex is traversed by exactly one path, and the external edges
are alternatingly occupied and empty (declaring that the topmost left external edge is occupied); e.g., for $n=4$,
\begin{center}
\begin{tikzpicture}[scale=0.75]
\draw[double distance=1.6pt] (0.01,0.01) grid (4.99,4.99);
\begin{scope}[blue,ultra thick,scale=5,line join=round]
\draw (0.2,0) -- (0.2,0.2) -- (0.6,0.2) -- (0.6,0.0);
\draw (1,0.2) -- (0.8,0.2) -- (0.8,0.4) -- (0.8,0.6) -- (1,0.6);
\draw (0,0.4) -- (0.6,0.4) -- (0.6,0.8) -- (0.8,0.8) -- (0.8,1);
\draw (0,0.8) -- (0.2,0.8) -- (0.2,0.6) -- (0.4,0.6) -- (0.4,1);
\end{scope}
\end{tikzpicture}
\end{center}
Numbering the occupied external edges clockwise from the leftmost top one, we can associate to an FPL the connectivity
of these external edges encoded as a (full) link pattern; in the present example, \linkpattern[inverted]{1/8,2/7,3/4,5/6}.
(Actually the choice of the starting point for the labelling of the external edges is irrelevant, at least for enumerative purposes, since
Wieland~\cite{Wieland} constructs a bijection of FPLs which rotates cyclically the connectivity of the external edges.)
Denote by $\FPL_r$ the set of FPLs with connectivity given by link pattern $r$.

One way to see the connection with what precedes is 
\begin{conj}\label{conj:geomRS}
$\Psi_r$ can be decomposed as a sum of products of the form
\[
\Psi_r=
\sum_{f\in \FPL_r}
\prod_{\alpha=1}^{n(n-1)}
\frac{
t^{-r_{f,\alpha}/2}z_{j_{f,\alpha}}-t^{r_{f,\alpha}/2} z_{i_{f,\alpha}}
}
{
t^{-1/2}-t^{1/2}
}
\]
where $r_{f,\alpha}\in \{1,2\}$.
\end{conj}
This conjecture 
is formulated somewhat implicitly in
\cite[\S 4]{DF-qKZ-TSSCPP}.

Specializing the $z_i$ to $1$ and $\tau$ to $1$ (i.e.,\ $t$ to a nontrivial cubic root of unity), 
leads to the following result, which,
by combining the $q$KZ approach \cite{artic31,artic34,hdr} to loop models 
and the proof \cite{CS-RS} by Cantini and Sportiello of the Razumov--Stroganov conjecture \cite{RS-conj}, is actually a theorem:
\begin{thm}
\[
\Psi_r(z_i=1,\ i=1,\ldots,N;\ \tau=1) = |\FPL_r|
\]
\end{thm}
Deriving Conj.~\ref{conj:geomRS} from the properties
of the sheaf $\sh_r$ would provide a geometric
justification for the Razumov--Stroganov correspondence.

\subsection{The rectangular case and plane partitions}
\label{sec:res}
In the present work, we study $\sh_r$ in the case that $r$ is a rectangular
Young diagram:
\begin{equation}\label{eq:defabc1}
r=
\begin{tikzpicture}[baseline=-1.1cm]
\node[tableau] {
&&&\hcell&\hcell&\hcell&\hcell
\\
&&
\\
\vcell
\\
\vcell
\\
\vcell
\\
};
\draw[latex-latex] (-0.2,-2) -- node[left=1mm] {$\ss N-n-c$} (-0.2,-0.52);
\draw[latex-latex] (-0.2,-0.48) -- node[left=1mm] {$\ss c$} (-0.2,0.5);
\draw[latex-latex] (0,0.7) -- node[above=1mm] {$\ss b$} (1.48,0.7);
\draw[latex-latex] (1.52,0.7) -- node[above=1mm] {$\ss n-b$} (3.5,0.7);
\end{tikzpicture}
\end{equation}
for two nonnegative integers $b,c$. We also define for future use
$a=N/2-(b+c)$. We then prove Conj.~\ref{conj:ourconj1} in that case.

We further specialize to $N=2n$, as in \S \ref{sec:qKZ}.
If $a<0$, we shall immediately conclude that
$\pi_*[\sigma_r]=0$, corresponding to the trivial case of Conj.~\ref{conj:ourconj2}. If $a\ge0$, we note that the link pattern corresponding to $r$ is of type ``$(a,b,c)$'', that is, of the form
\begin{equation}\label{eq:defabc2}
r=
\begin{tikzpicture}[baseline=-0.5cm]
\begin{scope}[yscale=-1]
\linkpattern[squareness=0.4,tikzstarted]{2/11}
\linkpattern[squareness=0.45,tikzstarted]{1/12}
\linkpattern[squareness=0.55,tikzstarted]{3/6,4/5,7/10,8/9}
\node at (3.9,1.85) {$\sss\vdots$};
\node at (4.2,1.9) {$a$};
\node at (2.7,0.4) {$\sss\vdots$};
\node at (2.9,0.5) {$b$};
\node at (5.1,0.4) {$\sss\vdots$};
\node at (5.3,0.5) {$c$};
\end{scope}
\end{tikzpicture}
\end{equation}
We shall then prove Conj.~\ref{conj:ourconj2} and \ref{conj:geomRS} in that case.
Note that such a type of link pattern was already considered in \cite{Zuber-conj,artic27,artic38} in the
context of FPL enumeration and the Razumov--Stroganov conjecture, but without any connection to
geometry. In particular, the following result will play a role in what follows:
\begin{thm}[\cite{artic27}]\label{thm:fpl-pp}
There is a bijection between $FPL_{(a,b,c)}$
(FPLs with connectivity $(a,b,c)$) and $PP(a,b,c)$, which is by definition the set of plane partitions of size $c\times b$ and maximal height $a$.
\end{thm}

For the purposes of this paper, it is best to use the following
definition of $PP(a,b,c)$: it is the set of $a$-tuplets of
$\subsarg{b}{b+c}$ for which the order $\le$ (pointwise
comparison of ordered elements of subsets) is a total order:
\[
PP(a,b,c)
:= \left\{(s_1,\ldots,s_a): s_1\le s_2\le\cdots\le s_a \right\}
\]
In turn, we can depict elements of $PP(a,b,c)$ in various equivalent
ways, as demonstrated in Fig.~\ref{fig:loz}: from top left to bottom right, 
plane partitions, lozenge tilings, Non-Intersecting Lattice
Paths (NILPs), dimer configurations.  These representations will be
discussed again in what follows when they are needed.

\def\a{2}\def\b{4}\def\c{3}
\def\pp{{2,1,1,0},{1,1,0,0},{1,0,0,0}}
\def\XY{(2,2), (4,1), (5,1), (6,1), (2,3), (3,3), (5,2), (6,2), (1,5), (3,4), (4,4), (6,3)}
\def\X{(3,2), (4,4), (5,6), (1,1), (1,2), (2,4)}
\def\Y{(3,4), (5,5), (6,5), (7,5), (1,1), (3,2), (5,3), (7,4)}
%
\let\mymatrixcontent\empty
  \foreach \row in \pp{
    \foreach \h in \row {%
        \expandafter\gappto\expandafter\mymatrixcontent\expandafter{\h \&}%
      }%
    \gappto\mymatrixcontent{\\}%
  }
\begin{figure}
\begin{tikzpicture}
\loz\XY\X\Y
\bigboxes
\node[tableau,ampersand replacement=\&,nodes in empty cells=false] at (-8,0) {\mymatrixcontent}; 
%
\begin{scope}[xshift=-6.8cm,yshift=-6.5cm]
\nilp\X\Y
\end{scope}
\begin{scope}[yshift=-6.5cm]
\dimer\XY\X\Y
\end{scope}
\node[outer sep=0.3cm] at (-3,-9.8) {$(\{1,3,5,7\},\{3,5,6,7\})\in PP(2,4,3)$};
\draw[latex-] (-5,0) -- node[above,align=center] {\scriptsize record heights\\[-1.5mm]\scriptsize viewed from top} (-2.8,0);
\draw[latex-latex] (0,-2.5) -- node[above,rotate=-90] {\scriptsize dual} (0,-4.2);
\draw[latex-] (-4.7,-6.2) -- node[above,align=center] {\scriptsize squash\\[-1.5mm]\scriptsize horizontal edges} (-2.7,-6.2);
\draw[-latex] (-7,-1.8) -- node[above,rotate=-90,align=center] {\scriptsize level curves,\\[-1.5mm]\scriptsize shifted} (-7,-4.2);
\draw[-latex] (-1.5,-1.5) to node[above,sloped] {\scriptsize follow 
\tikz[scale=0.2,rotate=-45]{\def\a{1}\def\b{1}\def\c{1}\lozX(0,0)} 
and 
\tikz[scale=0.2,rotate=-45]{\def\a{1}\def\b{1}\def\c{1}\lozY(0,0)} 
} (-5.2,-5.2);
\draw[latex-] (-3.25,-9.2) -- node[sloped,above] {\scriptsize record down steps} (-4.75,-7.7);
\end{tikzpicture}
\caption{Various depictions of a particular $a=2,b=4,c=3$ configuration.}\label{fig:loz}
\end{figure}

\junk{These are in an obvious bijection with semistandard rectangular 
Young tableaux $\tau$ (of shape $b\times c$) with entries $\leq a+b$; 
given such a $\tau$, let $s_i = \{(j,k)\ :\ \tau_{(j,k)} \leq i+j$.
Note that these are {\em not}\/ the same rectangular tableaux that came
up in the last subsection; rather both are in bijection with plane
partitions inside the $a\times b\times c$ box.}

\subsection{Plan of the paper}

In \S 2 we study CPLs using integrability and in particular
the Yang-Baxter equation, with which we show that the conjectured
formula \ref{conj:ourconj1} has the right symmetry and 
vanishing properties.
In \S 3 we give detail on conormal varieties to Grassmannian Schubert
varieties, and define the sheaves referenced in conjectures 
\ref{conj:ourconj1} and \ref{conj:ourconj2}. 
Our definition of these sheaves in \S \ref{ssec:beyond}, for arbitrary $G/P$, 
is much more general than is required for the rest of the paper.
In \S 4 we specialize to case of a smooth Schubert subvariety
of a Grassmannian $Gr(n,2n)$, determine the conjectured sheaf in this case,
and compute the degree $2^{bc} |PP(a,b,c)|$ of the conormal variety.
In \S 5 we define the family degenerating the conormal variety
and the sheaf it bears, whose special fiber we determine in \S 6.
It turns out to have one component for each element of $PP(a,b,c)$,
a toric complete intersection of $bc$ many quadrics. 
With this in hand, we prove in \S 7 our conjectures \ref{conj:ourconj1} 
and \ref{conj:ourconj2} in the Grassmann-Grassmann case.

\section{Integrability of the CPL model}\label{sec:CPL}
\subsection{Basic properties}
Introduce the following notation, borrowed from integrable models:
\[
\begin{tikzpicture}[baseline=-3pt]
\draw[bgplaq] (-0.375,-0.375) rectangle (0.375,0.375);
\draw[invarrow=0.4,invarrow=0.9] (-0.5,0) -- (0.5,0) node[right] {$\ss y$} (0,-0.5) -- (0,0.5) node[above] {$\ss z$};
\end{tikzpicture}
=
a(y/z)\,
\begin{tikzpicture}[baseline=-3pt]
\plaq{a}
\end{tikzpicture}
+
b(y/z)\,
\begin{tikzpicture}[baseline=-3pt]
\plaq{b}
\end{tikzpicture}
\]
where recall from Conj.~\ref{conj:ourconj1} that 
$a(x):=t^{-1/2}x^{1/2}-t^{1/2} x^{-1/2}$ and
$b(x):=x^{-1/2}-x^{1/2}$.
The dotted lines are purely cosmetic and will be frequently omitted.

\begin{prop}\label{prop:ybe}
The following identities hold: the Yang--Baxter equation
\begin{equation}\label{eq:ybe}
\begin{tikzpicture}[baseline=-3pt,y=2cm]
\draw[arrow=0.1,arrow=0.4,arrow=0.7,rounded corners] (-0.5,0.5) node[above] {$\ss z_1$} -- (0.75,0) -- (1.5,-0.5) (0.5,0.5) node[above] {$\ss z_2$} -- (0.25,0) -- (0.5,-0.5) (1.5,0.5) node[above] {$\ss z_3$} -- (0.75,0) -- (-0.5,-0.5);
\end{tikzpicture}
=
\begin{tikzpicture}[baseline=-3pt,y=2cm]
\draw[arrow=0.1,arrow=0.4,arrow=0.7,rounded corners] (-0.5,0.5) node[above] {$\ss z_1$} -- (0.25,0) -- (1.5,-0.5) (0.5,0.5) node[above] {$\ss z_2$} -- (0.75,0) -- (0.5,-0.5) (1.5,0.5) node[above] {$\ss z_3$} -- (0.25,0) -- (-0.5,-0.5);
\end{tikzpicture}
\end{equation}
the unitarity equation
\begin{equation}\label{eq:unit}
\begin{tikzpicture}[baseline=-3pt]
\draw[arrow=0.07,arrow=0.57,rounded corners] (-0.5,1) node[above] {$\ss z_1$} -- (0.5,0) -- (-0.5,-1) (0.5,1) node[above] {$\ss z_2$} -- (-0.5,0) -- (0.5,-1);
\end{tikzpicture}
=
a(z_1/z_2)a(z_2/z_1)
\begin{tikzpicture}[baseline=-3pt]
\draw[arrow=0.1,arrow=0.6,rounded corners] (-0.5,1) node[above] {$\ss z_1$} -- (-0.5,-1) (0.5,1) node[above] {$\ss z_2$} -- (0.5,-1);
\end{tikzpicture}
\end{equation}
and the special value
\begin{equation}\label{eq:1}
\begin{tikzpicture}[baseline=-3pt]
\draw[arrow=0.1,arrow=0.6] (-0.5,0.5) node[above] {$\ss z$} -- (0.5,-0.5) (0.5,0.5) node[above] {$\ss z$} -- (-0.5,-0.5);
\end{tikzpicture}
=
(t^{1/2}-t^{-1/2})
\begin{tikzpicture}[baseline=-3pt]
\draw[arrow=0.1,arrow=0.6,rounded corners] (-0.5,0.5) node[above] {$\ss z$} -- (-0.1,0) -- (-0.5,-0.5) (0.5,0.5) node[above] {$\ss z$} -- (0.1,0) -- (0.5,-0.5);
\end{tikzpicture}
\end{equation}
\end{prop}
All these identities should be understood as an equality of the coefficients on both sides of all the diagrams with
a given connectivity of the external points, with the rule that each closed loop incurs a weight of $\tau$.
The proof of the proposition is a standard calculation.

Let us now introduce the CPL partition function with given top-connectivity $r\in \subs$:
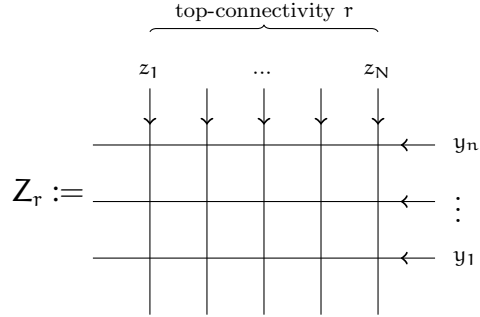
\[
Z_r := 
\begin{tikzpicture}[scale=0.75,baseline=1.5cm]
\draw[decorate,decoration=brace] (1,5) -- node[above] {$\ss\text{top-connectivity } r$} (5,5);
\foreach\y/\txt in {1/y_1,2/\vdots,3/y_n}
\draw[invarrow=0.92] (0,\y)  -- (6,\y)  node[right=1mm] {$\ss \txt$};
\foreach\x/\txt in {1/z_1,2/,3/\cdots,4/,5/z_N}
\draw[invarrow=0.87] (\x,0)  -- (\x,4) node[above] {$\ss \txt$};
\end{tikzpicture}
\]
where we recall the connectivity rules: (b,l), (b,r), (b,t), (l,t), (t,t), and the connectivity of the top midpoints
is given by the link pattern $r$.

Denote by $\tau_i$ the elementary transposition $i\leftrightarrow i+1$.
By abuse of notation, also let $\tau_i$ denote the operator acting on polynomials of the variables $z_i$, $i=1,\ldots, N$ that permutes $z_i$ and $z_{i+1}$. 
\begin{lem}\label{lem:sym0}
We have the two symmetry properties:
\begin{itemize}
\item $Z_r$ is a symmetric function of the $y_i$.
\item Assume $i$ and $i+1$ are not connected in $r$. 
Then $\tau_i Z_r=Z_r$.
\end{itemize}
\end{lem}
\begin{proof}
This is a standard Yang--Baxter-based proof. Repeated application of equation \eqref{eq:ybe} leads to
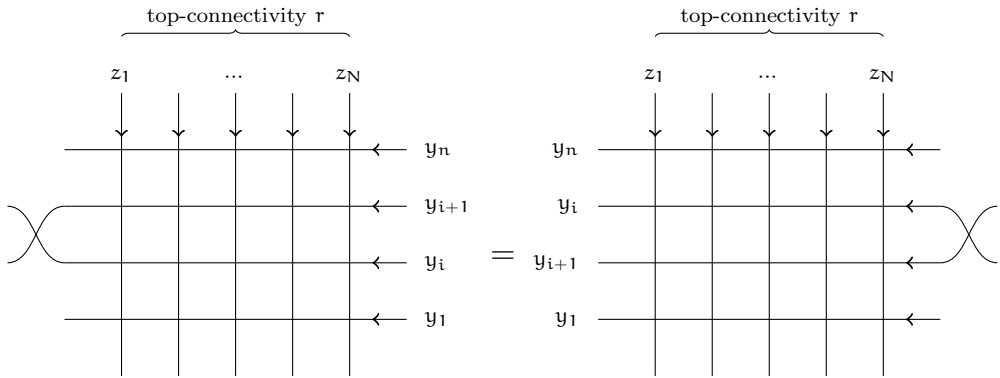
\[
\begin{tikzpicture}[scale=0.75,baseline=1.5cm]
\draw[decorate,decoration=brace] (1,6) -- node[above] {$\ss\text{top-connectivity } r$} (5,6);
\foreach\y/\txt in {1/y_1,2/y_i,3/y_{i+1},4/y_n}
\draw[invarrow=0.92] (0,\y)  -- (6,\y)  node[right=1mm] {$\ss \txt$};
\draw (0,2) .. controls (-0.5,2) and (-0.5,3) ..  (-1,3);
\draw (0,3) .. controls (-0.5,3) and (-0.5,2) ..  (-1,2);
\foreach\x/\txt in {1/z_1,2/,3/\cdots,4/,5/z_N}
\draw[invarrow=0.87] (\x,0)  -- (\x,5) node[above] {$\ss \txt$};
\end{tikzpicture}
=
\begin{tikzpicture}[scale=0.75,baseline=1.5cm]
\draw[decorate,decoration=brace] (1,6) -- node[above] {$\ss\text{top-connectivity } r$} (5,6);
\foreach\y/\txt in {1/y_1,2/y_{i+1},3/y_i,4/y_n}
\draw[invarrow=0.92] (0,\y) node[left=1mm] {$\ss \txt$} -- (6,\y);
\draw (6,2) .. controls (6.5,2) and (6.5,3) ..  (7,3);
\draw (6,3) .. controls (6.5,3) and (6.5,2) ..  (7,2);
\foreach\x/\txt in {1/z_1,2/,3/\cdots,4/,5/z_N}
\draw[invarrow=0.87] (\x,0)  -- (\x,5) node[above] {$\ss \txt$};
\end{tikzpicture}
\]
Imposing top-connectivity $r$ means in particular that
left vertices are not allowed to connect between themselves, and the same for right vertices.
This implies that the extra plaquette on either side of the equation above (represented by a crossing) must be of the form
\tikz[scale=0.75,rotate=45,baseline=-3pt]{\plaq{b}}, so it does not affect connectivity and can be removed, and its
weight $a(y_i/y_{i+1})$ compensates as well between l.h.s.\ and r.h.s. Once these plaquettes are removed, we get 
back to $Z_r$ itself, except on the left hand side $y_i$ and $y_{i+1}$ are switched.
This implies symmetry
of $Z_r$ by exchange of $y_i$ and $y_{i+1}$ for all $i=1,\ldots,n-1$, and therefore by any permutation of the $y$'s.

The same argument, with the extra crossing inserted at the bottom (necessarily of the form \tikz[scale=0.75,rotate=45,baseline=-3pt]{\plaq{a}})
and then moved upwards, leads to the formula
\begin{equation}\label{eq:divdiff}
a(z_{i+1}/z_i) Z_r = a(z_{i+1}/z_i)\tau_i Z_r+b(z_{i+1}/z_i)\sum_{s:\ e_i s=r} \tau^{\delta_{r,s}} \tau_i Z_s
\end{equation}
where the summation is over link patterns $s$ which can be obtained from $r$ by concatenating them with
$e_i=\tikz[scale=0.75,rotate=45,baseline=-3pt]{\plaq{b}}$.
This shows in particular that if $i$ and $i+1$ are not connected in $r$, the summation is empty and $\tau_i Z_r=Z_r$.
\end{proof}

\subsection{Reformulation of conjecture~\ref{conj:ourconj1}}
\label{sec:reform1}
With these notations, Conj.~\ref{conj:ourconj1} is written
\[
[\sh_r]=m_r \prod_{i,j=1}^n a(y_i/y_j)^{-1} Z_r
\]

We now propose an alternate formulation of this conjecture. A class in $K_T(T^\ast Gr(n,N))$ is entirely
determined by its restrictions to $T$-fixed points, which are
indexed by $\subs$. Recall from \S \ref{sec:gen} that the restriction map
to the coordinate subspace $\mathbb C^s$, $s\in\subs$,
denoted by $|_s$, amounts to
the specialization $y_i= z_{s_i}$, $i=1,\ldots,n$. We are thus naturally led to the computation of $Z_r$
after such a substitution, which we can perform with the help of Prop.~\ref{prop:ybe}; here shown on an example:
\begin{align}\label{eq:specZ}
Z_r\vert_s
&=
\begin{tikzpicture}[scale=0.75,baseline=1.5cm]
\draw[decorate,decoration=brace] (1,5) -- node[above] {$\ss\text{top-connectivity } r$} (5,5);
\foreach\y/\x in {1/4,2/3,3/1}
\draw[invarrow=0.92] (0,\y)  -- (6,\y)  node[right=1mm] {$\ss z_\x$};
\foreach\x in {1,...,5}
\draw[invarrow=0.87] (\x,0)  -- (\x,4) node[above] {$\ss z_\x$};
\begin{scope}[overlay]
\draw[-latex] (6,0.3) -- (6,-1) node[below=2mm] {$\ss s=\{1,3,4\}$};
\end{scope}
\end{tikzpicture}
=a(1)^n
\ 
\begin{tikzpicture}[scale=0.75,baseline=1.5cm]
\draw[decorate,decoration=brace] (1,5) -- node[above] {$\ss\text{top-connectivity } r$} (5,5);
\draw[rounded corners=8pt,invarrow=0.8] (0,3)  -- (1,3) -- (1,4) node[above] {$\ss z_1$};
\draw[rounded corners=8pt,invarrow=0.9] (0,2)  -- (3,2) -- (3,4) node[above] {$\ss z_3$};
\draw[rounded corners=8pt,invarrow=0.93] (0,1)  -- (4,1) -- (4,4) node[above] {$\ss z_4$};
\draw[rounded corners=8pt,invarrow=0.94] (1,0)  -- (1,3) -- (6,3);
\draw[rounded corners=8pt,invarrow=0.9] (2,0)  -- (2,4) node[above] {$\ss z_2$};
\draw[rounded corners=8pt,invarrow=0.9] (3,0)  -- (3,2) -- (6,2);
\draw[rounded corners=8pt,invarrow=0.83] (4,0)  -- (4,1) -- (6,1);
\draw[rounded corners=8pt,invarrow=0.9] (5,0)  -- (5,4) node[above] {$\ss z_5$};
\end{tikzpicture}
\\[3mm]\notag
&=
\prod_{i\in s, j\in s} a(z_i/z_j)
\prod_{\substack{i\in s, j\in \bar s\\i<j}} a(z_i/z_j)
\ 
\begin{tikzpicture}[scale=0.75,baseline=1.5cm]
\draw[decorate,decoration=brace] (1,5) -- node[above] {$\ss\text{top-connectivity } r$} (5,5);
\draw[rounded corners=8pt,invarrow=0.8] (0,3)  -- (1,3) -- (1,4) node[above] {$\ss z_1$};
\draw[rounded corners=8pt,invarrow=0.9] (0,2)  -- (3,2) -- (3,4) node[above] {$\ss z_3$};
\draw[rounded corners=8pt,invarrow=0.93] (0,1)  -- (4,1) -- (4,4) node[above] {$\ss z_4$};
\draw[rounded corners=8pt,invarrow=0.9] (2,0)  -- (2,4) node[above] {$\ss z_2$};
\draw[rounded corners=8pt,invarrow=0.88] (5,0)  -- (5,4) node[above] {$\ss z_5$};
\end{tikzpicture}
\\\notag
&=
\prod_{i\in s, j\in s} a(z_i/z_j)
\prod_{\substack{i\in s, j\in \bar s\\i<j}} a(z_i/z_j)
\ 
\begin{tikzpicture}[baseline=1cm]
\draw[decorate,decoration=brace] (1,2.75) -- node[above] {$\ss\text{top-connectivity } r$} (5,2.75);
\draw[rounded corners=8pt,invarrow=0.85] (1,0)  -- (1,2) node[above] {$\ss z_1$};
\draw[rounded corners=8pt,invarrow=0.85] (2,0)  -- (3,2) node[above] {$\ss z_3$};
\draw[rounded corners=8pt,invarrow=0.85] (3,0)  -- (4,2) node[above] {$\ss z_4$};
\draw[rounded corners=8pt,invarrow=0.85] (4,0)  -- (2,2) node[above] {$\ss z_2$};
\draw[rounded corners=8pt,invarrow=0.85] (5,0)  -- (5,2) node[above] {$\ss z_5$};
\end{tikzpicture}
\end{align}
where between the first and second line we have pulled out all the south-east lines using Prop.~\ref{prop:ybe}, producing the first
factor, and then
removed them altogether because the constraint that bottom vertices cannot connect between themselves forces plaquettes
of type $a$, producing the second factor.
After rearranging the lines to produce the final picture, the connectivity 
must be understood as follows: the top vertices have connectivity
given by the link pattern $r$, whereas the first $n$ bottom vertices
cannot connect between themselves, and similarly for the last $N-n$.

There are two ways to understand the resulting picture. On the one hand, reintroducing the dotted lines,
we recognize the complement of the Young diagram of $s$ (cf \eqref{eq:bij}):
\[
s=\{1,3,4\}\in \subsarg{3}{5}
\quad
\rightarrow
\quad
\tableau{&&\\&\dottedcell&\dottedcell\\}
\quad
\rightarrow
\quad
\begin{tikzpicture}[baseline=1cm,yscale=0.75]
\draw[bgplaq] (3.5,0) -- ++(-1.333,1.333) -- ++(0.333,0.667) -- ++(1.333,-1.333) -- cycle;
\draw[rounded corners=8pt,invarrow=0.85] (1,0)  -- (1,2) node[above] {$\ss z_1$};
\draw[rounded corners=8pt,invarrow=0.85] (2,0)  -- (3,2) node[above] {$\ss z_3$};
\draw[rounded corners=8pt,invarrow=0.85] (3,0)  -- (4,2) node[above] {$\ss z_4$};
\draw[rounded corners=8pt,invarrow=0.85] (4,0)  -- (2,2) node[above] {$\ss z_2$};
\draw[rounded corners=8pt,invarrow=0.85] (5,0)  -- (5,2) node[above] {$\ss z_5$};
\draw[dotted] (3.5,0) ++(-0.667,0.667) -- ++(0.333,0.667);
\draw[dotted] (3.5,0) -- ++(-1.333,1.333) -- (0.5,1.333);
\draw[dotted] (3.5,0) ++(0.333,0.667) -- (5.5,0.667);
\end{tikzpicture}
\]

On the other hand, we also recognize this diagram to be the graphical
representation of a (or any) reduced word of the Grassmannian
permutation of $\{1,\ldots,N\}$
which sends $\{1,\ldots,n\}$ (at the bottom) to $s$ (at the top).\footnote{%
In fact, the Yang--Baxter equation \eqref{eq:ybe}
would ensure independence of the choice of reduced word, 
but is not necessary in the Grassmannian case, 
nor more generally in the $321$-avoiding case.
Moreover, because of the unitarity equation \eqref{eq:unit},
the reducedness of the word is irrelevant. }

Let us therefore define $Z_{r,s}$ to be the prefactor
$\prod_{\substack{i\in s, j\in\bar s\\i<j}} a(z_i/z_j)$ times the
CPL partition function for the diagram defined in either of the two ways above,
with top-connectivity given by $r$,
in which we recall that each crossing represents as above one of the
two plaquettes with their weights and that closed loops have a weight of $\tau$.

We can thus reformulate Conj.~\ref{conj:ourconj1} as follows:
{\renewcommand\theconj{1'}\begin{conj}\label{conj:ourconj1'}
The restriction of the $K_T$-class of the sheaf $\sh_r$ 
to the fixed point $s$ satisfies
\begin{equation}\label{eq:ourconj1'}
[\sh_r\vert_s] = m_r
\,
Z_{r,s}
\end{equation}
\end{conj}}

This is the same {\em sort}\/ of restriction-to-fixed-points formula as
in \cite{CSu} (or the AJS/Billey and Graham/Willems formulae for 
restricting Schubert not conormal Schubert classes), 
except that \cite{CSu} is working with Maulik--Okounkov's stable basis, 
not our basis, and (less importantly) that formula in \cite{CSu}
is in $H^*_T$ not $K_T$.

\begin{rmk}\label{rmk:welldef}
For $i,j\in \{1,\ldots,N\}$, and $s=\{\ldots,i,\ldots\}$, $s'=\{\ldots,j,\ldots\}$ related by the transposition $i\leftrightarrow j$,
because of the very definition of the $Z_{r,s}$
as a specialization $(y_1,\ldots,y_n)=(z_{s_1},\ldots,z_{s_n})$ of $Z_r$, the
following congruence holds: 
\[
z_i-z_j\ \big\vert\ \left(Z_r\vert_s - Z_r\vert_{s'}\right)
\]
This is exactly the $K$-theoretic GKM criterion for being in the image of 
the restriction map (see e.g. \cite[A.4]{KRosu}, \cite[corollary 5.11]{VV}).
Contrary to $Z_r$ itself, the $Z_r\vert_s$ are Laurent polynomials, 
and are therefore the point restrictions of a uniquely defined
element of $K_T(Gr(n,N)) \iso K_T(T^* Gr(n,N))$.
\junk{
  and therefore $Z_r$, as an element of $K_T(Gr(n,N))[a(y_i/y_j)^{-1}]$,
  is actually equal to an element of $K_T(Gr(n,N))$ itself.
  \rem{does that make sense?}}
\end{rmk}

\subsection{Two lemmas}
We need two more lemmas.
\begin{lem}\label{lem:zero}$Z_{r,s}$ satisfies the triangularity property:
$Z_{r,s}=0$ unless $s\subseteq r$, and
\begin{equation}\label{eq:Zll}
Z_{r,r}=\prod_{\substack{i\in r, j\in\bar r\\i<j}} a(z_i/z_j)\prod_{\substack{i\in r, j\in\bar r\\i>j}} b(z_i/z_j).
\end{equation}
\end{lem}
\begin{proof}
Induction on $s$. If $s=\{1,\ldots,n\}$, the diagram of $Z_{r,s}$ contains
no plaquette, and the resulting top-connectivity is the completely unpaired link
pattern, i.e.,\ $r=s$.

Now, pick $s\in \subs$, and assume that the property is true for any 
$s'$, $s\subsetneq s'$.
Denote by $s^C$ the complement of the Young diagram of $s$
in the rectangle $(N-n)\times n$.

Pick any protruding box in $s^C$. Choose $s'$ to be the Young diagram obtained
from $s$ by adding that box. Equivalently, choose an $i$ such that
$i\not\in s$, $i+1\in s$, and define the subset $s'=\tau_i s$.
This means that the $Z_{r,s}$ can be obtained
from the $Z_{r',s'}$ by adding an extra crossing, e.g.,
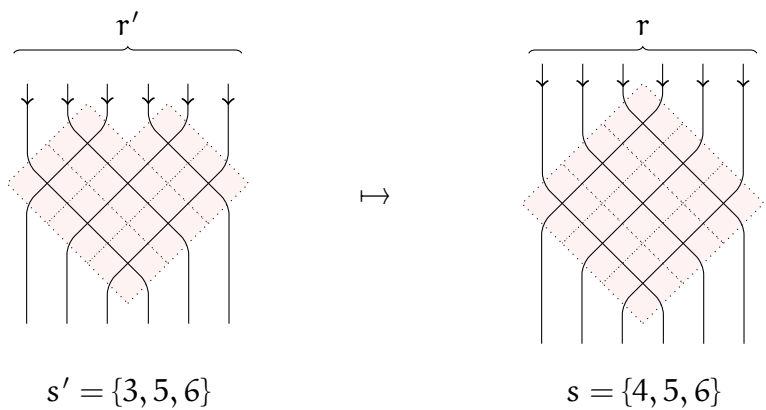
\[
\begin{tikzpicture}[baseline=0]
\node[rotate=45,inner sep=0pt] { \tikz{ 
\node[loop]
{
\plaq{}&\plaq{}&\\
\plaq{}&\plaq{}&\plaq{}\\
\plaq{}&\plaq{}&\plaq{}\\
};
\begin{scope}[x=\loopcellsize,y=\loopcellsize]
\draw[invarrow=0.95,rounded corners] (loop-1-1.west) ++(-1.5,-1.5) -- ++(1.5,1.5) -- ++(2,0) -- ++(0.5,0.5);
\draw[invarrow=0.95,rounded corners] (loop-2-1.west) ++(-1,-1) -- ++(1,1) -- ++(3,0) -- ++(0.5,0.5);
\draw[invarrow=0.95,rounded corners] (loop-3-1.west) ++(-0.5,-0.5) -- ++(0.5,0.5) -- ++(3,0) -- ++(1,1);
\draw[invarrow=0.95,rounded corners] (loop-3-1.south) ++(-0.5,-0.5) -- ++(0.5,0.5) -- ++(0,3) -- ++(1,1);
\draw[invarrow=0.95,rounded corners] (loop-3-2.south) ++(-1,-1) -- ++(1,1) -- ++(0,3) -- ++(0.5,0.5);
\draw[invarrow=0.95,rounded corners] (loop-3-3.south) ++(-1.5,-1.5) -- ++(1.5,1.5) -- ++(0,2) -- ++(0.5,0.5);
\end{scope}
}
};
\node at (0,-2.5) {$s'=\{3,5,6\}$};
\draw[decoration=brace,decorate] (-1.5,2) -- node[above=1mm] {$r'$} (1.5,2);
\end{tikzpicture}
\mapsto
\begin{tikzpicture}[baseline=0]
\node[rotate=45,inner sep=0pt] { \tikz{ 
\node[loop]
{
\plaq{}&\plaq{}&\plaq{}\\
\plaq{}&\plaq{}&\plaq{}\\
\plaq{}&\plaq{}&\plaq{}\\
};
\begin{scope}[x=\loopcellsize,y=\loopcellsize]
\draw[invarrow=0.95,rounded corners] (loop-1-1.west) ++(-1.5,-1.5) -- ++(1.5,1.5) -- ++(3,0) -- ++(0.5,0.5);
\draw[invarrow=0.95=0.95,rounded corners] (loop-2-1.west) ++(-1,-1) -- ++(1,1) -- ++(3,0) -- ++(1,1);
\draw[invarrow=0.95,rounded corners] (loop-3-1.west) ++(-0.5,-0.5) -- ++(0.5,0.5) -- ++(3,0) -- ++(1.5,1.5);
\draw[invarrow=0.95,rounded corners] (loop-3-1.south) ++(-0.5,-0.5) -- ++(0.5,0.5) -- ++(0,3) -- ++(1.5,1.5);
\draw[invarrow=0.95,rounded corners] (loop-3-2.south) ++(-1,-1) -- ++(1,1) -- ++(0,3) -- ++(1,1);
\draw[invarrow=0.95=0.6,rounded corners] (loop-3-3.south) ++(-1.5,-1.5) -- ++(1.5,1.5) -- ++(0,3) -- ++(0.5,0.5);
\end{scope}
}};
\node at (0,-2.5) {$s=\{4,5,6\}$}; 
\draw[decoration=brace,decorate] (-1.5,2) -- node[above=1mm] {$r$} (1.5,2);
\end{tikzpicture}
\]
By the induction hypothesis, the only nonzero $Z_{r',s'}$ are the ones
for which $s'\subseteq r'$. The extra crossing can take two forms:
\begin{enumerate}
\item \tikz[scale=0.75,rotate=45,baseline=-3pt]{\plaq{a}}, which does not
change the connectivity; these contribute
to $Z_{r',s}$, which satisfies the upper triangularity of the lemma
as $s\subsetneq s'\subseteq r'$.
\junk{AK: spell out more please. PZJ: how? $s\subseteq r'$ so having diagrams contributing to $Z_{r',s}$ does not contradict }
\item \tikz[scale=0.75,rotate=45,baseline=-3pt]{\plaq{b}}, that is,
$r=e_ir'$, where as before
we denote $e_ir'$ the link pattern obtained from $r'$ 
by pasting this extra plaquette.
There are four possibilities, depending on the local configuration of
$r'$:
\setlength{\loopcellsize}{0.6cm}
\item[(2a)] $i\in r'$, $i+1\not\in r'$: then $e_ir'=\tau_ir'$, i.e.,\ the Young
diagram of $(e_ir')^C$ is obtained from $r'{}^C$ by adding an extra box:
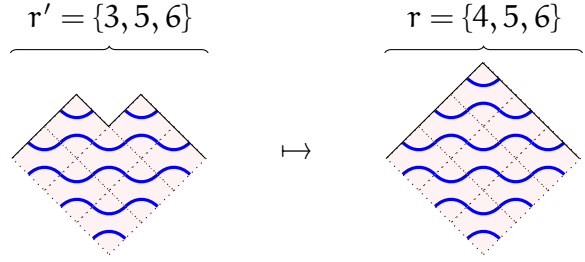
\[
\begin{tikzpicture}[baseline=0]
\node[rotate=45,inner sep=0pt] { \tikz{ 
\node[loop]
{
\plaq{b}&\plaq{b}&\\
\plaq{b}&\plaq{b}&\plaq{b}\\
\plaq{b}&\plaq{b}&\plaq{b}\\
};
\begin{scope}[x=\loopcellsize,y=\loopcellsize]
\draw (loop-1-1.north west) -- ++(2,0) -- ++(0,-1) -- ++(1,0) -- ++(0,-2);
\end{scope}
}
};
\draw[decoration=brace,decorate] (-1.3,1.4) -- node[above=1mm] {$r'=\{3,5,6\}$} (1.3,1.4);
\end{tikzpicture}
\qquad\mapsto\qquad
\begin{tikzpicture}[baseline=0]
\node[rotate=45,inner sep=0pt] { \tikz{ 
\node[loop]
{
\plaq{b}&\plaq{b}&\plaq{b}\\
\plaq{b}&\plaq{b}&\plaq{b}\\
\plaq{b}&\plaq{b}&\plaq{b}\\
};
\begin{scope}[x=\loopcellsize,y=\loopcellsize]
\draw (loop-1-1.north west) -- ++(3,0) -- ++(0,-3);
\end{scope}
}};
\draw[decoration=brace,decorate] (-1.3,1.4) -- node[above=1mm] {$r=\{4,5,6\}$} (1.3,1.4);
\end{tikzpicture}
\]
Since the same procedure goes from $s'{}^C$ to $s^C$, we still have
$(e_ir')^C\subseteq s^C$, or $s\subseteq e_ir'=r$.
\item[(2b)] $i,i+1\in r'$:
\[
\begin{tikzpicture}[baseline=0]
\node[rotate=45,inner sep=0pt] { \tikz{ 
\node[loop]
{
\plaq{b}&\plaq{b}&\\
\plaq{b}&\plaq{b}&\plaq{a}\\
\plaq{b}&\plaq{b}&\plaq{b}\\
};
\begin{scope}[x=\loopcellsize,y=\loopcellsize]
\draw (loop-1-1.north west) -- ++(2,0) -- ++(0,-1) -- ++(1,0) -- ++(0,-2);
\end{scope}
}
};
\end{tikzpicture}
=
\begin{tikzpicture}[baseline=0]
\node[rotate=45,inner sep=0pt] { \tikz{ 
\node[loop]
{
\plaq{b}&\plaq{b}&\\
\plaq{b}&\plaq{b}&\\
\plaq{b}&\plaq{b}&\plaq{b}\\
};
\begin{scope}[x=\loopcellsize,y=\loopcellsize]
\draw (loop-1-1.north west) -- ++(2,0) -- ++(0,-2) -- ++(1,0) -- ++(0,-1);
\end{scope}
}
};
\draw[decoration=brace,decorate] (-1.3,1) -- node[above=1mm] {$r'=\{3,4,6\}$} (1.3,1);
\end{tikzpicture}
\qquad\mapsto\qquad
\begin{tikzpicture}[baseline=0]
\node[rotate=45,inner sep=0pt] { \tikz{ 
\node[loop]
{
\plaq{b}&\plaq{b}&\plaq{b}\\
\plaq{b}&\plaq{b}&\plaq{a}\\
\plaq{b}&\plaq{b}&\plaq{b}\\
};
\begin{scope}[x=\loopcellsize,y=\loopcellsize]
\draw (loop-1-1.north west) -- ++(3,0) -- ++(0,-3);
\draw (loop-1-1.north west) ++(1,0) -- ++(0,-2) -- ++(1,0) -- ++(0,1) -- ++(1,0);
\end{scope}
}};
\end{tikzpicture}
=
\begin{tikzpicture}[baseline=0]
\node[rotate=45,inner sep=0pt] { \tikz{ 
\node[loop]
{
\plaq{b}&&\\
\plaq{b}&\plaq{b}&\\
\plaq{b}&\plaq{b}&\plaq{b}\\
};
\begin{scope}[x=\loopcellsize,y=\loopcellsize]
\draw (loop-1-1.north west) -- ++(1,0) -- ++(0,-1) -- ++(1,0) -- ++(0,-1) -- ++(1,0) -- ++(0,-1);
\end{scope}
}};
\draw[decoration=brace,decorate] (-1.3,1) -- node[above=1mm] {$r=\{2,4,6\}$} (1.3,1);
\end{tikzpicture}
\]
Effectively using the Temperley--Lieb relation $e_ie_{i-1}e_i=e_i$, we see
that the extra $e_i$ destroys the plaquette south-west of it,
resulting in $e_ir'=\tau_{i-1}r'$, so that
$s\subsetneq s'\subseteq r'\subsetneq e_ir'=r$.
\item[(2c)] $i,i+1\not\in r'$ is treated similarly as the previous case:
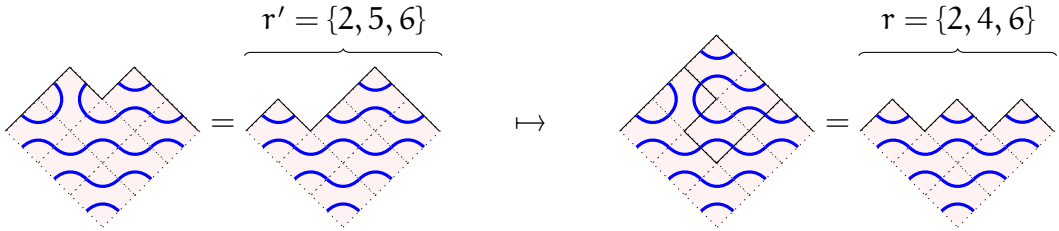
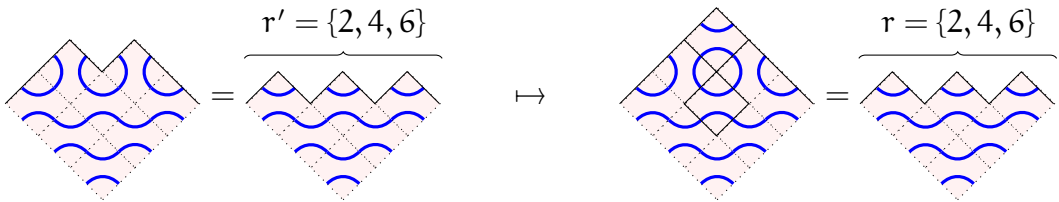
\[
\begin{tikzpicture}[baseline=0]
\node[rotate=45,inner sep=0pt] { \tikz{ 
\node[loop]
{
\plaq{b}&\plaq{a}&\\
\plaq{b}&\plaq{b}&\plaq{b}\\
\plaq{b}&\plaq{b}&\plaq{b}\\
};
\begin{scope}[x=\loopcellsize,y=\loopcellsize]
\draw (loop-1-1.north west) -- ++(2,0) -- ++(0,-1) -- ++(1,0) -- ++(0,-2);
\end{scope}
}
};
\end{tikzpicture}
=
\begin{tikzpicture}[baseline=0]
\node[rotate=45,inner sep=0pt] { \tikz{ 
\node[loop]
{
\plaq{b}&&\\
\plaq{b}&\plaq{b}&\plaq{b}\\
\plaq{b}&\plaq{b}&\plaq{b}\\
};
\begin{scope}[x=\loopcellsize,y=\loopcellsize]
\draw (loop-1-1.north west) -- ++(1,0) -- ++(0,-1) -- ++(2,0) -- ++(0,-2);
\end{scope}
}
};
\draw[decoration=brace,decorate] (-1.3,1) -- node[above=1mm] {$r'=\{2,5,6\}$} (1.3,1);
\end{tikzpicture}
\qquad\mapsto\qquad
\begin{tikzpicture}[baseline=0]
\node[rotate=45,inner sep=0pt] { \tikz{ 
\node[loop]
{
\plaq{b}&\plaq{a}&\plaq{b}\\
\plaq{b}&\plaq{b}&\plaq{b}\\
\plaq{b}&\plaq{b}&\plaq{b}\\
};
\begin{scope}[x=\loopcellsize,y=\loopcellsize]
\draw (loop-1-1.north west) -- ++(3,0) -- ++(0,-3);
\draw (loop-1-1.north west) ++(2,0) -- ++(0,-1) -- ++(-1,0) -- ++(0,-1) -- ++(2,0);
\end{scope}
}};
\end{tikzpicture}
=
\begin{tikzpicture}[baseline=0]
\node[rotate=45,inner sep=0pt] { \tikz{ 
\node[loop]
{
\plaq{b}&&\\
\plaq{b}&\plaq{b}&\\
\plaq{b}&\plaq{b}&\plaq{b}\\
};
\begin{scope}[x=\loopcellsize,y=\loopcellsize]
\draw (loop-1-1.north west) -- ++(1,0) -- ++(0,-1) -- ++(1,0) -- ++(0,-1) -- ++(1,0) -- ++(0,-1);
\end{scope}
}};
\draw[decoration=brace,decorate] (-1.3,1) -- node[above=1mm] {$r=\{2,4,6\}$} (1.3,1);
\end{tikzpicture}
\]
\item[(2d)] $i\not\in r'$, $i\in r'$:
\[
\begin{tikzpicture}[baseline=0]
\node[rotate=45,inner sep=0pt] { \tikz{ 
\node[loop]
{
\plaq{b}&\plaq{a}&\\
\plaq{b}&\plaq{b}&\plaq{a}\\
\plaq{b}&\plaq{b}&\plaq{b}\\
};
\begin{scope}[x=\loopcellsize,y=\loopcellsize]
\draw (loop-1-1.north west) -- ++(2,0) -- ++(0,-1) -- ++(1,0) -- ++(0,-2);
\end{scope}
}
};
\end{tikzpicture}
=
\begin{tikzpicture}[baseline=0]
\node[rotate=45,inner sep=0pt] { \tikz{ 
\node[loop]
{
\plaq{b}&&\\
\plaq{b}&\plaq{b}&\\
\plaq{b}&\plaq{b}&\plaq{b}\\
};
\begin{scope}[x=\loopcellsize,y=\loopcellsize]
\draw (loop-1-1.north west) -- ++(1,0) -- ++(0,-1) -- ++(1,0) -- ++(0,-1) -- ++(1,0) -- ++(0,-1);
\end{scope}
}
};
\draw[decoration=brace,decorate] (-1.3,0.6) -- node[above=1mm] {$r'=\{2,4,6\}$} (1.3,0.6);
\end{tikzpicture}
\qquad\mapsto\qquad
\begin{tikzpicture}[baseline=0]
\node[rotate=45,inner sep=0pt] { \tikz{ 
\node[loop]
{
\plaq{b}&\plaq{a}&\plaq{b}\\
\plaq{b}&\plaq{b}&\plaq{a}\\
\plaq{b}&\plaq{b}&\plaq{b}\\
};
\begin{scope}[x=\loopcellsize,y=\loopcellsize]
\draw (loop-1-1.north west) -- ++(3,0) -- ++(0,-3);
\draw (loop-1-1.north west) ++(2,0) -- ++(0,-2) -- ++(-1,0) -- ++(0,1) -- ++(2,0);
\end{scope}
}};
\end{tikzpicture}
=
\begin{tikzpicture}[baseline=0]
\node[rotate=45,inner sep=0pt] { \tikz{ 
\node[loop]
{
\plaq{b}&&\\
\plaq{b}&\plaq{b}&\\
\plaq{b}&\plaq{b}&\plaq{b}\\
};
\begin{scope}[x=\loopcellsize,y=\loopcellsize]
\draw (loop-1-1.north west) -- ++(1,0) -- ++(0,-1) -- ++(1,0) -- ++(0,-1) -- ++(1,0) -- ++(0,-1);
\end{scope}
}};
\draw[decoration=brace,decorate] (-1.3,0.6) -- node[above=1mm] {$r=\{2,4,6\}$} (1.3,0.6);
\end{tikzpicture}
\]
i.e., $r'$ pairs $i$ and $i+1$, in which case effectively using $e_i^2\propto e_i$, we find
$e_ir'=r'$, and once again we note 
$s\subsetneq s'\subseteq r'=r$.
\end{enumerate}
\junk{AK: I feel that in the last three displayed equations, 
  the left two figures should be switched. PZJ: I don't have a strong opinion...
  AK: I suppose it's equation mapsto equation, instead of the second
  figure getting mapped}
Only possibility (2a) can lead to $r=s$, and only in the case
$r'=s'$ (as in the example); we immediately obtain inductively the formula for $Z_{r,r}$.

\junk{this is tetris: throw an extra plaquette into any subdiagram
of the complement of $s'$ 
at column of protruding box and see what happens: either it falls into
a valley (e.g., for the whole of $s'$), and all is ok, or it falls on a slope,
and then one needs to use the TL relation $e_i e_{i\pm1}e_i=e_i$ to kill
two boxes, and everything ok again, or a peak, $e_i^2=\tau e_i$}

\end{proof}

\begin{lem}\label{lem:sym}
Assume $i$ and $i+1$ are not connected in $r$. Then $\tau_i 
Z_{r,s} = Z_{r,\tau_i s}$. 
\end{lem}
\begin{proof}
This is a direct consequence of the two parts of Lemma~\ref{lem:sym0},
\end{proof}

\junk{\begin{rmk}
More generally (i.e., without any assumption on $r$), one can
specialize $y_i=z_{s_i}$ in the formula~\eqref{eq:divdiff} from the proof of Lemma~\ref{lem:sym0}, and obtain
\[
a(z_{i+1}/z_i)(Z_{r,s}-\tau_i Z_{r,\tau_i s}) = b(z_{i+1}/z_i) 
\sum_{r':\ e_ir'=r} \tau^{\delta_{r,r'}} \tau_i Z_{r',\tau_is}
\]
These expressions are all Laurent polynomials in the $z_i$, and therefore, $b(z_{i+1}/z_i)$ (or equivalently $z_i-z_{i+1}$) 
divides $Z_{r,s}-\tau_i Z_{r,\tau_i s}$. Since it also trivially divides $\tau_i Z_{r,\tau_i s} - Z_{r,\tau_i s}$,
we conclude
\[
z_i-z_{i+1} \big| Z_{r,s} - Z_{r,\tau_i s}
\]
This is nothing but the GKM criterion for being in the image of the
restriction map (see e.g. \cite[A.4]{KRosu}, \cite[corollary 5.11]{VV}), 
cf Conjecture~\ref{conj:ourconj1'}. This in turn
implies that the original expression in Conjecture~\ref{conj:ourconj1}
(r.h.s.\ of \eqref{eq:ourconj1}) is well-defined in (non-localized)
$K_T(T^\ast Gr(n,N))$.
\end{rmk}}

\section{Geometry and the conjectured sheaves}
\subsection{Parametrization}
We coordinatize the Grassmannian 
via its Pl\"ucker embedding:
\begin{equation}\label{eq:pluc}
Gr(n,N)=\left\{
\left[p_s, s\in \subs\right] \in \mathbb P^{{N\choose n}-1}:
\ \forall s_\pm \in \subsarg{n\pm1}{N},
\ 
\sum_{i\in s_+ \setminus s_-} p_{s_-\cup i}p_{s_+ \backslash i}=0
\right\}
\end{equation}
(We use here the implicit convention that when adding/subtracting indices from 
a subset, 
the index is added/subtracted at the end, but the sign of the permutation
sorting the indices into increasing order must be introduced.)

In these coordinates, Hodge showed that 
Schubert varieties and cells are easily defined as:
\[
X^r=\left\{
\left[p_s, s\in \subs\right]\in Gr(n,N)\ :
\ 
p_s=0\text{ unless }s\subseteq r
\right\},
\qquad
X_o^{r}=X^r\cap\{p_r\ne0\}
\]
One has $\dim X^r=|r|$, the area of the partition.

\newcommand\Mat{\text{Mat}}
The cotangent bundle $T^\ast Gr(n,N)$ can be identified with
\[
T^\ast Gr(n,N)=\left\{(V,u)\in Gr(n,N)\times \Mat(N)\ :\
\Im u \subseteq V \subseteq \Ker u\right\}
\]
and its projection $\mu$ (the \defn{moment map}) to the second factor has image
$\{u\in \Mat(N)\ :\ u^2=0, \text{rank}(u)\leq n\}$,
the closure of a nilpotent $GL(N)$-orbit (of which $\mu$ is the
Springer resolution, which won't be especially relevant).
In Pl\"ucker coordinates, and denoting by $M=(M_{i,j})$ the transposed (for convenience)
matrix of $u$, this space is defined by the following equations:
\begin{equation}\label{eq:paramT}
 T^\ast Gr(n,N)=\left\{
  \begin{aligned}[]
  \left[p_s, s\in \subs \right]\in Gr(n,N),\ 
  &(M_{i,j})\in \Mat(N)\ :
  \\ 
  \sum_{j\in s_+} p_{s_+\backslash j} M_{i,j} &=0,\ s_+\in \subsarg{n+1}{N}, 1\le i\le n
  \\
  \sum_{i\in\bar s_-} p_{s_-\cup i} M_{i,j} &=0,\ s_-\in \subsarg{n-1}{N}, 1\le j\le n
 \end{aligned} \right\}
\end{equation}


Let $\Mat_<(N)$ denote the space of strictly upper triangular matrices.
If we consider pairs $([p_s],M) \in T^*Gr(n,N)$ such that $M \in \Mat_<(N)$,
we get the union of the conormal bundles of Schubert cells, or equivalently,
the union of the conormal Schubert varieties $CX^r$; this is because
$M\in \Mat_<(N)$ is the moment map condition for the action of 
the Borel subgroup $B_-$ of invertible upper triangular matrices.
\begin{equation}\label{eq:paramC}
\bigcup_{r\in \subs} CX^r=\left\{
\begin{aligned}[]
\left[p_s, s\in \subs\right]\in Gr(n,N),\ 
&(M_{i,j})\in \Mat_<(N):
\\ 
\sum_{j\in s_+} p_{s_+\backslash j} M_{i,j} &=0,\ s_+\in \subsarg{n+1}{N}, 1\le i\le n
\\
\sum_{i\in\bar s_-} p_{s_-\cup i} M_{i,j} &=0,\ s_-\in \subsarg{n-1}{N}, 1\le j\le n
\end{aligned}
\right\}
\end{equation}
(the only difference from (\ref{eq:paramT}) 
being the $<$ subscript on the $\Mat$).

The $CX^r$ are the irreducible components of that space; under the map $\mu$
they are sent into the \defn{orbital scheme} 
$\{M\in \Mat_<(N)\ :\ M^2=0,\ \text{rank}(M)\leq \min(n,N-n) \}$.
More precisely, if the link pattern $r$ has the maximal number of chords,
then $CX^r$ is sent to an irreducible component of that orbit closure,
called an \defn{orbital variety} (the other images $\mu(CX^r)$ 
have smaller dimension than the components, i.e. the general fibers have
positive dimension; we work out these general fibers in \S~\ref{sec:fiber}).
As the action of $B_+$ on the nilpotent orbit closure $\{M\ :\ M^2=0\}$
has finitely many orbits (it is \defn{spherical}),
each $\mu(CX^r)$ is a $B_+$-orbit closure.
It is easy to describe a representative of that orbit \cite{artic39}:
define $r_<$ to be the upper triangular matrix with $1$s at $(i,j)$ 
for each chord $i<j$ of the link pattern of $r$, $0$s elsewhere. Then
\[
\mu(CX^r)=\overline{B_+\cdot r_<}
\]
where $\cdot$ means conjugation action.

Rank conditions of Southwest submatrices are preserved by $B_+$-conjugation, 
so that we find the following inclusion
and (by working a little harder \cite{Rothbach-PhD}) even the equality of sets
\begin{multline}\label{eq:rk}
  \mu(CX^r)=\{M=(M_{i,j}): M^2=0, \text{ and for each $(i,j)$,} \\
  \ \text{rank of M Southwest of $(i,j)$}\le 
  \text{rank of $r_<$ Southwest of $(i,j)$}\}.
\end{multline}

\junk{but by recent paper they're Cohen--Macaulay (do I mean the actual reduced
objects or the conjectured ones?? that depends on whether we have
enough equations to define the prime ideal), so just need to check
generic reducedness to conclude (can even do the crossing link pattern
case for the same price,
which we may already have done in our brauer loop scheme papers). phew. though of course we don't care here. }

\subsection{Reformulation of conjectures~\ref{conj:ourconj2} and \ref{conj:geomRS}}\label{sec:reform2}
Since $\mu:\ \bigcup_r CX^r \to \Mat_<(N)$ is proper, we can define 
$\mu_*[\sh_r] \in K_T(\Mat_<(N)) \iso K_T(pt) 
\iso  \ZZ[t^\pm,z_1^\pm,\ldots,z_N^\pm]$ without localization 
(i.e. without tensoring with the fraction field of $K_T(pt)$). Its relation
\begin{equation}\label{eq:pi2mu}
  \mu_* [\sh_r] = \pi_* [\sh_r] \prod_{1\le i<j\le N} (1-t\, z_i/z_j)
\end{equation}
to $\pi_*[\sigma_r]$, the (improper) pushforward to a point, derives from
the weights $t\, z_i/z_j$, $1\le i<j\le N$ of the $T$-action on 
the affine space $\Mat_<(N)$.

Conj.~\ref{conj:ourconj2} can then be reformulated
in terms of $\mu_* [\sh_r]$ as
{\renewcommand\theconj{2'}\begin{conj}\label{conj:ourconj2'}
The pushforward of $\sh_r$ to $\Mat_<$ in $T$-equivariant $K$-theory
is equal, up to normalization, to $\Psi_r$:
\[
\mu_*[\sh_r]=
\begin{cases}
0& r\not\subseteq \stc
\\
(1-t)^{n(n-1)} \tilde m_r
\ 
\Psi_r
&r\subseteq\stc
\end{cases}
\]
\end{conj}}
\noindent
and similarly, Conj.~\ref{conj:geomRS} can be rewritten
assuming Conj.~\ref{conj:ourconj2} as
{\renewcommand\theconj{3'}\begin{conj}\label{conj:geomRS'}
$\mu_* [\sh_r]$ can be decomposed as a sum of products of the form
\[
\mu_*[\sh_r]= \sum_{f\in \FPL_r} m_f\,
\prod_{a=1}^{n(n-1)}
\left(1-t^{r_{f,a}} z_{i_{f,a}}/z_{j_{f,a}}\right)
\]
where $r_{f,a}\in \{1,2\}$, and $m_f$ is a monomial (explicitly,
$m_f=\tilde m_r 
t^{-|\{a:\,r_{f,a}=2\}|/2}
\prod_{a=1}^{n(n-1)} z_{j_{f,a}}
$).
\end{conj}}

We comment on the vanishing conclusion of conjecture \ref{conj:ourconj2'}. 
If we coarsen from $K$-homology to ordinary (Borel--Moore) homology, the 
$K$-class $[\sigma_r]$ maps to the fundamental class $[CX^r]$. 
As we will show in \S~\ref{sec:fiber}, the condition $r\subseteq\stc$
is equivalent to $\mu:\ CX^r \to \mu(CX^r)$ being birational, so when
$r\not\subseteq\stc$ we get the homology vanishing result $\mu_*[CX^r] = 0$. 
In this sense, the subtle construction of $\sigma_r$ in the next section is 
our attempt to refine this simple vanishing in homology to a much more precise
vanishing in $K$-homology. (What one learns for free is that the
sheaves $R\mu_i \sigma_r$ are supported on proper subschemes of $\mu(CX^r)$
when $r\not\subseteq\stc$, rather than
learning that their support is actually empty.)

\junk{note that in cohomology, (1) and (2) are equivalent (somewhat similarly to the fact that AJS--Billey and pipedream formulae are equivalent):
(1) implies (2) is obvious (localization formula), but (2) also implies (1) by using a bigger system and slicing... not sure if the trick still works in $K$-theory though
-- need to slice modules instead of rings} 
\junk{AK: ``obvious''? If it's really that easy then yeah I want it in
  the paper}

\subsection{The conjectured sheaves $\{\sigma_r\}$}\label{ssec:beyond}
We define a sheaf $\sh_r$ on any conormal Schubert variety 
$CX^r \subseteq T^* G/P$, although our most general conjectures
concern the case $G/P$ a Grassmannian.

We will need to twist the structure sheaf of 
$CX^r := \overline{C X_\circ^r}$
by a line bundle not available on $T^* G/P$. So let $g:\ G/B\onto G/P$
be the $G$-equivariant projection, 
and $w \in W$ the minimum-length lift of 
$r \in W/W_P$, making $g:\ X^w_\circ \to X^r_\circ$ an isomorphism.
These give us the commuting squares
\[\begin{matrix}
  G/B &{\from}&  g^*(T^* G/P)  &\infrom&  \wt{CX^r}  
&:=\text{ closure of}& 
  \{(x \in X^w_\circ, \vec v\in T^*_{f(x)} G/P)\ :\ \vec v \perp T_{f(x)} X^r_\circ \}
\\
  \phantom{g}\downarrow g &&  \downarrow &&   \downarrow && \downarrow\!\wr  \\
G/P &\from& T^* G/P   &\infrom& {CX^r} 
&:=\text{ closure of}& 
\{(y \in X^r_\circ, \vec v \in T^*_y G/P)\ :\ \vec v \perp T_{y} X^r_\circ \}
\end{matrix}
\]
where the closures taken of the fourth column, to define the third, 
are taken inside the second. 
This fourth vertical map (before taking the closures), 
taking $(x,\vec v) \mapsto (g(x),\vec v)$, is an isomorphism. 
Hence its closure $\wt{CX^r} \to CX^r$ is birational; call this map $Cg$. 
Denote the composite $G/B \from \wt{CX^r}$ (with image $X^w$) 
of the left two maps on top by $f_{G/B}$,
and the corresponding composite $G/P \from {CX^r}$ (with image $X^r$) 
on bottom by $f_{G/P}$. 
The first and third columns are then a commuting square
$$ \begin{matrix}
  X^w & \stackrel{f_{G/B}}{\from} &   \wt{CX^r} \\
  \phantom{g}\downarrow g &&  \phantom{Cg}\downarrow Cg \\
  X^r & \stackrel{f_{G/P}}{\from} &   CX^r 
\end{matrix} $$

The space $G/P$ comes with a list of fundamental weights $\omega_i$
orthogonal to the negative simple roots in $P$. 
Let $\omega_{G/P}$ be the sum of these, i.e. its Borel--Weil line bundle
$\calO(\omega_{G/P})$ is the smallest ample line bundle on $G/P$.

On some homogeneous spaces $G/Q$, in particular every $G/B$, the anticanonical 
line bundle possesses a (unique) square root which we will denote 
$\calO(\rho_{G/Q})$, e.g. $\rho_{G/B} = \omega_{G/B}$. 
(By contrast, for $G/P = Gr(n,N)$
we have $\rho_{G/P} = \frac{N}{2} \omega_{G/P}$.)
Then $\calO(-\rho_{G/Q})$ is the Borel--Weil line bundle corresponding
to $-\rho_{G/Q} := -\frac{1}{2} \sum_{\beta\in \Delta^G_+ \setminus \Delta^Q_+} \beta$,
and all of its sheaf cohomology groups vanish (unless $P=G$).
{\em We assume that $G/P$ has such a square root,} 
with which to define our sheaf:
\begin{eqnarray*}
  \sh_r 
&:=& f_{G/P}^*\calO(\rho_{G/P} - \omega_{G/P}) 
\tensor (Cg)_* \, f_{G/B}^* \, \calO(-\rho_{G/B})
= (Cg)_* \, f_{G/B}^* \, \calO\left(-g^*(\omega_{G/P})-\frac{1}{2}\sum_{\beta \in \Delta^P_+} \beta\right)  
\end{eqnarray*}
For example if $G/P=G/B$
(where $\rho_{G/P}=\omega_{G/B}$ and $Cg=Id$), 
this simplifies to $f_{G/B}^* \, \calO(-\rho_{G/B})$.

We now describe this construction in the $Gr(n,N)$ case,
with more detail (when $X^r$, also, is a Grassmannian) to come in 
\S\ref{sec:shea}. 
The anticanonical line bundle is $\calO(N)$, possessing a
square root iff $N$ is even (which we therefore assume),
and $\calO(\omega_{G/P}) = \calO(1)$. 
\junk{  We ignore issues of $T$-equivariance, which will only contribute a
  monomial shift to the character. }

\junk{and thence to the $(a,b,c)$ case of this paper. 
  ..., and show how it reduces to the 
  line bundle $\calO(a)$ in the $(a,b,c)$ case of this paper.}

First, for each weight $\kappa$ the Borel--Weil line bundle 
$\calO(\kappa)$ on $G/B$ can be identified 
with the sheaf of rational functions on $G/B$ bearing
poles only along the Schubert divisors $X_{r_\alpha}$, of order 
at most $\langle \alpha,\kappa \rangle$. For $\kappa=-\rho$
these orders-of-pole are all $-1$, i.e. the sections must vanish 
along $\bigcup_\alpha X_{r_\alpha}$.

We want the intersection of this divisor $\bigcup_\alpha X_{r_\alpha}$ with 
$f\big(\wt{X^r}\big) = X^w$. By Monk's rule, the intersection of $X_{r_k}$ 
with $X^w$ is (non-equivariantly) rationally equivalent to
$\sum_{i\leq k<j\atop w \circ (i \leftrightarrow j) \lessdot w} [X^{w \circ (i \leftrightarrow j)}]$
where $\lessdot$ indicates a strong Bruhat cover.
Summing over $k$, we get
$\sum_{w \circ (i \leftrightarrow j) \lessdot w} (j-i)\, [X^{w \circ (i \leftrightarrow j)}]$.

Now we want to push this to $CX^r$, where the term
$[X^{w \circ (i \leftrightarrow j)}]$ drops dimension unless 
$w \circ (i \leftrightarrow j)$ is again $n$-Grassmannian,
mapping to some $X^{r'}$ with $r'$ being $r$ minus an outer corner.
If that removed box is in position $(x,y)$, then $i=n+1-x$ and $j=n+y$,
hence $j-i$ is the diagonal ($:= x+y-1$) of the removed box.

\junk{Slowly here... $\dim X^r=|r|$, inside $n\times (N-n)$ box. 
  If $r=\emptyset$ then $w=123\ldots n|n+1\ldots N$,
  i.e. $w$ lists the $n$ up-moves then the $N-n$ right-moves from the
  partition's word constructed from SW to NE. If the box is at $(x,y)$
  then the right-move before its corner is listed at position $i=n+y$ in
  $w$, and the up-move after is at position $j=n+1-x$. 
  Hence $j-i = 2n+1-x-y$. }

To restate: if we push this class on $G/B$ down to $G/P$, we get 
$\sum_{r' = r\setminus (x,y)} (x+y-1) [X^{r'}]$, which (because of those $-1$s)
is off by $\calO(1)$ from the anticanonical class of $X^r$
\cite[proposition 2.2.8(iv)]{Brion-flag}.
In particular, the sheaf $\sigma_r$ is a line bundle iff $X^r$ is Gorenstein,
which (by \cite{WY-SchubGroth}) happens exactly 
when all the outer corners $(x,y)$
of the partition $r$ have $x+y-1=d$ for the same $d$. 
In this Gorenstein case we can therefore skip construction of $\wt{CX^r}$
and just pull $\calO(-d)$ directly from $G/P = Gr(n,N)$ to $CX^r$.

\junk{
In the special case that all outer corners of $r$ lie in the
same diagonal $d$, then
$$ \sum_{w \circ (i \leftrightarrow j) \lessdot w} 
(j-i)\, [X^{w \circ (i \leftrightarrow j)}] 
\quad\mapsto\quad d \sum_{r' \lessdot r} [X^{r'}] = d\, [\calO(1)] = [\calO(d)] $$
i.e. this divisor class is the restriction of the line bundle $\calO(d)$
from the Grassmannian $G/P = Gr(n,N)$. 
As such, in this special case, we can skip constructing $\wt{CX^r}$
and just pull $\calO(-d)$ directly from $G/P$ to $CX^r$.
It is interesting that this same-diagonal condition is also 
the one for $X^r$ to be Gorenstein \cite{WY-SchubGroth}. 
\rem{PZJ: clearly, this needs rewriting in view of my recent findings. 
  AK: I think this is the spot to state it as a conjecture, 
  in the Grassmannian case only of course. PZJ: what conjecture? the fact
that the sheaf on the base is the canonical sheaf is a theorem, not a conjecture}
  }

Finally, we twist by $\calO(\rho_{G/P}-\omega_{G/P}) = \calO(N/2-1)$,
obtaining (in this Gorenstein case) the sheaf 
$$ \sh_r = f_{G/P}^* \, \calO(N/2-d-1).$$
It will be convenient below to take $d':=d+1$, i.e. $d' = x+y$ for each
outer corner $(x,y)$. 

\junk{Example: $a=b=c=1$, so $r=001011$, $w=124356$. 
  The maps $X^w \to X^r$ and $\wt{CX^r} \to CX^r$ are isomorphisms. }

\subsection{Fibers of \texorpdfstring{$\mu$}{mu}, 
  and a small part of conjecture \ref{conj:ourconj2'}}\label{sec:fiber}

\junk{The image of $\mu:\ CX^r \to \Mat(N)_<$ is closed (since $\mu$ is proper),
  irreducible (since $X^r$ and $CX^r$ are), and $B$-invariant, 
  hence is a $B$-orbit on $\{M\in \Mat(N):\ M^2=0\}$
  (since there are only finitely many orbits). Hence by \cite{Meln}
  $\mu(CX^r)$ is $\overline{B\cdot \pi_<}$ for a unique involution $\pi$}

The general fiber of $CX^r \to \overline{B\cdot r_<}$ can be computed
most easily at the point $r_<$. 
\junk{First, we claim that it is enough
to consider the fiber inside $C(X^r_\circ)$ and take the closure.
\rem{What's the easy argument? The issue is that the fiber might have
  multiple components, because $B\cdot r_<$ is not simply connected.}}
That fiber consists of $V \in X^r_\circ$ with $\Im r_< \leq V \leq \Ker r_<$.
To describe it, we look back at the second picture from (\ref{eq:bij}),
whose $n$ top edges run either SW/NE or NW/SE.
If we picture $X^r$ as the row-spans of row-echelon
$k\times n$ matrices $V$, drawn atop a partition tilted as in that figure,
then we have three conditions to impose on the matrix $V$:
\begin{enumerate}
\item Its row-span should lie in $X^r$. So we can assume that it is zero
  in row $i$ to the left of $r$'s $i$th element, which lies above
  the $i$th NW/SE edge. We don't go so far as to assume that $V$ is in 
  reduced row-echelon form with pivots $1$ in the $r$ columns, 
  as that would only get us $X^r_\circ$ not its closure $X^r$.
  \junk{ $|$partition$|$ = dimension, so full partition = pivots on far left}
\item $\Im r_< \leq \text{rowspan}(V)$. Equivalently, there {\em is indeed}\/ a pivot
  above any NW/SE edge connecting to another edge (necessarily, 
  connecting Westward to a NE/SW edge),
  and the rest of its row in $V$ must be all zeroes.
  Of course we can use the pivotal $1$ and row operations 
  to kill the rest of the column, too.
\item $\text{rowspan}(V) \leq \Ker r_<$. Equivalently, each NE/SW edge
  that connects to another edge (necessarily, Eastward to a NW/SE)
  should have an entirely $0$ column.
\end{enumerate}
Let $R$ be the number of pairings, so $c' = n - R$, $b' = N-n-R$ are the
number of unmatched red dots in the link pattern 
with rays going left and right, respectively.
Then matched columns are constrained by (2), (3) above, and unmatched
columns contribute to $c'$, $b'$, as indicated below the matrix
in this example:
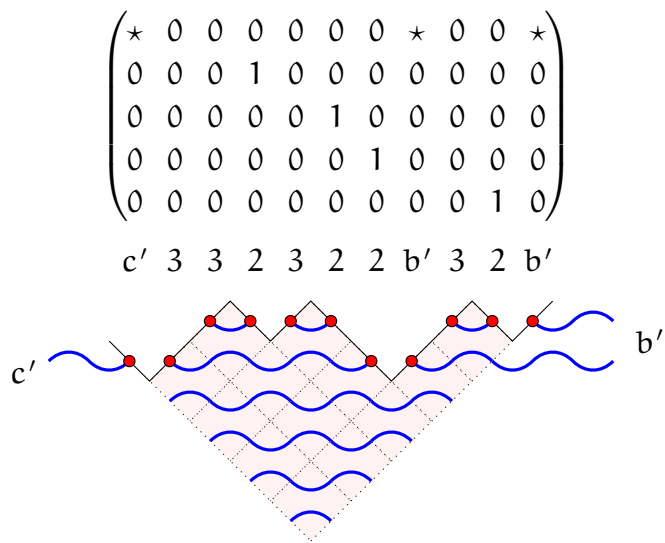
\setcounter{MaxMatrixCols}{100} 

\begin{center}
\begin{tikzpicture}[baseline=0]
\node[matrix of math nodes,column sep={\loopcellsize*0.707,between origins}] 
at (0,2.8) (mat) {
  \star & 0 & 0 & 0 & 0 & 0 & 0 & \star & 0 & 0 & \star \\
  0 & 0 & 0 & 1 & 0 & 0 & 0 & 0 & 0 & 0 & 0 \\
  0 & 0 & 0 & 0 & 0 & 1 & 0 & 0 & 0 & 0 & 0 \\
  0 & 0 & 0 & 0 & 0 & 0 & 1 & 0 & 0 & 0 & 0 \\
  0 & 0 & 0 & 0 & 0 & 0 & 0 & 0 & 0 & 1 & 0 \\[2mm]
  c'& 3 & 3 & 2 & 3 & 2 & 2 & b'& 3 & 2 & b'\\
};
\node at (mat-3-1.west) {$\left(\vphantom{\rule{2pt}{1.5cm}}\right.$};
\node at (mat-3-11.east) {$\left.\vphantom{\rule{2pt}{1.5cm}}\right)$};
\useasboundingbox (-5,-2.7) rectangle (5,1);
\node[rotate=45,inner sep=0pt] { \tikz{ 
\node[loop,outer sep=0pt]
{
\node[rectangle,draw=none,minimum size=\loopcellsize] (loop-\the\pgfmatrixcurrentrow-\the\pgfmatrixcurrentcolumn) {};&&&&&\\
\plaq{b}&\plaq{b}&&&\\
\plaq{b}&\plaq{b}&\plaq{b}&&&\\
\plaq{b}&\plaq{b}&\plaq{b}&&&\\
\plaq{b}&\plaq{b}&\plaq{b}&\plaq{b}&\plaq{b}&\\
};
\begin{scope}[x=\loopcellsize,y=\loopcellsize]
\draw (loop-1-1.north west) -- ++(0,-1) -- node[/linkpattern/vertex] {} ++(1,0) -- node[/linkpattern/vertex] {} ++(1,0) -- node[/linkpattern/vertex] {} ++(0,-1) -- node[/linkpattern/vertex] {} ++(1,0) -- node[/linkpattern/vertex] {} ++(0,-1) -- node[/linkpattern/vertex] {}  ++(0,-1) -- node[/linkpattern/vertex] {} ++(1,0) -- node[/linkpattern/vertex] {} ++(1,0) -- node[/linkpattern/vertex] {} ++(0,-1) -- ++(1,0);
\draw[/linkpattern/edge] (loop-1-1.west) node[/linkpattern/vertex] {} .. controls ++(-0.3,0) and ++(0,-0.3) .. ++(-0.5,0.5) .. controls ++(0,0.3) and ++(0.3,0) .. ++(-0.5,0.5) node[left=2mm,rotate=-45] {$c'$};
\draw[/linkpattern/edge] (loop-5-5.south) ++(1,0) node[/linkpattern/vertex] {} .. controls ++(0,-0.3) and ++(-0.3,0) .. ++(0.5,-0.5) .. controls ++(0.3,0) and ++(0,0.3) .. ++(0.5,-0.5);
\draw[/linkpattern/edge] (loop-5-5.south) .. controls ++(0,-0.3) and ++(-0.3,0) .. ++(0.5,-0.5) .. controls ++(0.3,0) and ++(0,0.3) .. ++(0.5,-0.5) .. controls ++(0,-0.3) and ++(-0.3,0) .. ++(0.5,-0.5)  ++(0.5,0) node[right,rotate=-45] {$b'$};
\end{scope}
}};
\end{tikzpicture}
\end{center}

If we remove the $0$ columns from condition (3) leaving $N-R = n+b'$ columns, 
and remove the rows and columns of the $R = n-c'$ pivots from condition (2), 
we are left with a $c'\times (b'+c')$ matrix full of $\star$s.
Therefore our fiber is the subGrassmannian $Gr(c',c'+b')$ of $n$-planes
inside a fixed $(n+b')$-plane, and containing a fixed $(n-c')$-plane.

When all the outer corners of the partition $r$
are on the same diagonal $d$ (the Gorenstein case),
and letting $d':=d+1$, 
then $c' = \min(0,d'-n)$, $b' = \min(0,d'-(N-n))$.
If we assume the fiber $Gr(c',c'+b')$ is not a point, 
then $d' > \max(n,N-n) \geq N/2$ and $c'+b' = 2d'-N$.
On a Grassmannian $Gr(k,m)$ which isn't a point, the
line bundles $\calO(j)$ have no%
\footnote{\label{foot:BWB}
  Proof: $\calO(j)$ on $Gr(k,m)$ is the pushforward from $GL(m)/B$ of 
  the $\calO(j^k,0^{m-k})$ Borel--Weil line bundle, so instead of pushing
  $\calO(j)$ from $Gr(k,m)$ to a point we can push $\calO(j^k,0^{m-k})$
  from $GL(m)/B$ to a point (through $Gr(k,m)$).
  By Borel--Weil--Bott,
  a line bundle $\calO(\lambda_1\geq\ldots\geq\lambda_m)$ 
  on $GL(m)/B$ has no cohomology iff $\lambda+(m,m-1,\ldots,1)$
  has a repeat, as $\rho_{GL(m)/B} = (m,m-1,\ldots,1)$ up to a constant.
  Here, that sequence (reversed) is $[1,m-k], [m-k+1+j,m+j]$,
  and these two intervals avoid overlap only if $m-k+1+j > m-k$ i.e. $j>-1$,
  or if $m+j<1$ i.e. $j<1-m$, or if one of the intervals is empty
  i.e. $k=0,m$. 
}
sheaf cohomology if $0>j>-m$.
Our line bundle $\calO(N/2-d')$ on the fiber $Gr(c',c'+b')$ is exactly 
halfway through this cohomology desert:
$$ 0 \quad>\quad N/2-d' \quad>\quad -(c'+b') = 2(N/2-d') $$

The fact $H^0($fiber$;\, \calO(N/2-d'))=0$ alone is already enough to 
demonstrate the $H^0$ vanishing statement of Conjecture~\ref{conj:ourconj2'}, 
in this Gorenstein case, since a section in $H^0(CX^r; \sigma_r)$ 
that vanishes on the general fiber must vanish everywhere.  

\section{The case of $r$ a \texorpdfstring{$c\times b$}{c x b} rectangle}
We now 
consider the Young diagram $r$ which is a $c\times b$
rectangle, as in \eqref{eq:defabc1},
and write $c\times b$ for this partition. 
Rectangles are precisely those diagrams
for which the Schubert variety $X^r$ is smooth, and in this case $X^{c\times b}$
is isomorphic to the Grassmannian $Gr(b,b+c)$. The embedding inside
$Gr(n,N)$ is particularly simple in Pl\"ucker coordinates, namely
\[
X^{c\times b}=
\left\{
[p_s,s\in S]: 
p_s=\Bigg\{ \begin{aligned}
&p_{\tilde s}&& s=\{\tilde s_1+\bar n-c,\ldots,\tilde s_b+\bar n-c,\bar n+b+1,\ldots,N\}\\
&0&& s\cap \{1,\ldots,\bar n-c\}\ne \O\text{ or } 
s\not\supseteq \{\bar n+b+1,\ldots,N\}
\end{aligned}
\right\}
\]
where $\bar n=N-n$,
and the $p_{\tilde s}$ are the Pl\"ucker coordinates of $Gr(b,b+c)$.
From now on we use these ``reduced'' indices consisting of subsets $\tilde s$
of $\{1,\ldots,b+c\}$ of cardinality $b$, i.e., in $\subsarg{b}{b+c}$, as well as the corresponding order $\le$ (or its opposite order $\subseteq$).

This rectangular case is even more special than the Gorenstein case 
(discussed at the end of \S\ref{ssec:beyond}) 
where the outer corners of $r$ lay in a single diagonal: 
now $r = c\times b$ has only one outer corner, in diagonal $d=c+b-1$, 
and the sheaf $\sigma_r$ is $\calO(N/2-d-1) = \calO(a)$.
(Recall that we are assuming $N$ even whenever we discuss the sheaf $\sigma_r$,
and have defined $a := N/2-b-c$.)
Since $X^{c\times b}$ is itself a Grassmannian 
it has its own $\calO(1)$ line bundle, 
which conveniently is the restriction of the $\calO(1)$ from $Gr(n,N)$.
Hence our space and sheaf are simply the conormal bundle
$C_{Gr(n,N)} Gr(b,b+c)$ and its $\calO(a)$ line bundle
pulled up from the base $Gr(b,b+c)$. 

\junk{PZJ: at which stage did we get rid of $a<0$? AK: later,
  in prop \ref{prop:noHi}, as is now explained there}

\subsection{Equations of $CX^{c\times b}$}\label{ssec:eqnscx}
We now wish to write the equations of the one conormal bundle $CX^{c\times b}$
(not, as in (\ref{eq:paramC}), the union of many such). 
Let $f$ be the map $(V,u)\mapsto V$.
Appending Eqs.~\eqref{eq:paramC}, rewritten in terms of the remaining Pl\"ucker 
coordinates, to the equations of $X^{c\times b}$ above
would produce the union of all conormal Schubert varieties 
inside $f^{-1}(X^{c\times b})$, namely $\bigcup_{s\subseteq r} CX^s$;
to exclude the others we need equations involving $M$ as well. 

We proceed as follows. 
Since $X^{c\times b}$ is smooth, $CX^{c\times b}\stackrel{f}{\to} X^{c\times b}$
is a vector bundle, and the defining equations of $CX^{c\times b}$ must be linear
in the fiber $(M_{i,j})$.
We therefore select among Eqs.~\eqref{eq:rk} those that are linear: 
they are of the form ``$M_{i,j}=0$ if there are no nonzero entries
of $r_<$ Southwest of $(i,j)$, that is, if there are no chords of $r$ inside 
the interval $[i,j]$''. This implies the following block structure of $M$:
\begin{equation}\label{eq:block0}
M=\bordermatrix{&\bar n-c&b+c&n-b\cr
\bar n-c&0&B&\star\cr
b+c&&0&C\cr
n-b&&&0\cr}
\end{equation}
where the upper-right block has not been named since its
entries never occur in any equation.

We can now write Eqs.~\eqref{eq:paramC} in terms of the submatrices $B$ and $C$:
\begin{align}
\sum_{j\in s_+} B_{i,j} p_{s_+\backslash j}&=0,\qquad s_+\in \subsarg{b+1}{b+c},\ 1\le i\le \bar n-c\label{eq:B0}
\\
\sum_{j\in\bar s_-} C_{j,k} p_{s_-\cup  j}&=0,\qquad s_-\in \subsarg{b-1}{b+c},\ 1\le k\le n-b\label{eq:C0}
\end{align}

In order to check that we have obtained {\em all}\/ the equations of the vector bundle
$CX^{c\times b}$, it is convenient to use the following observation. Inside $GL(N)$ acting
on $Gr(n,N)$ and therefore on $T^\ast Gr(n,N)$, the subgroup
$G:=GL(\bar n-c)\times GL(b+c)\times GL(n-b)$ leaves $X^{c\times b}$ invariant
and the $GL(b+c)$ factor acts transitively on the base, making $CX^{c\times b}$
a $G$-{\em equivariant}\/ vector bundle. Now by inspection, the block structure
\eqref{eq:block0} as well as Eqs.~(\ref{eq:B0},\ref{eq:C0}) are $G$-invariant, so that we only need
to check that the fiber of the vector bundle $CX^{c\times b}\stackrel{f}{\to} X^{c\times b}$
has the correct dimension at one particular point, 
for example a coordinate subspace with coordinates $s$
(where all $p_t = 0$ for $t\neq s$).
For each $s_+\in \subsarg{b+1}{b+c}$ containing $s$, one $B_{i,j}$ vanishes, 
and similarly for $s_-\in \subsarg{b-1}{b+c}$ contained in $s$, 
one $C_{j,k}$ vanishes. So we find
\begin{alignat*}{5}
\dim(\text{fiber})&=\dim(B)&&+\dim(C)&&+\dim(\star)&&-\dim(B\text{ eqs})&&-\dim(C\text{ eqs})
\\
&=(\bar n-c)(b+c)&&+(n-b)(b+c)&&+(\bar n-c)(n-b)&&- c(\bar n-c)&&- b(n-b)
\\
&=n\bar n-bc
\end{alignat*}
which means the total space has dimension $n\bar n=n(N-n)$, which is indeed the dimension 
of $CX^{c\times b}$ (a Lagrangian subvariety of $T^\ast Gr(n,N)$).

\newcommand\rank{\text{rank}}
Note that the other equations of~\eqref{eq:rk} are
\begin{align}
BC&=0 \label{eq:quad}\\
\rank(B)&\le b \label{eq:Brank}\\
\rank(C)&\le c \label{eq:Crank}
\end{align}
The derivation just performed shows that they could be obtained from
Eqs.~(\ref{eq:B0}, \ref{eq:C0}), up to saturation w.r.t.\ the irrelevant
ideal generated by the $p_s$. They define an $A_3$ quiver locus \cite{KMS} 
whose degree we will compute below in proposition \ref{prop:orbvardeg},
to be used in \S \ref{ssec:total}.

Summarizing: the embedding $CX^{c\times b} \into Gr(b,b+c) \times \Mat(N)$
gives the realization
\begin{align}\label{eq:presentation}
  CX^{c\times b} &\cong Proj\ 
  \CC \left[(p_s)_{s\in \subsarg{b}{b+c}}, (B_{i,j})_{[\bar n-c]\times[b+c]},
    (C_{j,k})_{[b+c]\times[n-b]}, \star_{[\bar n-c]\times[n-b]}\right]   \\
  & \bigg/ \ \langle\text{Pl\"ucker relations in the }(p_s),\ 
  \text{and Equations }(\ref{eq:B0})-(\ref{eq:Crank}) \rangle \nonumber 
\end{align}
where the Pl\"ucker coordinates have degree $1$, and the $B,C$ degree $0$,
in this $\NN$-graded ring.
\junk{The $\star$ variables in the upper right block are an annoying distraction,
and to help indicate their unimportance we won't even bother to give
them a letter. \rem{the referee, whoever he is, will dislike this... besides we've already said this}}

\subsection{Proof of Conj.~\ref{conj:ourconj1}
in the rectangular case} 
We start with the analysis of $Z_{r,s}$.

\begin{prop} For $r$ the $c\times b$ rectangle,
$Z_{r,s}$ is entirely determined by Lemmas~\ref{lem:zero} and \ref{lem:sym}.
\end{prop}
\begin{proof}
According to Lemma~\ref{lem:zero}, $Z_{r,s}$ is zero unless the Young diagram of $s$ sits inside the $c\times b$ rectangle, or equivalently,
the subset $s$ is of the form
$s=\{\tilde s_1+\bar n-c,\ldots,\tilde s_b+\bar n-c,\bar n+b+1,\ldots,N\}$, 
$\tilde s\in \subsarg{b}{b+c}$. Note that all such subsets can be obtained
from $r=\{\bar n-c+1,\ldots,\bar n+b-c,\bar n+b+1,\ldots,N\}$ by a permutation which acts nontrivially only on $\{\bar n-c+1,\ldots,\bar n+b\}$. Now there
are no chords in this subset in $r$, cf \eqref{eq:defabc2}. 
Lemma~\ref{lem:sym} then implies that 
all nonzero specializations of $Z_{r,s}$ are obtained from
each other, and in particular from $Z_{r,r}$, by permutation of the $z_{\bar n-c+1},\ldots,z_{\bar n+b}$. And $Z_{r,r}$ is itself
given by Lemma~\ref{lem:zero}.
\end{proof}


Next, we analyze the sheaf $\sh_r$ and show that
the restrictions $[\sh_r]|_s$ of its $K_T$-class to fixed points
satisfies the same properties as $Z_{r,s}$'s, up to normalization.
 
If $s\not\subseteq r$, we have $\mathbb C^s \not\in CX^{c\times b}$, and therefore the restriction is trivially zero.
At the particular fixed point $s=r$, the weight of the 
line bundle itself is $\prod_{i\in s} z_i^{-a}$, 
and the weights in the normal directions
to $CX^{c\times b} \subseteq T^*Gr(n,N)$ are given by a standard calculation, resulting in
\begin{equation}\label{eq:shll}
  [\sh_{r}]|_{r}=
  \prod_{i\in r} z_i^{-a} 
  \prod_{i\in r,\ j\not\in r}
  \begin{cases}
    1-t z_j/z_i& i<j
    \\
    1-z_i/z_j & i>j
  \end{cases}
\end{equation}
\junk{matches nicely
\[
m_r 
\left(-t^{1/2}\right)^{-bc}
\left(\prod_1^{a+b} z\right)^{n/2}
\left(\prod_{a+b+1}^{a+b} z\right)^{(c-a-b)/2}
\left(\prod_{a+2b+1}^{a+2b+c} z\right)^{(a+c-b)/2}
\left(\prod_{a+2b+c+1}^{2n} z\right)^{-n/2}
=
\left(\prod_{i\in s} z_i\right)^{-a}
\]
from the partition function.
}

Furthermore, the already mentioned $GL(b+c)$-equivariance ensures that the restrictions at the fixed points
indexed by
$s=\{\tilde s_1+\bar n-c,\ldots,\tilde s_b+\bar n-c,\bar n+b+1,\ldots,N\}$, 
$\tilde s\in \subsarg{b}{b+c}$, are related to each
other by the permutation of the $\{ z_{\bar n-c+1},\ldots,z_{\bar n+b}\}$ which sends $s$ to $r$.

Comparing Eq.~\eqref{eq:shll} with Eqs.~\eqref{eq:ourconj1'} and \eqref{eq:Zll},
and carefully keeping track of the monomials,
we conclude that Conjecture~\ref{conj:ourconj1'} 
from \S \ref{sec:reform2}
is verified with
\[
m_r= t^{bc/2}
\prod_{i=1}^{\bar n-c} z_i^{-n/2}
\prod_{i=\bar n-c+1}^{\bar n+b} z_i^{b-n/2}
\prod_{i=\bar n+b+1}^{N} z_i^{b+c-n/2}
\]

\subsection{Vanishing of higher cohomology}\label{sec:shea}

Recall that $f:\ CX^{c\times b} \to X^{c\times b} \subseteq Gr(n,N)$ is the 
projection of the 
conormal bundle, and recall from the beginning of this section that
\[
\sh_r = f^* \mathcal O(a)
\]
%
where $a = N/2-b-c$ and is integral.
Conjecture \ref{conj:ourconj2} concerns the pushforward of this line bundle
to a point, which comes only from its global sections:

\begin{prop}\label{prop:noHi}
  If $a\geq 0$,
  then the line bundle $\sh_{c\times b}$ on $CX^{c\times b}$ 
  has no higher sheaf cohomology, and the restriction map on sections 
  $H^0(CX^{c\times b};\ f^* \calO(a)) \to H^0(X^{c\times b};\ \calO(a))$ 
  is surjective.
  If $a < -|N/2-n|$, then $\sh_{c\times b}$ has no sheaf
  cohomology at all.
\end{prop}

\newcommand\Oplus\bigoplus

\begin{proof}
  We are grateful to Jake Levinson for explaining Weyman's sheaf
  cohomology techniques \cite{Weyman} to us, for the following application.

  As a vector bundle over $Gr(b,b+c)$, and ignoring the irrelevant $\star$ 
  variables of (\ref{eq:block0}), the space $CX^{c\times b}$ is the vector bundle 
  $$ Hom(\CC^{\bar n-c},S)\tensor Hom(Q,\CC^{n-b}) 
  \iso (Q^*)^{\oplus n-b} \oplus S^{\oplus \bar n - c} $$
  where $S,Q$ are the tautological sub and quotient bundles on $Gr(b,b+c)$,
  from the sequence
  $$ 0 \to S \to \CC^n \tensor \calO_{Gr(b,b+c)} \to Q \to 0. $$

  Sheaf cohomology is about the (derived) pushforward of $f^*\calO(a)$ 
  to a point. We push it first to $Gr(b,b+c)$ (then from there to a point):
  \begin{eqnarray*}
    f_* f^*(\calO(a)) &\iso& \calO(a) \tensor f_*(\calO_{CX^{c\times b}}) \\
    &\iso& (Alt^b S^*)^{\tensor a}
    \tensor Sym^\bullet(((Q^*)^{\oplus n-b} \oplus S^{\oplus \bar n-c})^*) \\
    &\iso& (Alt^b S^*)^{\tensor a}
    \tensor Sym^\bullet(Q)^{\tensor n-b} \tensor Sym^\bullet(S^*)^{\tensor \bar n-c}
  \end{eqnarray*}
  We can decompose each functor $(Sym^\bullet)^{\tensor m}$ into Schur functors
  $Sc_\lambda$
  \begin{eqnarray*}
    (Sym^\bullet)^{\tensor m} 
    &\iso& \left(\Oplus_k Sc_{(k,0,\ldots,0)}\right)^{\tensor m}
    \iso \Oplus_{k_1,\ldots,k_m} \Tensor_{i=1}^m Sc_{(k_i,0,\ldots,0)}
    \iso \Oplus_{\lambda = (\lambda_1 \geq \ldots \geq\lambda_m)} 
           (Sc_\lambda)^{\oplus \dim S_\lambda(\CC^m)}
  \end{eqnarray*}
  The third isomorphism uses the Pieri rule $m-1$ times to assemble SSYT out of
  $m$ horizontal strips of various lengths $(k_i)$. The number of ways
  to achieve a particular shape $\lambda$ is the number of SSYT with
  values $\leq m$, which is the dimension of the $GL(m)$-irrep
  $S_\lambda(\CC^m)$. Now
  \begin{eqnarray*}
    f_* f^*(\calO(a)) &\iso&
  \Oplus_{\lambda = (\lambda_1 \geq \ldots \geq\lambda_{n-b}) \atop
    \mu = (\mu_1 \geq \ldots \geq\mu_{\bar n-c})}
  Sc_{(a,a,\ldots,a)}(S^*) \tensor Sc_\lambda(Q)^{\oplus \dim Sc_\lambda(\CC^{n-b})}
  \tensor Sc_\mu(S^*)^{\oplus \dim Sc_\mu(\CC^{\bar n-c})} \\
&\iso&
  \Oplus_{\lambda = (\lambda_1 \geq \ldots \geq\lambda_{n-b}) \atop
    \mu = (\mu_1 \geq \ldots \geq\mu_{\bar n-c})}
       Sc_\lambda(Q)^{\oplus \dim Sc_\lambda(\CC^{n-b})}
  \tensor Sc_{\mu+(a,\ldots,a)}(S)^{\oplus \dim Sc_\mu(\CC^{\bar n-c})} 
  \end{eqnarray*}
  If $n-b > \dim Q=c$, then $Sc_\lambda(Q)$ will vanish if $\lambda_{c+1} > 0$.
  Whereas if $n-b < c$, we can pad out $\lambda$ with $c-(n-b)$ zeroes. 
  The same remarks apply to $\bar n - c$ vs $\dim S=b$.
  So hereafter we regard $\lambda,\mu$ as sequences of length
  $c,b$ respectively. The total padding is
  $$ {\min(0,b+c-n) \atop + \min(0,b+c-\bar n)} \quad=\quad
  \begin{cases}
    0 &{\text{if }b+c \leq \min(n,\bar n) \text{ hence }a\geq 0}\\
    b+c-\min(n,\bar n) > -a&\text{if } b+c \in (\min(n,\bar n), \max(n,\bar n)] \\
    2(b+c)-N = -2a > -a &{\text{if }b+c > \max(n,\bar n) }
  \end{cases} 
  $$
  depending on whether we pad neither, one of, 
  or both of $\lambda,\mu$ with $0$s. The first and third cases will be
  the ones addressed in the proposition, as in the third case,
  $$ -a = b+c-N/2 > \max(n,\bar n)-N/2 = \max(n-N/2, \bar n-N/2) = |n-N/2|. $$

  The summand $Sc_\lambda(Q) \tensor Sc_{\mu+(a,\ldots,a)}(S^*)$ on
  $Gr(b,b+c)$ is itself the pushforward of the line bundle
  $\calO(\lambda, -(a,\ldots,a)-w_0 \mu)$ on $GL(b+c)/B$ along the
  fiber bundle $GL(b+c)/B \onto Gr(b,b+c)$. 
  Now we need to study the sequence
  $(\lambda, -(a,\ldots,a)-w_0 \mu) + \rho$ 
  (recall $\rho = (c+b,c+b-1,\ldots,3,2,1)$ from footnote \ref{foot:BWB})
  to apply Borel--Weil--Bott.

  If $a\geq 0$, then this sequence $(\lambda, -(a,\ldots,a)-w_0 \mu)$ 
  is weakly decreasing, so its line bundle is dominant and has no
  higher cohomology. 
  The first summand 
  $\calO(a) \tensor Sym^0((S^{\oplus \bar n-c} \oplus (Q^*)^{\oplus n-b})^*)
  \iso \calO(a)$ already gives us enough sections over $CX^{c\times b}$ to
  restrict to the sections we want on $X^{c\times b}$.

  If we're in the third case, so both $\lambda,\mu$ are each padded out with a 
  positive number of zeroes (indeed, strictly more than $a':=-a$ of them total), 
  then $(\lambda, -w_0 \mu)$ looks like this:
  $$ \left( 
    \mathrlap{\overbrace{\phantom{\lambda_1,\ldots,\lambda_{n-b},0,\ldots,0}}^c}
    \lambda_1,\ldots,\lambda_{n-b},
    \mathrlap{\underbrace{\phantom{0,\ldots,0,  0,\ldots,0}}_{>a'}}
    0,\ldots,0,
    \mathrlap{\overbrace{\phantom{0,\ldots,0, -\mu_{\bar n-c},\ldots,-\mu_1}}^b}
    0,\ldots,0, -\mu_{\bar n-c},\ldots,-\mu_1\right ) $$
  \junk{
    If $b+c>\max(n,\bar n)$ then $a':=-a>0$ 
    and $\lambda,\mu$ are both padded with $0$s,
    for a total of $>a'$ $0$s. Therefore that section of 
    $(\lambda, (a',\ldots,a')-w_0 \mu)$
    contains a segment of length $a'+1$ (partly in $\lambda$ and 
    partly in $w_0 \mu$) that looks like $(0, 0, \ldots, a', a')$.
  }
  Choose a segment of length $a'+1$ in that region of $0$s, 
  say in positions $k\ldots k+a'$, including
  a $0$ from each side (this is where we use $b+c > \max(n,\bar n)$)
  i.e. $k < c < c+1 < k+a'$. 
  When we add $(a',\ldots,a')$ to $-w_0 \mu$, that segment becomes
  $(\ldots,0,a',\ldots)$. Then when we add $\rho$, that segment
  becomes $(N-k+1, N-k-2, \ldots, N-k+2,N-k+1)$,
  and this repeat of $N-k+1$ (and Borel--Weil--Bott) 
  kills all the sheaf cohomology.
\end{proof}

If we factor that pushforward through the affinization 
$\mu:\ CX^{c\times b} \to \Mat_<(N)$, $(V,M)\mapsto M$ 
where 
$ M=\begin{pmatrix}  0&B&\star\\ &0&C\\ &&0\end{pmatrix} $
as in \eqref{eq:block0}, 
then we obtain a module over the ring $\CC[(B_{i,j}),(C_{j,k})]$,
namely the degree $a$ component of 
the graded ring in Equation (\ref{eq:presentation}).

In particular,
a set of $\CC[(B_{i,j}),(C_{j,k})]$-module generators of the global sections of
$\sh_r$ is formed by all the monomials of degree $a$ in the
projective coordinates $(p_s)_{s\in S}$, and using the Pl\"ucker relations
as straightening law we can already generate this module
only using weakly Bruhat-decreasing $a$-tuplets from $\subsarg{b}{b+c}$.  

\subsection{The degree of the orbital variety
(assuming hereafter that $N=2n$ and \texorpdfstring{$a\geq 0$}{a>=0})}
In the rest of this paper, we assume that $N=2n$, i.e., $\bar n=n$
and $a+b+c=n$. According to Prop.~\ref{prop:noHi}, $\sigma_r$ has no
higher sheaf cohomology, and according to \S~\ref{sec:fiber}, 
if $a<0$, $H^0$ is also zero,
which is consistent with the first case of Conj.~\ref{conj:ourconj2'}. 
We therefore also assume that $a\ge 0$.

Note that the block structure of $M$ of \eqref{eq:block0} takes the
slightly more symmetric form
\begin{equation}\label{eq:block}
M=\bordermatrix{&a+b&b+c&c+a\cr
a+b&0&B&\star\cr
b+c&&0&C\cr
c+a&&&0\cr}
\end{equation}
and Eqs.~(\ref{eq:B0},\ref{eq:C0}) become
\begin{align}
\sum_{j\in s_+} B_{i,j} p_{s_+\backslash j}&=0,\qquad s_+\in \subsarg{b+1}{b+c},\ 1\le i\le a+b\label{eq:B}
\\
\sum_{j\in\bar s_-} C_{j,k} p_{s_-\cup  j}&=0,\qquad s_-\in \subsarg{b-1}{b+c},\ 1\le k\le a+c\label{eq:C}.
\end{align}

\junk{What would be wonderful here would be for the equations above
  to be the $a=1$ case of a more general
  $\sum_{S\subseteq PP(a,b,c)} B_* \prod^a p_s = 0$ equation,
  derivable by induction on $a$ (multiply the previous equation   
  by a $p_*$ factor, then use Pl\"ucker relations to rewrite with 
  only $PP(a,b,c)$. Then we'd use these later to present the
  $\calO(a)$ module, and not have to think about weighting 
  Pl\"ucker monomials other than $PP(a,b,c)$.
}

The affinization $\mu(CX^{c\times b})$ of this conormal bundle is
the orbital variety defined by equations (\ref{eq:quad}--\ref{eq:Crank}).

For \S\ref{ssec:total} to come, we need the following result:
\begin{prop}\label{prop:orbvardeg}
  For $r=c\times b$, the degree of the affine cone $\mu(CX^{c\times b})$ is 
  $2^{bc} |PP(a,b,c)|$.
\end{prop}

\begin{proof}
  We drop the $\star$ variables, since as they are unconstrained they don't
  affect the degree. We will reduce to the $a=0$ case, where the
  rank conditions (\ref{eq:Brank}-\ref{eq:Crank}) 
  are automatically satisfied and the variety $\{(B,C)\ :\ BC=0\}$ is
  a quadratic complete intersection of degree $2^{bc} = 2^{bc} |PP(0,b,c)|$.

  For general $a$, the equations (\ref{eq:quad}-\ref{eq:Crank})
  define a quiver cycle for the quiver + dimension vector
  $\CC^{a+b} \stackrel{B}{\to} \CC^{b+c} \stackrel{C}{\to} \CC^{c+a}$,
  whose degree we compute with the ``pipe formula'' of \cite{KMS}.

  The lacing diagrams \cite[\S 3]{KMS} for this quiver cycle are
  very simple: the bottom $b$ dots in the $a+b$ stack are connected
  (noncrossingly) to some subset $s \in {[b+c]\choose b}$ of
  the dots in the $b+c$ stack, leaving the complement $\bar s$ to 
  connect to the bottom $c$ dots in the $c+a$ stack.

  We extend these partial permutations to permutations $\pi(s),\rho(s)$
  of $a+b+c$, as in \cite[\S 2.1]{KMS}:
  \begin{itemize}
  \item $\pi(s) \in S_{a+b+c}$ takes 
    \begin{itemize}
    \item $[1,b]$ to $s \subseteq [1,b+c]$
    \item $[b+1,b+a]$ to $[b+c+1,b+c+a]$, then possibly a descent,
    \item $[b+a+1,b+a+c]$ to $\bar s$, while
    \end{itemize}
  \item $\rho(s) \in S_{a+b+c}$ takes 
    \begin{itemize}
    \item $\bar s$ to $[1,c]$,
    \item $[b+c+1,b+c+a]$ to $[c+1,c+a]$, then possibly a codescent,
    \item $s$ to $[c+a+1,c+a+b]$.
    \end{itemize}
  \end{itemize}
  So far the pipe formula tells us
  \[ \deg(\mu(CX^{c\times b})) 
  = \sum_{s \in {[b+c]\choose b}}
  |\{\text{pipe dreams for }\pi(s)\}|\ |\{\text{pipe dreams for }\rho(s)\}|.
  \]
  Since $\pi(s)$ and $\rho(s)^{-1}$ are Grassmannian permutations,
  their pipe dreams are in natural correspondence with SSYT \cite{KMY}.
  The shapes of the corresponding Young diagrams are given respectively by the subsets
  $\bar s$ (viewed a $c$-subset of
  $[a+b+c]$, so that the Young diagram sits inside the $(a+b)\times c$ rectangle)
  and
  $s$ (viewed a $b$-subset of
  $[a+b+c]$, so that the Young diagram sits inside the $(a+c)\times b$ rectangle):
  \begin{align*}
    \text{pipe dreams for }\pi(s) \ &\longleftrightarrow 
  SSYT(\bar s) \text{ with entries }\leq a+b 
  \\
 \text{pipe dreams for }\rho(s) \longleftrightarrow
  \text{pipe dreams for }\rho(s)^{-1} &\longleftrightarrow
  SSYT(s) \text{ with entries }\leq a+c
\end{align*}
 (Equivalently, viewing $s$ as a Young diagram inside the $c\times b$ rectangle,
  and similarly $\bar s$ as its transpose followed by
  a complementation inside $b\times c$, then we extend them by concatenating
  them vertically with a rectangular block of size $a\times b$ for the former, 
  $a\times c$ for the latter.)

These numbers compute the dimensions of certain $GL(a+b),GL(a+c)$
  representations; if we denote by $V_{r,k}$ the representation of $GL(k)$ associated with
  the Young diagram of the subset $r$, 
  then one has the Weyl dimension formula $\dim V_{r,k} = \frac{\Delta(\bar r)}{\Delta(1,\ldots,k)}
$ where $\Delta(\cdot)$ denotes the Vandermonde determinant. Here,
\begin{align*}
\dim V_{\bar s,a+b} \dim V_{s,a+c}
&=
\frac{\Delta(s_1,\ldots,s_b,b+c+1,\ldots,b+c+a)
\Delta(\bar s_1,\ldots,\bar s_c,b+c+1,\ldots,b+c+a)}
{\Delta(1,\ldots,a+b)\Delta(1,\ldots,a+c)}
\\
&=
\frac{\Delta(s)\Delta(\bar s)
\prod_{i=1}^b \prod_{j=1}^{a} (b+c+j-s_i)
\prod_{i=1}^c \prod_{j=1}^{a} (b+c+j-\bar s_i)
}
{\Delta(1,\ldots,b)\Delta(1,\ldots,c)
\prod_{i=1}^b \prod_{j=1}^{a} (i+j-1)
\prod_{i=1}^c \prod_{j=1}^{a} (i+j-1)
}
\\
&=
|PP(a,b,c)|\ \frac{\Delta(s)\Delta(\bar s)}{\Delta(1,\ldots,b)\Delta(1,\ldots,c)}
\\
&=
|PP(a,b,c)|\ \dim V_{\bar s,b}\dim V_{s,c}
\end{align*}
Plugging into the previous formula, this gives 
\begin{align*}
  \deg(\mu(CX^{c\times b})) 
  &= \sum_{s \in {[b+c]\choose b}} \dim V_{\bar s,a+b} \dim V_{s,a+c} 
\\
  &= |PP(a,b,c)| \sum_{s \in {[b+c]\choose b}} \dim V_{\bar s,b} \dim V_{s,c}
\\
& = |PP(a,b,c)| \ 2^{bc}
\end{align*}
the last step being the $a=0$ case solved at the beginning
(or one could use the RSK* correspondence).
\end{proof}

\section{The degeneration}

Hereafter, let $A$ be the ring with generators $\{B_{i,j},C_{j,k},\star_{i,k}\}$ 
and relations (\ref{eq:quad}-\ref{eq:Crank}), 
whose $\Spec$ is the orbital variety $\mu(CX^{c\times b})$,
i.e.,\ the degree $0$ part of the homogeneous coordinate ring 
of $CX^{c\times b}$ presented in (\ref{eq:presentation}).

\newcommand\w{\mathtt{wt}}

By proposition \ref{prop:noHi}, the $K_T$-theoretic pushfoward 
$\mu_*(\sh_{c\times b}) \in K^T(\Mat_<(N))$ is the class of 
the degree $a$ component of that homogeneous coordinate ring.
This is naturally a module over the degree $0$ component $A$.
As such, if we take $F_r$ to be the {\em free}\/ module spanned by
degree $a$ monomials in the Pl\"ucker coordinates, then we have 
a short exact sequence
\[ 0 \to M_r \to F_r \to H^0\left(CX^{c\times b};\ p^*\calO(a)\right) \to 0\]
where $M_r$ gives the relations between the degree $a$ monomials.

Using the ``straightening relations'' on Pl\"ucker coordinates, 
we can shrink our generating set to the Pl\"ucker monomials 
$p_S := \prod_{i=1}^a p_{s_i}$ where $S = (s_1 \leq \cdots \leq s_a)$,
i.e. $S$ lies in $PP(a,b,c)$, giving a smaller presentation
\[ 0 \to M'_r \to F'_r \to H^0\left(CX^{c\times b};\ p^*\calO(a)\right) \to 0
\]
However, we understand the relations generating $M_r$ much better
than those generating $M'_r$, so we will need to work with both sequences.

\junk{
We treat separately the Pl\"ucker relations, which only involve 
module generators but not the variables $B_{i,j}$, $C_{j,k}$.

For a first hint at the module $H^0(CX^{c\times b};\ p^*\calO(a))$ consider its 
map to $H^0(X^{c\times b};\ p^*\calO(a)) \iso H^0(Gr(b,b+c);\ \calO(a))$,
which by Borel--Weil is the $GL(b+c)$-irrep corresponding to the 
$b\times a$ rectangle. This has a basis indexed by semistandard 
Young tableaux inside that rectangle, with entries $\leq b+c$.
}

In order to analyze $\mu_*(\sh_{c\times b})$ in more detail,
we define in this section a degeneration of $M_r$ by assigning
weights to the generators of $F_r$ and the $B,C$ variables of $A$, 
and then keeping only lowest weight terms of elements of $M_r$.
Our principal goal in the next two sections is the following theorem:

\begin{thm}\label{thm:grob}
  The l.h.s.\ of \eqref{eq:B}, \eqref{eq:C}, \eqref{eq:quad} 
  times suitable Pl\"ucker monomials 
  form a Gr\"obner basis\footnote{In proposition \ref{prop:grobner} we give a 
   foundational result that defines the sense of ``Gr\"obner basis'' used here.}
  for the $A$-submodule $M'_r \leq F'_r$, i.e.
\begin{alignat*}{2}
\init(M'_r) &=& \Bigg\langle
&\init\left(\sum_{j\in s_+} B_{i,j} p_{s_+\backslash j}p_{s_2}\ldots p_{s_a}\right),\quad s_+\in \subsarg{b+1}{b+c},\ 1\le i\le a+b,
\\
&&&\init\left(\sum_{j\in\bar s_-} C_{j,k} p_{s_-\cup  j}p_{s_2}\ldots p_{s_a}\right),\quad s_-\in \subsarg{b-1}{b+c},\ 1\le k\le a+c
\\
&&&\init\left(\sum_{j=1}^{b+c} B_{i,j} C_{j,k}p_{s_1}\ldots p_{s_a}\right),\quad 1\le i\le a+b,\ 1\le k\le a+c
\Bigg\rangle \\
&\leq&&\init(F'_r)\ :=\ \text{the free $\init(A)$-module with basis
  $\{p_S\ :\ S\in PP(a,b,c)\}$} 
\end{alignat*}
where $s_1\le\cdots\le s_a$ run over $PP(a,b,c)$.
In particular, the $\init(A)$-module $\init(F'_r)/\init(M'_r)$ has
the same $T$-equivariant Hilbert series as the $A$-module $F'_r/M'_r$.
\end{thm}

In theorem \ref{thm:leadTerms} in \S \ref{ssec:wei} we will be more
precise about the actual leading forms. The first two types of
equations will have single terms, and those of the third type
will have two terms.

We prove this in three big steps. The first (proposition \ref{prop:pluc})
is about showing that, with the right term order on the Pl\"ucker variables, 
the Pl\"ucker monomials from $PP(a,b,c)$ dominate
(essentially allowing us to consider $F'_r$ instead of $F_r$).
The second step (\S \ref{ssec:linear}-\ref{ssec:quad}) is about finding
the leading forms of the relations in theorem \ref{thm:grob},
as just described.
Then in \S \ref{sec:degen} we show that those leading terms
define a module of the correct Hilbert series (and not larger), 
i.e. that the basis is Gr\"obner.

\subsection{Pl\"ucker relations}\label{ssec:pluc}
Given a monomial $p_{s_1}\ldots p_{s_a}$,
we always order elements of each subset
increasingly: $s_\ell=\{s_{\ell,1}< \cdots < s_{\ell,b}\}$, 
effectively indexing monomials with $a\times b$
arrays of integers which are increasing along rows.
We will often use the uppercase letter as a shorthand notation for
$S := (s_1,\ldots,s_a)=(s_{\ell,m})_{1\le\ell\le a,1\le m\le b}$, 
and similarly denote $p_S := p_{s_1}\ldots p_{s_a}$. 


Given a two-dimensional array of integers $S=(s_{\ell,m})$, $\ell=1,\ldots,a$, $m=1,\ldots,b$,
with no required monotonicity property, define
\[
w(S) := \sum_{m=1}^b\sum_{\ell=1}^a 
\left(\frac{1}{2} s_{\ell,m}+\ell-m\right)^2
\]
We define the weight of a generator of $F_r$, that is a monomial of degree $a$ in the $p_s$, to be
\begin{equation}\label{eq:wei1}
\w(p_{s_1}\ldots p_{s_a})=
\max_{\sigma\in \mathcal S_a}
w(\sigma(S)),
\qquad
\sigma(S)=(s_{\sigma(1)},\ldots,s_{\sigma(a)})
\end{equation}
(we postpone to \S\ref{ssec:wei} the definition of the weights of the 
$B,C$ variables of the ring, which 
we do not need for now).
Finally, for $x$ any linear combination of such monomials, we define
$\init(x)$ to be the sum of monomials of $x$ for which the function $\w$ is {\em minimal}.

If we let $\mathcal P$ denote the space of Pl\"ucker relations, 
cf \eqref{eq:pluc}, viewed as linear forms on such monomials, namely
\[
\mathcal P=\text{span}\left(
p_{s_1}\ldots p_{s_{a-2}}\sum_{i\in s_+} p_{s_-\cup i}p_{s_+ \backslash i},
\ (s_1,\ldots,s_{a-2})\in \subs, s_\pm \in \subsarg{n\pm1}{N}\right)
\]
(assuming $a\ge2$; otherwise $\mathcal P=\{0\}$),
then we can determine its degeneration:
\begin{prop}\label{prop:pluc}
\[
\init(\mathcal P) 
= \text{span}(p_{s_1}\ldots p_{s_a},\ (s_1,\ldots,s_a)\not\in PP(a,b,c))
\]
\end{prop}
Compare with \cite[theorem 14.6]{MS-book}, whose proof we adapt here.
\begin{proof}
Introduce the Stiefel map
\[
\phi(p_s)=\det_{1\le m,m'\le b} (x_{m,s_{m'}})
\]
where the $x_{m,j}$, $1\le m\le b$, $1\le j\le b+c$ are formal variables. 
The First Fundamental Theorem of Invariant Theory says that
$\phi$ is an isomorphism from the homogeneous coordinate ring 
of the Grassmannian to the ring $\mathbb C[x_{m,j}]^{SL(b)}$
of $SL(b)$-invariants. Then

\begin{equation}\label{eq:mess}
\phi(p_{S})=
\phi(p_{s_1})\ldots \phi(p_{s_a})=\sum_{J=(j_{\ell,m})} \kappa_{J} \prod_{m=1}^b \prod_{\ell=1}^a x_{m,j_{\ell,m}}
\end{equation}
To each monomial in this expansion associate its
$J=(j_{\ell,m})$, where each column is assumed to be weakly increasing:
$j_{1,m}\le\cdots\le j_{a,m}$, $m=1,\ldots,b$, and to that its weight
\begin{equation}\label{eq:wei1b}
\wt\w \left(\prod_{m=1}^b \prod_{\ell=1}^a x_{m,j_{\ell,m}} \right)
:=w(J)
\end{equation}
By a slight abuse of notation, also write 
$\wt\w(P)$ for the minimum of $\wt\w$ over all such monomials of $P$.

We now have the key lemma: 
\begin{lem}
  The unique term in \eqref{eq:mess} with lowest weight is the product
  of diagonal terms
  \[
  \prod_{\ell=1}^a x_{m,s_{\ell,m}}
  \]
  Furthermore, $\wt\w(\phi(p_{S})) \leq \w(p_{S})$
  with equality iff $S \in PP(a,b,c)$.
\end{lem}

\begin{proof}[Proof of lemma]
From the determinant structure, we see that any monomial with array 
$J$ can be obtained from $S$ by a sequence
of two operations:
\begin{itemize}
\item Permuting each row individually, noting that the original rows are sorted:
$s_{\ell,1}<\cdots<s_{\ell,b}$, $\ell=1,\ldots,a$.
\item Reordering each column individually so that $j_{1,m}\le\cdots\le j_{a,m}$, $m=1,\ldots,b$.
\end{itemize}
Each of these two operations increases $\wt\w$. We want to show that the minimum weight is attained when
the first of them is the identity permutation.

Denote the second operation $S\mapsto S^o$. We have the diagram
\[
\begin{tikzpicture}
\node (s) at (0,0) {$S$};
\node (t) at (4,0) {$T$};
\node (so) at (0,-2) {$S^o$};
\node (to) at (4,-2) {$T^o$};
\draw[->] (s) -- node[right,align=left] {\tiny reorder\\[-1mm]\tiny columns} (so);
\draw[->] (t) -- node[right,align=left] {\tiny reorder\\[-1mm]\tiny columns} (to);
\draw[->] (s) -- node[above] {\tiny permute rows} (t);
\end{tikzpicture}
\]
where $J=T^o$, and we want to compare $w(S^o)$ and $w(T^0)$.

We compute
\begin{align*}
w(T^o)-w(S^o) 
&=
\sum_{\ell=1}^a \sum_{m=1}^b 
\left( \left(\frac{1}{2}s^o_{\ell,m}+\ell-m \right)^2
-
\left( \frac{1}{2}t^o_{\ell,m}+\ell-m\right)^2\right)
\\
&=
\sum_{\ell=1}^a\sum_{m=1}^b 
(\ell-m)(t^o_{\ell,m}-s^o_{\ell,m})
\\
&=
\sum_{\ell=1}^a\sum_{m=1}^b 
\ell(t^o_{\ell,m}-s^o_{\ell,m})
-
\sum_{m=1}^b 
m \sum_{\ell=1}^a
(t_{\ell,m}-s_{\ell,m})
\end{align*}
where in the last line we have used the fact that for each $\ell$, the $t_{\ell,m}$ are
a permutation of the $s_{\ell,m}$. 

We now proceed by induction on the sum of inversion numbers of the permutations of the rows
taking $S$ to $T$.
Suppose $T\ne S$; then there exist two successive
columns $r,r+1$ where inversions occur on some row(s). We shall show that removing the inversions
on these two columns decreases strictly the weight
(this will be effectively equivalent to showing the property in the case $b=2$).

Consider the contribution of these two columns to $w(T^o)-w(S^o)$; it takes the form
\begin{multline}\label{eq:ineq}
\sum_{\ell=1}^a\sum_{m=r,r+1}
\ell(t^o_{\ell,m}-s^o_{\ell,m})
-
\sum_{\ell=1}^a
\left(
r
\sum_{\ell=1}^a
(t_{\ell,r}-s_{\ell,r})
+
(r+1)
(t_{\ell,r+1}-s_{\ell,r+1})
\right)
\\
=
\sum_{\ell=1}^a\sum_{m=r,r+1}
\ell(t^o_{\ell,m}-s^o_{\ell,m})
+
\sum_{\ell=1}^a
(s_{\ell,r+1}-t_{\ell,r+1})
\end{multline}
The second part is easy to analyze: each inversion 
$s_{\ell,r+1}=t_{\ell,r}>t_{\ell,r+1}=s_{\ell,r}$ 
produces a strict increase of the weight by $s_{\ell,r+1}-s_{\ell,r}$. All we need to prove
is that the first part is greater or equal to zero.

We proceed by induction again, this time on $a$.
Pick among all the entries of $T$ in the columns $r,r+1$ the largest one; call it $m$. 
Pick a row on which this entry appears (since it may appear multiple times); call $m'$ the other
entry on this row at the same columns $r,r+1$. 
We may always assume, by reordering of the rows, that this row is the last one.
Apply the induction to the other rows. We have
\[
\sum_{\ell=1}^{a-1}
\sum_{m=r,r+1}
\ell (t^{o(\ell-1)}_{\ell,m}-s^{o(\ell-1)}_{\ell,m})
\ge 0
\]
where the superscript $o(\ell-1)$ means that the reordering of the rows only affects the $\ell-1$
first rows.

Now compare to $w(t^o)-w(s^o)$:
\begin{multline*}
\sum_{\ell=1}^a\sum_{m=r,r+1}
\ell(t^o_{\ell,m}-s^o_{\ell,m})
=
\sum_{\ell=1}^{a-1} \sum_{m=r,r+1}
\ell (t^{o(\ell-1)}_{\ell,m}-s^{o(\ell-1)}_{\ell,m})
\\
+\sum_{\ell:\, t_{\ell,m}>m'} (t_{\ell,m}-m')
-\sum_{\ell:\, s_{\ell,r}>m'} (s_{\ell,r}-m')
\qquad (m'=t_{\ell,m})
\end{multline*}
where the extra terms of the r.h.s.\ take into account the reordering of $m'$.
Note that $t_{\ell,m}\ge s_{\ell,r}$ for $m\in \{r,r+1\}$.
We conclude that
$\sum_{\ell=1}^a
\ell(t^o_{\ell,m}-s^o_{\ell,m})\ge 0$, which is the induction hypothesis,
and combining the inequalities for the two parts of \eqref{eq:ineq},
we obtain the first part of the lemma.

Therefore,
\begin{align}
\wt\w(\phi(p_{S})))
&=\wt\w\left( \prod_{\ell=1}^a x_{m,s_{\ell,m}} \right)
=\sum_{m=1}^b 
\max_{\sigma\in \mathcal S_a}
\left(\sum_{\ell=1}^a 
\left(\frac{1}{2}s_{\sigma(\ell),m}+\ell-m\right)^2\right)
\end{align}
where, rather than use $s^o$, we have emphasized the maximum property for each column.
This is to be compared with \eqref{eq:wei1}, where the maximum is outside the summation over $m$
(one is only allowed to permute the rows globally). This immediately implies the inequality
$\wt\w(\phi(p_{S}))\le\w(p_{S})$. In case of equality,
the ordering of each column agrees, or equivalently the $s_a$ are totally ordered, which means
$S\in PP(a,b,c)$.
%
\end{proof}

(Continuing the proof of proposition \ref{prop:pluc}.)
Using the ``straightening law'' \cite[theorem 14.6]{MS-book},
the $p_{S}$ for $S\not\in PP(a,b,c)$, can be expressed
as linear combinations of noncrossing ones modulo the Pl\"ucker relations:
\begin{equation}\label{eq:lincomb}
p_{S}=\sum_{T\in PP(a,b,c)} c_{T}\, p_{T}
\end{equation}
Now pick $T\in PP(a,b,c)$ such that $\w(p_{T})$ is minimal among the terms with nonzero coefficients in the r.h.s.\ of \eqref{eq:lincomb}. Applying $\phi$ to \eqref{eq:lincomb} and noting that diagonal terms are all distinct for distinct $T$ (see a similar argument in the proof
of \cite[Theorem 14.16]{MS-book}), we find that the coefficient of the diagonal term of $\phi(p_{T})$ 
in $\phi(\text{r.h.s.\ of \eqref{eq:lincomb}})$ 
is precisely $c_{T}$ (no compensation can occur). Therefore it must also appear in the 
l.h.s.\ (and be equal to $\pm 1$),
so that
\[
\wt\w(\phi(p_{S}))\le \wt\w(\phi(p_{T}))=\w(p_{T})
\]
\junk{(in fact, we know that this bound is saturated by the particular choice $T=S^o$) -- but is $c_{S^o}$ nonzero? anyway this remark is not needed}

We conclude that
\[
\w(p_{S})<\wt\w(\phi(p_{S}))\le \w(p_{T})
\qquad
\forall t: c_t\ne 0
\]
so that the initial term of \eqref{eq:lincomb} is $p_{S}$.

It remains to show that $\init(\mathcal P)$ is no {\em larger}\/ than 
$\text{span}(p_{s_1}\ldots p_{s_a},\ (s_1,\ldots,s_a)\not\in PP(a,b,c))$.
For this, we use the dimension count
$\dim H^0(X^{c\times b};\ \calO(a)) = |PP(a,b,c)|$ explained at the
beginning of \S \ref{ssec:pluc}, and the second statement from proposition
\ref{prop:noHi}.
\end{proof}

\junk{we could have formulated things in a more uniform way, taking each set of equations -- plucker, linear, quadratic -- and extracting ``some'' equations out of their init, w/o worrying if we got it all. but it would be a problem because some components might just be empty, which doesn't contradict dimension, and we wouldn't be able to conclude that the equations we picked were grobner.}

\junk{now we can degenerate the linear equations, and it's
much easier: we can't just kill crossing monomials (careful!), but at least we can compare them using the usual weight
function, which is simple}

Consequently, we have a smaller presentation
\[ 0 \to M'_r\to F'_r\to H^0\left(CX^{c\times b};\ p^*\calO(a)\right) \to 0\]
where
\[ F'_r:=\text{span}(p_{s_1}\ldots p_{s_a}, (s_1,\ldots,s_a)\in PP(a,b,c)). \]


\subsection{Lozenge tilings and weights}\label{ssec:wei}
We have seen in the previous section the appearance of 
the subset $PP(a,b,c)$ of increasing $a$-tuples from $\subsarg{b}{b+c}$. 
As we continue our degeneration, we shall see that it is intimately 
connected to the combinatorics of lozenge tilings, which are one possible
graphical description of elements of $PP(a,b,c)$.
As shown on the top-right picture from Figure~\ref{fig:loz},
they are fillings of a $a\times b\times c$ hexagon with lozenges (of
unit edge length) in three orientations. We shall use the following
(redundant) coordinate system on such hexagons, based on an underlying
Kagome lattice:
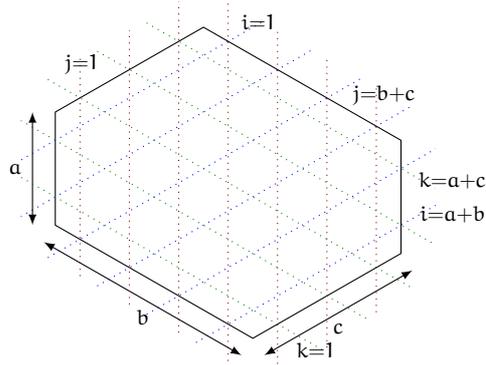
\begin{center}
\begin{tikzpicture}[scale=0.75]
\draw (0,0) ++(-30:2) coordinate (u) -- ++(0,2) coordinate (v) -- ++(30:3) coordinate (w) -- ++(-30:4) coordinate (x) -- ++(0,-2) coordinate (y) -- ++(30:-3) coordinate (z) -- cycle;
\draw[latex-latex] ([xshift=-0.4cm]u) -- node[left] {$\ss a$} ([xshift=-0.4cm]v);
\draw[latex-latex] ([shift={(-0.2,-0.3)}]u) -- node[below] {$\ss b$} ([shift={(-0.2,-0.3)}]z);
\draw[latex-latex] ([shift={(0.2,-0.3)}]z) -- node[below] {$\ss c$} ([shift={(0.2,-0.3)}]y);
\begin{scope}
\clip ([shift={(210:0.7)}]u) -- ([shift={(150:0.7)}]v) -- ([shift={(90:0.7)}]w) -- ([shift={(30:0.7)}]x) -- ([shift={(-30:0.7)}]y) -- ([shift={(-30:1)}]z) -- ([shift={(-90:0.7)}]z) -- cycle;
\foreach\i in {1,...,6} \draw[dotted,blue] (-30:\i-0.5) ++(30:-2) -- ++(30:9); 
\foreach\j in {1,...,7} \draw[dotted,red!50!black] (-30:\j+1.5) ++(0,-2) -- ++(0,10);
\foreach\k in {1,...,5} \draw[dotted,green!50!black] (30:\k-0.5) ++(-30:-2) -- ++ (-30:9);
\end{scope}
\node at (5.3,2.6) {$\ss i=1$};
\node at (8.7,-0.8) {$\ss i=a+b$};
\node at (2.2,1.85) {$\ss j=1$};
\node at (7.5,1.3) {$\ss j=b+c$};
\node at (6.3,-3.2) {$\ss k=1$};
\node at (8.7,-0.2) {$\ss k=a+c$};
\end{tikzpicture}
\end{center}
The blue (resp.\ red, green) lines are constant $i$ (resp.\ $j$, $k$) curves.
We have the relation $i-a+k-j-1/2=0$. 

\newcommand\Bloz{\tikz[scale=0.4]{\def\a{1}\def\b{1}\def\c{1}\lozX(0,0)}}
\newcommand\Cloz{\tikz[scale=0.4]{\def\a{1}\def\b{1}\def\c{1}\lozY(0,0)}}
\newcommand\BCloz{\tikz[scale=0.4]{\def\a{1}\def\b{1}\def\c{1}\lozXY(0,0)}}
Let us call lozenges \Bloz\ ``of type $B$'', \Cloz\ ``of type $C$'', 
and \BCloz\ ``of type $BC$''.
We notice that the center of a lozenge of type $B$ has integer
coordinates $(i,j)$, that of type $C$ integer coordinates $(j,k)$, and
that of type $BC$ integer coordinates $(i,k)$.  We shall always use
such coordinates for each type of lozenge. On the example of
Fig.~\ref{fig:loz}, the \Cloz\ lozenges of type $C$ have coordinates $\Y$,
listed in English-reading order.

We record a few facts about the plane partitions, hexagons, dimers, and rhombi.
One can act on a plane partition by adding or removing a box, which on
the dimer configuration corresponds to rotating a $3$-dimer hexagon by
$60^\circ$. Hence the number of dimers in each orientation is constant,
and can be computed from the case of the empty plane partition; 
there are $ac$ dimers of type $B$, $ab$ of type $C$, and $bc$ of type $BC$.
It is also easy to compute the number of hexagons:
if we replace $(a,b,c)$ by $(0,b+a,c+a)$, adding $a\choose 2$ on each side,
the region becomes a parallelogram with $(b+a-1)(c+a-1)$ hexagons.
Hence there are $(b+a-1)(c+a-1)-2{a\choose 2} = ab+ac+bc-a-b-c+1$ hexagons.

With these conventions, the bijection from lozenge tilings to subsets
in $PP(a,b,c)$ is to keep track of the $j$ coordinates of every
lozenge of type $C$.  More precisely, once we draw (as on the
lower-left of Fig.~\ref{fig:loz}) paths made of lozenges of type $B$
and $C$, then the subset $s_i$, $i=1,\ldots,a$, records the locations
of down steps of path $i$ indexed from bottom to top.

Finally we introduce one more coordinate, called $y$, which is the 
vertical coordinate of the center of any lozenge,
where the lower left corner of the hexagon has
coordinate $y=1/2$. With these conventions, the conversion from $(i,j,k)$ to $y$, effectively giving the $y$ coordinate 
of the centers of the three types of lozenges, is
\begin{align}
y^B_{i,j} &= \frac{1}{2}j-(i-a)+\frac{1}{2} \label{eq:yB}
\\
y^C_{j,k} &= k-\frac{1}{2}j \label{eq:yC}
\\
y^{BC}_{i,k} &=k-(i-a)+\frac{1}{2} \label{eq:yBC}
\end{align}

We are now ready to 
introduce the weights of the variables $B_{i,j}$ and $C_{j,k}$.
We define
\begin{align}\label{eq:wtB}
\w(B_{i,j})&=(y^B_{i,j})^2
&& 1\le i\le a+b,\ 1\le j\le b+c
\\\label{eq:wtC}
\w(C_{j,k})&=(y^C_{j,k})^2
&& 1\le j\le b+c,\ 1\le k\le a+c
\end{align}
With this we can define the Gr\"obner degeneration $\init(A)$ of
the coordinate ring $A$ of the orbital variety defined by
(\ref{eq:quad})-(\ref{eq:Crank}). We will not directly determine a
Gr\"obner basis for $A$; rather, {\bf it will be easiest to study 
  $\init(A)$ by its action on the summands of $\init(M'_r)$.} 

These weights (\ref{eq:wtB},\ref{eq:wtC}), 
combined with the definition of $\w(p_{S})$ given
in \eqref{eq:wei1}, also allow definition of the weight of an
arbitrary monomial $\prod B_{i,j} \prod C_{j,k}\,p_S$ in
$F_r$ as the sum of weights of its factors.
For $x\in F_r$ we then define the initial form $\init(x)$ 
to be the sum of monomials of $x$ for which the function $\w$ is minimal.

\begin{thm}\label{thm:leadTerms}
  For $S \in PP(a,b,c)$ with corresponding dimer configuration $D_S$, 
  and $d \in D_S$ a dimer, we obtain a generator of $\init(M'_r)$
  as follows. 
  \begin{itemize}
  \item If $d$ lies in a lozenge of type $B$ with center $(i,j)$,
    then $B_{ij} p_S \in \init(M'_r)$.
  \item If $d$ lies in a lozenge of type $C$ with center $(j,k)$,
    then $C_{jk} p_S \in \init(M'_r)$.
  \item If $d$ lies in a lozenge of type $BC$ with center $(i,k)$, \\
    then $(B_{i,i+k-a-1}C_{i+k-a-1,k} + B_{i,i+k-a}C_{i+k-a,k})p_S \in \init(M'_r)$.
  \end{itemize}
  Taken over all $S$ and $d\in D_S$, these monomial and binomial
  relations generate the $\init(A)$-submodule $\init(M'_r)$.
\end{thm}

Showing that these generators are the leading forms of 
the generators from theorem \ref{thm:grob} will occupy sections 
\S~\ref{ssec:linear}-\ref{ssec:quad}. Then showing they actually
generate will come in \S~\ref{sec:degen}.

As discussed in \S \ref{ssec:pluc}, we are mostly interested in the generators
$p_{S}$ where $S\in PP(a,b,c)$;
we note that \eqref{eq:wei1} can be rewritten, in that case, as
\begin{equation}\label{eq:wtp}
\w(p_{s_1}\ldots p_{s_a})=\sum_{\substack{\text{lozenges of type $C$}\\\text{of $S$ at $(j,k)$}}}
(y^C_{j,k})^2,
\qquad S= (s_1\le\cdots\le s_a)
\end{equation}
(note that the maximum is attained for the identity permutation, then use \eqref{eq:yC}).
This seems to break the symmetry between lozenges of type $B$ and those
of type $C$; however it is not hard to show that
\begin{equation}\label{eq:wtp2}
\w(p_{s_1}\ldots p_{s_a})=const+\sum_{\substack{\text{lozenges of type $B$}\\\text{of $S$ at $(i,j)$}}}
(y^B_{i,j})^2,\qquad S= (s_1\le\cdots\le s_a)
\end{equation}
where $const$ is an irrelevant constant (depending only on $a$ and $c-b$).
Equivalently, we have the two explicit expressions
\begin{align}\label{eq:expr1}
\w(p_{s_1}\ldots p_{s_a})&=
\sum_{i=1}^a \sum_{j=1}^b \left(\frac{1}{2}s_{i,j}+i-j\right)^2
\\\label{eq:expr2}
&=
const+\sum_{i=1}^a \sum_{k=1}^c 
\left( -\frac{1}{2}\bar s_{i,k}+i+k-\frac{1}{2}\right)^2
\end{align}

To state our main result of the next two sections, we need some
foundational results on Gr\"obner bases of modules \cite[\S 15]{Eisenbud}:

\begin{prop}\label{prop:grobner}
  Let $R \geq I$ be a polynomial ring and an ideal, with quotient $A := R/I$.
  \break
  Fix also a partial term order on $R$'s monomials with which to define
  Gr\"obner degenerations such as $\init(A) := R/\init(I)$.

  Let $P$ be an indexing set, and $M \leq A^P$ an $A$-submodule,
  with pullback
  $\begin{matrix} \wt M &\into& R^P \\ \downarrow && \downarrow\\ M&\into&A^P.
  \end{matrix} $ \\
  Then the $R$-module $R^p/\init(\wt M)$ descends to an $\init(A)$-module, 
  and any Gr\"obner basis for $\wt M \leq R^P$ descends to 
  generating sets for $M \leq A^p$
  and (stronger) for 
  $$ \init(M) := \init(A) \tensor_R \init(\wt M)
  \quad\leq\quad \init(A) \tensor_R R^P \quad\iso\quad \init(A)^P. $$
\end{prop}

\newcommand\calB{\mathcal{B}}
\begin{proof}
  The descent-of-basis claim is equivalent to the vanishing of 
  $\init(I)\tensor_R R^p/\init(\wt M)$,
  computable from $\init(I\tensor_R R^p/\wt M)
  \iso \init(I\tensor_R A^p/M) = \init(0)$. For the first claim of generation,
  a Gr\"obner basis $\calB$ for $\wt M$ generates $\wt M$, 
  hence descends to generates its quotient $M$. 
  But to be Gr\"obner, it must descend to generate $\init(\wt M)$, 
  of which $\init(M)$ is a quotient, giving the second claim of generation.
\end{proof}

With this in mind, we can speak sensibly of Gr\"obner bases for
submodules of free $A$-modules, where $A$ is presented as a quotient 
of a polynomial ring (as our $A$ is, in (\ref{eq:quad}-\ref{eq:Crank})). 

While the $B,C$ relations (\ref{eq:B}-\ref{eq:C}) define the module 
$F'_r/M'_r$, their initial forms aren't enough to generate
$\init(M'_r)$; we need the $BC$ relations from (\ref{eq:quad}) as well.

In fact, more detailed analysis (in \S \ref{ssec:quad})
will show that we only need a subset of these to produce a Gr\"obner basis,
one for each dimer of each dimer configuration.

\subsection{The linear equations}\label{ssec:linear}
We first discuss the equations
(\ref{eq:B},\ref{eq:C}), which are linear in the variables $B_{i,j}$ or
$C_{j,k}$, multiplied by ``spectator'' monomials $p_{s_2}\ldots p_{s_a}$ as in Thm.~\ref{thm:grob}.

It is perhaps instructive to consider first the special case of
$a=1$. The Pl\"ucker relations do not appear in that case, and the
linear relations are simply eqs.~(\ref{eq:B},\ref{eq:C}) without spectators.

Start with \eqref{eq:B}. The weights of its monomials are given by
\begin{align}\label{eq:wtstudy1}
\w(B_{i,j}p_{s_+\backslash j})
& = \left(\frac{1}{2}j-(i-1)+\frac{1}{2} \right)^2
+\sum_{m=1}^b \left(\frac{1}{2} ((s_+\backslash j)_m+1-m)\right)^2  
\\\notag
& = \left(\frac{1}{2}j-i+\frac{3}{2}\right)^2
+\sum_{m=1}^{h-1} \left(\frac{1}{2} (s_{+\,m}+1-m)\right)^2   && j=s_{+,h} 
\\\notag
&+\sum_{m=h+1}^{b+1} \left(\frac{1}{2} s_{+,m}+2-m\right)^2, \\\notag
&=\kappa_1 + \left( \frac{1}{2}j-i+\frac{3}{2} \right)^2
+\sum_{m=1}^{h-1}(s_{+,m}-2m+1) 
- \left(\frac{1}{2}j+2-h \right)^2 \\\notag
&=\kappa_2-(i+1/2)j+h j -\sum_{m=1}^{h-1} s_{+,m}
\end{align}
where $s_+$ is of cardinality $b+1$, and we order the elements of 
$s_+\backslash j$ as usual, $(s_+\backslash j)_1<\cdots<(s_+\backslash j)_b$.
From the third line to the fourth line, we have subtracted 
$\sum_m \left(\frac{1}{2}s_{+,m}+2-m\right)^2$, 
resulting in an irrelevant constant $\kappa_1$
(which is independent of $h$ or $j$). $\kappa_2$ is another such constant.

We claim that \eqref{eq:wtstudy1} has a unique minimum at $h=i$, i.e., $j=s_{+,i}$.
Indeed, compute the difference
\[
\w(B_{i,s_{+,h+1}}p_{s_+\backslash s_{+,h+1}})-
\w(B_{i,s_{+,h}}p_{s_+\backslash s_{+,h}})
=(h-i+1/2)(s_{+,h+1}-s_{+,h})
\]
which is negative for $h\le i-1$ and positive for $h\ge i$.
We conclude that
\begin{equation}\label{eq:degena1}
\init\left(\sum_{j\in s_+} B_{i,j} p_{s_+\backslash j}\right)
=B_{i,s_{+,i}} p_{s_+\backslash s_{+,i}}
\end{equation}

We can repeat the analysis for \eqref{eq:C}:
\begin{align}\label{eq:wtstudy2}
\w(C_{j,k}p_{s_-\cup j})
&=
(k-\frac{j}{2})^2
+
\sum_{m=1}^b (\frac{1}{2}(s_-\cup j)_m+1-m)^2
\\\notag
&=
(k-\frac{j}{2})^2
+\sum_{m=1}^{j-h} (\frac{1}{2}s_{-,m}+1-m)^2
\\\notag
&+(h-\frac{1}{2}j)^2
+\sum_{m=j-h+1}^{b-1}(\frac{1}{2}s_{-,m}-m)^2&& j=\bar s_{-,h}
\\\notag
&=\kappa_3+\frac{1}{2}j^2-j(k+h)+\sum_{m=1}^{j-h}(s_{-,m}-2m+1)
\\\notag
&=\kappa_3-\frac{1}{2}j^2+j(h-k)+\sum_{m=1}^{j-h}s_{-,m}
\end{align}
and once again
\[
\w(C_{\bar s_{-,h+1},k}p_{s_-\cup \bar s_{-,h+1}})
-
\w(C_{\bar s_{-,h},k}p_{s_-\cup \bar s_{-,h}})
=(h-k+1/2)(\bar s_{-,h+1}-\bar s_{-,h})
\]
which is negative for $h\le k-1$ and positive for $h\ge k$, so that
\begin{equation}\label{eq:degena1b}
\init\left(\sum_{j\in\bar s_-} C_{j,k} p_{s_-\cup j}\right)
=C_{\bar s_{-,k},k} p_{s_-\cup \bar s_{-,k}}
\end{equation}
Of course we could have made the reasoning even more similar to the previous one by using \eqref{eq:expr2} instead of \eqref{eq:expr1}.

Results \eqref{eq:degena1} and \eqref{eq:degena1b} both have a simple
diagrammatic interpretation. Given $s\in \subsarg{b}{b+c}$, 
we can draw the corresponding lozenge tiling:
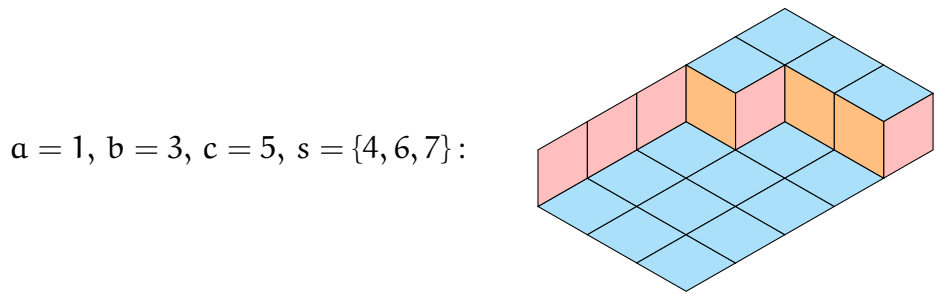
\[
\def\a{1}\def\b{3}\def\c{5}
\def\XY{(2, 1), (3, 1), (4, 1), (2, 2), (3, 2), (4, 2), (2, 3), (3, 3), (4, 3), (1, 5), (3, 4), (4, 4), (1, 6), (2, 6), (3, 6)}
\def\X{(1, 1), (1, 2), (1, 3), (2, 5), (4, 8)}
\def\Y{(4, 4), (6, 5), (7, 5)}
a=\a,\,
b=\b,\,
c=\c,\,
s=\{4,6,7\}:
\qquad
\vcenter{\hbox{\begin{tikzpicture}
\loz\XY\X\Y
\end{tikzpicture}}}
\]
There are exactly $b+c$ lozenges of type $B$ or $C$.
Consider one of these lozenges of type $B$.
Its coordinates are $(i,j)$ where $j\not\in s$ and $i$ is one plus the number of elements of $s$ less than $j$. 
(For example, the fifth lozenge of type $B$ on the example has coordinates $(2,5)$.)
Therefore,
defining $s_+=s\cup j$, one has $j=s_{+,i}$ and one can naturally associate to it the initial term of an equation as in \eqref{eq:degena1}
indexed by $s_+$ and $i$.
Similarly, to a lozenge of type $C$, with coordinates $(j,k)$, is naturally associated the initial term
of \eqref{eq:degena1b} with $s_-=s\backslash j$ and $j=\bar s_{-,k}$.

We have found that to each $s\in \subsarg{b}{b+c}$ viewed as a lozenge tiling of a $1\times b\times c$ hexagon, we can associate equations of the form $p_s B_{i,j}=0$ and $p_s C_{j,k}=0$
where $(i,j)$ (resp.\ $(j,k)$) runs over the coordinates of lozenges of type $B$ (resp.\ type $C$).
We now wish to extend this conclusion to general $a$.
The difference from the $a=1$ case is that we shall pick a subset of equations to do so (with the implicit assumption, to be proven subsequently,
that all other equations are redundant after the degeneration).

Reversing the logic we now start, for $a$ arbitrary, with an $S\in PP(a,b,c)$ viewed as a lozenge tiling, and one lozenge of type $B$ at $(i,j)$. In the NILP representation of $S$, it corresponds to a certain path labelled $\ell$ (between $1$ and $a$), with $j\in \bar s_\ell$.
We pick among Eqs.~\eqref{eq:B} the one with $s_+=s_\ell \cup j$, and multiply it by $\prod_{\ell'\ne\ell} p_{s_{\ell'}}$. We now carefully evaluate the weights of the various monomials in it.

The key observation is that from the definition~\eqref{eq:wei1} of the weight, we can bound from below
the weight of each monomial $p_{s_+\backslash j'}
\prod_{\ell'\ne\ell} p_{s_{\ell'}}$ by the expression
$w(s_1,\ldots,s_{\ell-1},s_+\backslash j',s_{\ell+1},\ldots,s_a)$ (note the ordering);
in the particular case $j'=j$, this bound is achieved because $S\in PP(a,b,c)$.
At that stage we can do the exact same calculation as in the case $a=1$;
skipping the details, we find
\begin{align*}
&w(s_1,\ldots,s_{\ell-1},s_+\backslash s_{+,h+1},s_{\ell+1},\ldots,s_a)
+\w(B_{i,s_{+,h+1}})
\\
&-
w(s_1,\ldots,s_{\ell-1},s_+\backslash s_{+,h},s_{\ell+1},\ldots,s_a)
-\w(B_{i,s_{+,h}})
\\
&=
(s_{+,h+1}-s_{+,h})(h-i+1/2+a-\ell)
\end{align*}
so that this function has a strict minimum at $h=i+a-\ell$, and we easily compute $s_{+,h}=j$, thus obtaining
\begin{equation}\label{eq:initlin1}
\init\left(\sum_{j'\in s_+} B_{i,j'}
p_{s_+\backslash j'}
\prod_{\ell'\ne\ell} p_{s_{\ell'}}
\right)
=
p_{S} B_{i,j}
\end{equation}

The exact same reasoning applies to a lozenge of type $C$ of $S$ at $(j,k)$ (on the path labelled $\ell$); computing
\begin{align*}&
w(s_1,\ldots,s_{\ell-1},s_-\cup \bar s_{-,h+1},s_{\ell+1},\ldots,s_a)
+\w(C_{\bar s_{-,h+1},k})
\\
&
-
w(s_1,\ldots,s_{\ell-1},s_-\cup \bar s_{-,h},s_{\ell+1},\ldots,s_a)
-\w(C_{\bar s_{-,h},k})
\\
&=
(\bar s_{-,h+1}-\bar s_{-,h})(h-k-1/2+\ell)
\end{align*}
leads to the initial term
\begin{equation}\label{eq:initlin2}
\init\left(\sum_{j'\in \bar s_-} C_{j',k}
p_{s_- \cup j'}
\prod_{\ell'\ne\ell} p_{s_{\ell'}}
\right)
=
p_{S} C_{j,k}
\end{equation}

\subsection{The quadratic equations}\label{ssec:quad}
These are Eqs.~\eqref{eq:quad}, which are equations of the support of $\mu_* \sh_{c\times b}$,
i.e., which are true acting on any $p_{s_1}\cdots p_{s_a}$.
Once we apply the degeneration given by weights (\ref{eq:wtB},\ref{eq:wtC}), obvious minimization of the
quadratic form in $j$ results in
\[
\init(p_{S}(BC)_{i,k})=
p_{S}\begin{cases}
B_{i,1}C_{1,k}& i+k\le a+1\\
B_{i,i+k-a-1}C_{i+k-a-1,k}+B_{i,i+k-a}C_{i+k-a,k}& a+1<i+k<n+1\\
B_{i,b+c}C_{b+c,k}&i+k\ge n+1
\end{cases}
\]

Rather than keeping all these initial terms, we shall show that many of them are redundant, i.e., can be derived from the initial terms \eqref{eq:initlin1} and \eqref{eq:initlin2} of the linear equations.
Because of Prop.~\ref{prop:pluc}, we may always assume that $S\in PP(a,b,c)$.

Let us start with the first type, that is $B_{i,1}C_{1,k}$, $i+k\leq a+1$.  
If we look at the $j=1$ slice of a lozenge tiling (i.e., the leftmost
vertical slice), we find that it always consists, from bottom to top,
of a series of \Cloz\ lozenges of type $C$, then a \BCloz\ lozenge of type $BC$,
then a series of \Bloz\ lozenges of type $B$. This means that among 
the initial terms of \eqref{eq:initlin1} and \eqref{eq:initlin2}, we have
\[
C_{1,1},\ldots,C_{1,i},B_{a-i,1},\ldots,B_{1,1}
\]
times $p_{S}$
for some $i$ between $0$ and $a$. This immediately implies that if $i+k\le a+1$, one of $B_{i,1}$ or $C_{1,k}$ is found in this list.

Similarly, one can show that the last case $B_{i,b+c}C_{b+c,k}$ is redundant because of the form of the rightmost slice $j=b+c$ of any lozenge tiling.

Finally, consider the middle case. In order to study it, it is convenient to go over to the dual picture of lozenge tilings, i.e., dimers, cf the lower right picture of Fig.~\ref{fig:loz}.
According to the previous section, each non-horizontal edge corresponds to a certain $B_{i,j}$ or $C_{j,k}$ depending on its orientation, and this edge is occupied by a dimer precisely when that variable times $p_{S}$ is the initial term of an equation. It is natural to associate to $(BC)_{i,k}$, $a+1<i+k<n+1$, the location $(i,k)$, which on this dual picture corresponds to a certain horizontal edge. Now it is easy to see that the two terms
$B_{i,i+k-a-1}C_{i+k-a-1,k}$ and $B_{i,i+k-a}C_{i+k-a,k}$ are precisely the product of the variables attached to the edges {\em adjacent}\/ to the horizontal edge at either endpoint:
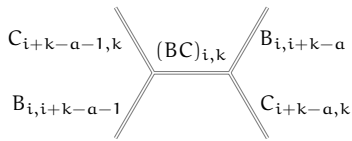
\begin{center}
\begin{tikzpicture}
\draw[emptypath] (0,0) -- node[black,above] {$\ss (BC)_{i,k}$} ++(0:1) -- node[black,right] {$\ss B_{i,i+k-a}$}  ++(60:1) (1,0) -- node[black,right] {$\ss C_{i+k-a,k}$} ++(-60:1) (0,0) -- node[black,left] {$\ss C_{i+k-a-1,k}$} ++(120:1) (0,0) -- node[black,left] {$\ss B_{i,i+k-a-1}$} ++(-120:1);
\end{tikzpicture}
\end{center}
The dimer condition means that every vertex belongs to exactly one dimer. This means that there are two scenarios:
\newcommand\dimerex{\draw[emptypath] (0,0) -- ++(0:1) -- ++(60:1) (1,0) -- ++(-60:1) (0,0) -- ++(120:1) (0,0) -- ++(-120:1);}%
\begin{itemize}
\item The horizontal edge is empty, and then at either endpoint one adjacent edge must be occupied by a dimer, e.g.,
\begin{tikzpicture}[scale=0.75,baseline=0]
\dimerex
\draw[mypath] (0,0) -- ++(120:1) (1,0) -- ++(60:1);
\end{tikzpicture}.
This immediately implies that the initial term associated to the horizontal edge is redundant.
\item The horizontal edge is occupied:
\begin{tikzpicture}[scale=0.75,baseline=0]
\dimerex
\draw[mypath] (0,0) -- (1,0);
\end{tikzpicture},
which equivalently means that there is a lozenge of type $BC$ at $(i,k)$. Then we decide to keep the corresponding initial term.
\end{itemize}
In the end, we see that a beautiful picture emerges: to {\em each}\/
lozenge of the tiling is associated exactly one initial term, either
linear in two orientations or quadratic in the last one (type $BC$).
Note in the latter type, we get {\em binomials}\/ instead of
only getting monomials, as we would in a generic degeneration. This
binomial behavior will lead to the appearance of certain toric varieties, 
as we explain now.

\section{The special fiber}\label{sec:degen}

We recall the notion of a {\em shelling}\/ of a simplicial complex and
explain what modifications to it are necessary to describe the special
fiber of our degeneration. All the definitions in the next paragraph
are standard; our reference is \cite[\S 1 and \S 13]{MS-book}.

\newcommand\Union\bigcup

A collection $\Delta \subseteq 2^V$ of subsets of a ``vertex set'' $V$ is a 
\defn{simplicial complex} if $F \in \Delta, G\subseteq F \implies G\in \Delta$.
It has a corresponding union of coordinate subspaces
\[ SR(\Delta) := \Union_{F\in \Delta} \CC^F \quad \subseteq \CC^V \]
called its \defn{(affine) Stanley--Reisner scheme}, 
and every such union $S \subseteq \CC^V$ comes from a unique simplicial complex 
$\Delta(S) := \{F \subseteq V\ :\ \CC^V \subseteq S\}$. 
The maximal elements of $\Delta$ are called its \defn{facets}; 
if they all have the same size $d+1$ then $\Delta$ is called \defn{pure
  of dimension $d$}.
A \defn{shelling} of a pure simplicial complex $\Delta$
is an ordering $F_1,\ldots,F_m$ of its facets such that
$\{G \subseteq F_i\ :\exists j<i, F_j \supset G\}$ 
is again pure and of codimension $1$ in $F_i$, for each $i$.
(Shellings only exist for nice-enough $\Delta$;
for example a union of two solid triangles at a point is an
unshellable complex.)

Let $A
:= \CC[x_v\ :\ v\in V] \big/ \langle \prod_{g\in G} x_g \rangle_{G\notin \Delta}$
be the \defn{Stanley--Reisner ring} of $\Delta$,
the coordinate ring of $SR(\Delta)$. 
Given a shelling of $\Delta$, we can associate a list of ring elements
\[ r_i := \prod \left\{ x_j \in F_i\ :\ F_i\setminus j \subseteq \Union_{j<i} F_j \right\} \]
and a filtration $M_i := \langle r_i,\ldots,r_m \rangle \leq A$ of the
regular module $M_1 = A$. Then the following is straightforward:

\begin{prop}\label{prop:shelling}
  Each summand of the $A$-module $gr\ M_1 := \bigoplus_{i\leq M} M_i/M_{i+1}$ 
  is a regular module over one component of $A$,
  i.e. $M_i/M_{i+1} = A\cdot \overline{r_i}$ and
  the annihilator ideals $ann(\overline{r_i})$ are the prime components 
  of the zero ideal.
\end{prop}

In this section we have a similar situation, but requiring 
three directions of generalization:
\begin{itemize}
\item The module $M_1$ is rank $1$ and torsion-free, but not actually free.
\item The ring $A$ also has to degenerate;
  it is not just the module $M_1$ that degenerates 
  (to its associated graded).
\item The components of the degenerate scheme $\Spec\, \init(A)$ are
  still toric varieties%
  \footnote{The equations defining toric varieties are binomial,
    but can be of high degree, and toric varieties are very rarely
    complete intersections.}
  and complete intersections, 
  but aren't quite coordinate subspaces; while some of their defining
  equations are coordinates (as in the Stanley-Reisner case), others
  are quadratic binomials (a new phenomenon).
\end{itemize}

\junk{some explanations are needed here, of what we're trying to achieve.
the word ``toric varieties'' should probably appear somewhere in this section}
For $S \in PP(a,b,c)$, let $F_S := \init(A)\cdot p_S$
be the cyclic submodule of $\init(F'_r)/\init(M'_r)$ 
generated by the element $p_S$. Since $\init(F'_r)$ is freely
generated by the $\{p_S\}$ as an $\init(A)$-module, 
the quotient $\init(F'_r)/\init(M'_r)$ is the sum
$\sum_{S\in PP(a,b,c)} F_S$. But much more is true:

\begin{thm}\label{thm:directsum}
  The $\init(A)$-module $F'_r/\init(M'_r)$ is the {\em direct}\/ sum
  $\bigoplus_{S\in PP(a,b,c)} F_S$. 
  Each $F_S$ is supported on a single component of $\Spec\, \init(A)$,
  and this gives a correspondence between $PP(a,b,c)$ and
  the components of $\Spec\, \init(A)$. 
\end{thm}

The rest of the section is devoted to its proof, 
and to determination of the individual $F_S$. 
We will then use this computation to finish the proof of theorem \ref{thm:grob},
and more importantly, to pave the way to proving (in \S \ref{sec:conclusion})
our conjectures \ref{conj:ourconj2} and \ref{conj:geomRS}
in the $c\times b$ rectangle case.

\subsection{The individual $F_S$}
Fix $S\in PP(a,b,c)$ for the rest of this subsection, which we will
usually think of as a dimer configuration
as in the Southeast picture in figure \ref{fig:loz}.
Let $H$ denote the set of hexagons in that picture, 
whose non-horizontal edges we corresponded (in \S \ref{ssec:wei})
with some of the $B,C$ ring generators. We computed $|H| = ab+ac+bc-a-b-c+1$
in \S \ref{ssec:wei}.

Equations (\ref{eq:initlin1}-\ref{eq:initlin2}) and \S \ref{ssec:quad}
show $F_S$ is a cyclic module over the ring
\[ A_S := \CC[B_{i,j},C_{j,k},\star_{i,k}] \ \bigg/\ \left\langle\ 
  \begin{matrix}
    B_{i,j} &\quad& \text{of type $B$ in $S$} \\
    C_{j,k} &\quad& \text{of type $C$ in $S$} \\
    B_{i,i+k-a-1}C_{i+k-a-1,k}+B_{i,i+k-a}C_{i+k-a,k}
    &\quad& \text{of type $BC$ in $S$} \\
  \end{matrix}
  \right\rangle
\]
who itself has $(a+b)(b+c) + (b+c)(c+a) + (c+a)(a+b)$ generators,
and one relation for each dimer in $S$ (a total of $ab+ac+bc$, 
as recorded in \S \ref{ssec:wei}).
More specifically, we have a natural map $A_S \to \init(A)/ann(F_S)$,
which we will soon show is an isomorphism. A key tool will be the following:

\begin{prop}\label{prop:hexgrading}
  Call a generator of $A_S$ \defn{relevant} if it appears on a 
  (non-horizontal) edge in the diagram of $H$ (including the half-edges
  around the boundary), but is not in $S$.
  There are $|H|+(a+c-1)+2b-ac = ab+bc+b$ relevant $B$ generators,
  and similarly $ac+bc+c$ relevant $C$ generators.
  Let $A'_S$ be the subring generated by those,
  suffering the $bc$ many quadratic binomial equations
  from the third group above.
  \junk{
    The ring $A_S$ has two kinds of irrelevant generators: $B,C$ variables
    that don't correspond to dimers in the $H$ region, and $\star$ variables.
    There are $b(b-1)$ such $B$ variables, 
    $c(c-1)$ such $C$ variables, and all $(c+a)(a+b)$ many $\star$ variables.
    Let $A'_S$ be the subring without them,
    leaving\footnote{It's no surprise that these numbers are close to
      the number of hexagons, since each hexagon has two $B$ dimers and
      {\em most}\/ of the $B$ dimers are in two hexagons; same for $C$.}
    only $ab+ac+bc+b$ many $B$ variables and $ab+ac+bc+c$ many $C$ variables.
    The ideal contains $ac$ and $ab$ of those, respectively,
    leaving $ab+ac+2bc+b+c$ nontrivial variables suffering $bc$ many 
    quadratic binomial equations.}
  
  Let $H_+$ be $H$ plus the partial hexagons around the outside,
  and $L = \ZZ^{H_+}$, the space of $\ZZ$-valued functions on $H_+$.
  Then $A'_S$ has an $L$-grading, where the $B$ or $C$ generator
  corresponding to a non-horizontal edge $\gamma$ is given weight 
  $f_\gamma \in L$, defined by
  \[ f_\gamma(h) =
  \begin{cases}
    +1 &\text{if $h$ is the (partial) hexagon above $\gamma$} \\
    -1 &\text{if $h$ is the (partial) hexagon below $\gamma$} \\
    0 &\text{otherwise.}
  \end{cases}
  \]
  Moreover, this grading is \defn{fine,} meaning that its homogeneous
  components are $1$-dimensional.
\end{prop}

\begin{proof}
  Each hexagon in $H$ has a type $B$ edge on its Northwest side, 
  which gives all the $B$ edges except for the $a+c-1$ of them on the 
  East and Southeast,
  and $2b$ one-ended edges coming off the Northeast and Southwest sides.
  Of those, there are $ac$ in $S$ we must remove.
  Flip left/right for the type $C$ edges statement.
  \junk{
    First we count the irrelevant $B$ and $C$ generators. 
    Observe that the $(a,b,c)$ region in the Southeast of figure \ref{fig:loz} 
    (a hexagon of hexagons) is the intersection of the corresponding 
    $(a+b,0,c+b)$ and $(a+c,b+c,0)$ regions (each being vertical
    parallelograms of hexagons). Those two larger regions contain 
    dimers corresponding to {\em all}\/ the $B$ and $C$ generators, and they 
    add $2{b\choose 2}$, $2{c\choose 2}$ many such dimers, respectively.
  }
  
  We have to check that the quadratic binomial relation
  coming from a horizontal edge is $L$-homogeneous:
  \[
  \tikz[scale=0.65,baseline=-3pt]{\dimerex\node at (0.5,0.5) {\Large$+$};\node at (0.5,-0.5) {\Large$-$};}
  \quad=\quad
  \tikz[scale=0.65,baseline=-3pt]{\dimerex\node at (0.5,0.5) {\Large$+$};\node at (1.8,0) {\Large$-$};}
  \quad+\quad 
  \tikz[scale=0.65,baseline=-3pt]{\dimerex\node at (1.8,0) {\Large$+$};\node at (0.5,-0.5) {\Large$-$};}
  \quad=\quad 
  \tikz[scale=0.65,baseline=-3pt]{\dimerex\node at (0.5,0.5) {\Large$+$}; \node at(-0.8,0) {\Large$-$};}
  \quad+\quad 
  \tikz[scale=0.65,baseline=-3pt]{\dimerex\node at(-0.8,0) {\Large$+$};\node at (0.5,-0.5) {\Large$-$};}
  \]

  The exponent vector of a Laurent monomial $m$ is 
  a $\ZZ$-valued function $g$ on $H_+$'s nonhorizontal edges, vanishing on $S$
  (and on the half-edges around the boundary).
  Each horizontal dimer $\gamma \in S$ gives a relation of the form
  $ B_{i,i+k-a-1}C_{i+k-a-1,k}/B_{i,i+k-a}C_{i+k-a,k} = -1 $,
  whose exponent vector we call $r_\gamma$.
  To show the grading is fine (as could have been predicted 
  from \cite{EisenbudSturmfels}), we need to show that 
  $$ f := \sum_{\text{edges }\gamma} g(\gamma) f_\gamma = 0 $$
  implies that $g$ is in the span of the $(r_\gamma:\ \gamma \in S$
  horizontal$)$.

  We can use the $(r_\gamma)$ to modify $g$ as follows, from left to right: 
  for each horizontal $\gamma \in S$, connected to some $B$-edge $\gamma'$ 
  on its West side (necessarily not in $S$), add $g(\gamma') r_\gamma$ to $g$
  thereby making the new value at $\gamma'$ be $0$.
  Now we can assume that $g$ vanishes not only on $S$, but on each $B$-edge
  to the Southwest of a horizontal $S$-edge. 
  In short, if a $B$-edge is has nonzero $g$-value, then the $C$-edge
  below it has zero $g$-value (since it's in $S$), unless both are
  on the right boundary of $H$ (so don't have a full horizontal edge
  giving an $r_\gamma$ to use).

  The displayed equation above says that in each partial hexagon $h\in H_+$, 
  the sum of the NW and NE $g$-values equals the sum of the SW and SE.
  Each partial hexagon $h$ on the West side of $H_+$ (there are initially
  $a+1$ of them) has only
  NE and SE contributors to $f(h)$, and we've just shown that one of them
  vanishes. Hence the other one does too. Rip off this vertical line
  of (initially $a$ many) hexagons from the West side and repeat the argument.
  The new shape will again be a hexagon of hexagons, albeit not with
  edge-lengths $(a,b,c,a,b,c)$, but that doesn't affect the argument.
\end{proof}

In particular,
\begin{eqnarray*}
  \dim A'_S &\geq& \#\text{generators} - \#\text{relations} \\
  &=& (ab+bc+b) + (ac+bc+c) - bc \quad=\quad ab+ac+bc+b+c
\end{eqnarray*}
with equality iff it is a complete intersection of those $bc$
quadric hypersurfaces. So now we compute its dimension (to show that
this inequality is indeed strict).

\junk{
\rem{Really, we should use some generic functional on $L$ to define
  the weighting degenerating $A$ and $F'_r$, so that we can 
  guarantee in advance that all $\init(A)$ is $L$-graded, not just 
  these related rings $A_S$.}

\rem{Guesses: (1) $wt(p_S) := \sum_{\gamma\in S} f_\gamma$. For $S$ the 
  empty plane partition, this is $-1$ at the corner and $0$ elsewhere.
  Adding a box changes the local values by 
  $\begin{matrix}    &+& \\ -&&- \\ &+2& \\ -&&- \\ &+&  \end{matrix}$.
  (But what should it be for $p_S, S\notin PP(a,b,c)$?)
  (2) The functional is any function on hexagons that's sufficiently
  convex in the up/down direction; just a hunch.}
}

Let $L_S \leq L$ be the set of actual $L$-gradings of monomials 
occurring in $A'_S$.
To analyze $L_S$, we first consider its perp (with respect to the dot product 
on $L$ where the hexagon vectors $\{\vec v_h\ :\ h\in H_+\}$ are orthonormal). 

\begin{prop}\label{prop:LS}
  \begin{enumerate}
  \item $\dim L = ab+ac+bc+a+b+c+1$.
  \item $L_S$ is a cone (i.e. closed under $+$), 
    and $A'_S$ is its monoid algebra, a domain of dimension $\dim L_S$.
  \item $f\in L_S^\perp$ iff it is constant on the
    regions of $S \cup \{$all horizontal edges$\}$.
  \item Hence $\dim L_S^\perp = a+1$ and $\dim L_S = ab+ac+bc+b+c.$
  \item $A'_S$ and $A_S$ are complete intersections, of degree $2^{bc}$.
  \end{enumerate}
\end{prop}

\begin{proof}
  \begin{enumerate}
  \item 
    The number of partial hexagons in the $H_+$ for $(a,b,c)$
    is the number of hexagons that $H$ would have for $(a+1,b+1,c+1)$, 
    namely $(a+1)(b+1)+(a+1)(c+1)+(b+1)(c+1)-(a+1)-(b+1)-(c+1)+1$
    or $ab+ac+bc+a+b+c+1$.
  \item 
    The binomial relations defining $A'_S$ don't involve any
    of the variables being killed in the linear relations. 
    So any monomial $m$ in the non-killed variables must be nonzero in $A'_S$,
    since each binomial relation just lets us rewrite $m$ as 
    another monomial (up to sign). Hence if $f_1,f_2 \in L$ come from
    nonzero monomials $m_1,m_2$, then $m_1 m_2$ is nonzero also and
    has $L$-grading $f_1+f_2$.

    By \cite{EisenbudSturmfels}, if $A'_S$ were not a domain then it
    would satisfy a relation $m(p-q)=0$ where $m,p,q$ are monomials.
    But then $mp,mq$ would be monomials with the same $L$-grading,
    as was just now forbidden in proposition \ref{prop:hexgrading}.
    The dimension statement is standard in toric geometry \cite{Fulton-TVs}.
  \item Each non-horizontal dimer $\gamma$ {\bf not} in $S$ gives a
    nonvanishing generator of $A_S$, hence a vector $f_\gamma \in L$.
    To be perpendicular to that, $f\in L_S^\perp$ needs to take on the
    same value on the two hexagons that would have been separated by $\gamma$.
  \item 
  Looking at the bottom row of Fig.~\ref{fig:loz}, we notice that including
  all horizontal edges has the same topological effect as contracting
  them all to produce a NILP configuration out of a dimer configuration. 
  In particular, since there are $a$ NILPs,
  the number of regions separated by occupied edges is obviously $a+1$.
  Then subtract this from $\dim L$ computed in (1).
  \item The dimension of $A'_S$ is at least $ ab+ac+bc+b+c $,
    as we checked just before the proposition.
    Since by parts (1,4) that is the actual dimension of $A'_S$,
    it is a complete intersection,
    so its degree is the product of the degrees of its defining equations.
    There is one degree $2$ equation for each horizontal dimer,
    of which there are $bc$ (as seen in \S \ref{ssec:wei}).
    Since $A_S$ is $A'_S$ tensor a polynomial ring, it too is 
    a complete intersection of this degree $2^{bc}$.
  \end{enumerate}
  \vskip -.2in
\end{proof}

\subsection{The total initial module}\label{ssec:total}
So far we have a surjection
\begin{align*}
  \bigoplus_{S\in PP(a,b,c)} A_S &\onto\ \init(F'_r)/\init(M'_r) \\
  (a_S \in A_S:\ S\in PP(a,b,c)) &\mapsto \sum_{S\in PP(a,b,c)} a_S p_S
\end{align*}
that we need to show is an isomorphism. 
We will use a module-theoretic
extension of the argument in \cite[lemma 1.7.5]{KM-Schubert}: 

\begin{lem}\label{lem:degreecheck}
  Both sides of the surjection above are graded modules over 
  the polynomial ring $\CC[B_{i,j},C_{j,k},\star_{i,k}]$. 
  Hence we can speak of their degrees, 
  and since their supports are of the same dimension (namely $ab+ac+bc$), 
  the surjection gives an inequality on degrees.
  If the map is not an isomorphism, then this inequality
  on degrees is strict.
\end{lem}

\begin{proof}
  As with (quotients by) ideals, the degree of a graded module $M$ is
  defined as the leading coefficient of its
  Hilbert polynomial (times $\dim(supp(M))!$). 
  There is a minor annoyance that the map
  defined above only becomes a graded map if we shift the 
  grading on each $A_S$ summand by $\deg(p_S)$, but shifting the argument
  of the Hilbert polynomial doesn't change the leading term,
  so we can ignore this subtlety.

  If this map has an element $(k_S \in A_S)_{S\in PP(a,b,c)} \neq 0$
  in its kernel $K$, then the kernel contains the module $\bigoplus A_S k_S$.
  Since each $A_S$ is a domain by proposition \ref{prop:LS}, 
  this submodule of the kernel again has the dimension $ab+ac+bc$, so
  $\deg(RHS) = \deg(LHS) - \deg(K) < \deg(LHS)$.
\end{proof}

\begin{proof}[Proofs of theorems \ref{thm:directsum}, \ref{thm:grob},
  and \ref{thm:leadTerms}.]
  We already knew that $F'_r/in(M'_r)$ is generated by $(F_S)$,
  i.e. is a quotient of $\bigoplus F_S$, and that each $F_S$
  is supported only on the component $\Spec A_S$ of $\Spec A$.

  By proposition \ref{prop:orbvardeg}, the degree of $A_S$,
  hence also of its free rank $1$ module $F_S$, is $2^{bc}$.
  So $\bigoplus F_S$ has degree $2^{bc} |PP(a,b,c)|$.

  Meanwhile its quotient $F'_r/in(M'_r)$ is a degeneration of a
  rank $1$ sheaf over $\Spec A$, 
  which we calculated to be $2^{bc} |PP(a,b,c)|$ 
  in proposition \ref{prop:orbvardeg}. 
  Now we use lemma \ref{lem:degreecheck} to know that the
  map $\bigoplus F_S \onto F'_r/in(M'_r)$ is an isomorphism.
  This completes the proof of theorem \ref{thm:directsum}.

  To establish the Gr\"obnerness of theorems \ref{thm:grob}
  and \ref{thm:leadTerms}, we need to
  know that the family doesn't have any components supported only
  on the special fiber. But we've determined the components $(F_S)$ of the
  special fiber, and if we cut any of them down the degree would decrease
  below that of the general fiber.
\end{proof}

\section{Conclusion}\label{sec:conclusion}
We now use the results of the last two sections, in particular
Theorems~\ref{thm:grob}, \ref{thm:directsum} and Proposition~\ref{prop:LS}, 
to prove the remaining two Conjectures \ref{conj:ourconj2} and
\ref{conj:geomRS} (or more precisely, the equivalent
Conj.~\ref{conj:ourconj2'} and \ref{conj:geomRS'}), in the $(a,b,c)$ case.

First, we consider Conj.~\ref{conj:geomRS'}.

\subsection{Polynomiality and Conjecture~\ref{conj:geomRS'}}
According to Theorem~\ref{thm:grob}, if we wish to compute the 
Hilbert series of $\mu_* \sh_{c\times b}$,
we can instead use that of the degenerated $A$-module $F'_r/M'_r$. 
The latter, according to Theorem~\ref{thm:directsum}, is a direct sum
$\bigoplus_{S\in PP(a,b,c)} F_S$ where each $F_S$ is free of rank $1$
over a complete intersection $A_S$.

As a result, we have the following formula:
\begin{multline}\label{eq:final}
\mu_*[\sh_{c\times b}]=
\prod_{\substack{1\le i<j\le a+b\\\text{or}\\a+b+1\le i<j\le a+2b+c\\\text{or}\\a+2b+c+1\le i<j\le 2n}}
(1-t\, z_i/z_j)\ 
\sum_{S\in PP(a,b,c)}
\prod_{\substack{i\in s_\ell,\\1\le \ell\le a}} z_i^{-1}
\\
\prod_{\substack{\text{lozenges $(i,k)$}\\\text{of type $BC$ of $S$}}} (1-t^2z_i/z_{k+a+2b+c})
\prod_{\substack{\text{lozenges $(i,j)$}\\\text{of type $B$ of $S$}}} (1-t\,z_i/z_{j+a+b})
\prod_{\substack{\text{lozenges $(j,k)$}\\\text{of type $C$ of $S$}}} (1-t\,z_{j+a+b}/z_{k+a+2b+c})
\end{multline}
where the prefactors correspond to the $0$'s of $M$, cf~\eqref{eq:block},
the monomial is the weight of the generator $p_S$ of $F_S$, and the other
factors come from the various equations of $A_S$.

Now we recall Thm.~\ref{thm:fpl-pp},
namely that lozenge tilings of a $a\times b\times c$ hexagon 
are in bijection with FPLs with connectivity $(a,b,c)$.
Therefore, \eqref{eq:final} is of the form conjectured in Conj.~\ref{conj:geomRS'}.

In order for the monomials $m_f$ to have the right form, we need to have
\begin{equation}\label{eq:prefac'}
\tilde m_r^{-1}=t^{-bc/2}
\prod_{i=1}^{a+b} z_i^{i-1}
\ 
\prod_{j=1}^{b+c} z_{j+a+b}^{a+j-1}
\ 
\prod_{k=1}^{a+c} z_{k+a+2b+c}^{b+k-1}
\end{equation}
and carefully rearranging the monomials, we find
\begin{multline}\label{eq:final2}
\tilde m_r^{-1}
\mu_*[\sh_{c\times b}]=
\prod_{\substack{1\le i<j\le a+b\\\text{or}\\a+b+1\le i<j\le a+2b+c\\\text{or}\\a+2b+c+1\le i<j\le 2n}}
(z_j-t\, z_i)\ 
\sum_{S\in PP(a,b,c)}
\prod_{\substack{\text{lozenges $(i,k)$}\\\text{of type $BC$ of $S$}}} t^{-1/2}(z_{k+a+2b+c}-t^2z_i)
\\[-4mm]
\prod_{\substack{\text{lozenges $(i,j)$}\\\text{of type $B$ of $S$}}} (z_{j+a+b}-t\,z_i)
\prod_{\substack{\text{lozenges $(j,k)$}\\\text{of type $C$ of $S$}}} (z_{k+a+2b+c}-t\,z_{j+a+b})
\end{multline}
which fits precisely with the form of the monomials
in Conj.~\ref{conj:geomRS}.

We also find
\begin{cor}\label{cor:facpol}
\[
\tilde m_r^{-1}\mu_*[\sh_{c\times b}]=
\prod_{\substack{1\le i<j\le a+b\\\text{or}\\a+b+1\le i<j\le a+2b+c\\\text{or}\\a+2b+c+1\le i<j\le 2n}}
(t^{-1/2}z_j-t^{1/2} z_i)\ 
\Phi_r
\]
where $\Phi_r$ is a symmetric polynomial in the $\{z_1,\ldots,z_{a+b}\}$, the $\{z_{a+b+1},\ldots,z_{a+2b+c}\}$, and the $\{z_{a+2b+c+1},\ldots,z_{2n}\}$.
\end{cor}
\begin{proof} (compare with \cite[Theorem 1]{artic31}).
The polynomiality (as opposed to Laurent polynomiality)
is explicit in \eqref{eq:final2}. The symmetry stems from the fact that $\pi_*[\sh_{c\times b}]$, the pushforward to a point
of a $GL(a+b)\times GL(b+c)\times GL(c+a)$-equivariant sheaf, possesses the required symmetry, the prefactor coming from 
\eqref{eq:pi2mu} and \eqref{eq:prefac'}.
\end{proof}

Now we move on to Conj.~\ref{conj:ourconj2}.
It should be noted that general equivariant localization arguments allow to show
that Conj.~\ref{conj:ourconj1} implies Conj.~\ref{conj:ourconj2} without any need to use Gr\"obner degenerations. However, in the case $(a,b,c)$,
is is simpler to prove Conj.~\ref{conj:ourconj2'}
directly using the degeneration, as we show now.

\subsection{Wheel condition, specialization and Conjecture~\ref{conj:ourconj2'}}
Given $1\le i_1<i_2<i_3\le N$, we investigate the specialization $z_{i_\ell}= t^{\ell-1}z$, $\ell=1,2,3$, of $\tilde m_r^{-1}\mu_*[\sh_r]$.
If two of the indices $i_1,i_2,i_3$ fall in the same interval
$\{1,\ldots,a+b\}$, $\{a+b+1,\ldots,a+2b+c\}$, $\{a+2b+c+1,\ldots,2n\}$,
then according to Corollary~\ref{cor:facpol}, the prefactor of 
$\tilde m_r^{-1}\mu_*[\sh_r]$ vanishes. Now assume $i_1=i$, $i_2=j+a+b$, $i_3=a+2b+c+k$, $1\le i\le a+b$, $1\le j\le b+c$, $1\le k\le a+c$.
With the notation of Cor.~\ref{cor:facpol}, due to the symmetry of $\Phi_r$, we can choose $i,j,k$ as we wish and
we may as well assume $j=i+k-a-1$.
Now in each plane partition appearing in \eqref{eq:final2}, there must be a lozenge of type $B$ at $(i,j)$, of type $C$ at $(j,k)$ or of type
$BC$ at $(i,k)$
(corresponding to the fact that one of the equations $B_{i,j}=0$, $C_{j,k}=0$ or $(BC)_{i,k}=0$ must be satisfied). Therefore $\Phi_r$ vanishes,
and we conclude that $\tilde m_r^{-1} \mu_*[\sh_r]$ satisfies the wheel condition of Thm.~\ref{thm:wheel}.

Next we investigate the ``dual basis'' specializations of the same theorem. 
Pick $s\in\LPN$ and set $z_i=t^{\pm 1/2}$ depending on whether $i\in s$ or not.
Because of the same prefactors in Cor.~\ref{cor:facpol}, if a $t^{-1/2}$ occurs
before a $t^{1/2}$ in each of the three same intervals mentioned above, the specialization is zero. Furthermore, $s\in\LPN$ implies that all 
these specializations
satisfy the ``Dyck path'' condition that in any sequence $(z_1,\ldots,z_\ell)$, $\ell=1,\ldots,N$, there must be more
$t^{-1/2}$ than $t^{1/2}$. This leaves the unique possibility 
\[
(z_1,\ldots,z_N)=(
\underbrace{t^{-1/2},\ldots,t^{-1/2}}_{a+b},
\underbrace{t^{1/2},\ldots,t^{1/2}}_{b},
\underbrace{t^{-1/2},\ldots,t^{-1/2}}_{c},
\underbrace{t^{1/2},\ldots,t^{1/2}}_{a+c}
)
\]
which is exactly the case $s=r$.
In the sum of \eqref{eq:final2}, a single term survives, corresponding to the ``full'' lozenge tiling of the type
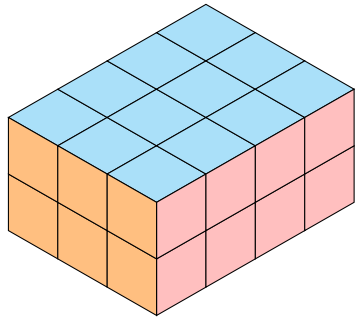
\[
\def\a{2}\def\b{3}\def\c{4}
\def\pp{{2, 2, 2}, {2, 2, 2}, {2, 2, 2}, {2, 2, 2}}
\let\mymatrixcontent\empty
  \foreach \row in \pp{
    \foreach \h in \row {%
        \expandafter\gappto\expandafter\mymatrixcontent\expandafter{\h \&}%
      }%
    \gappto\mymatrixcontent{\\}%
  }
\def\XY{(1, 3), (2, 3), (3, 3), (1, 4), (2, 4), (3, 4), (1, 5), (2, 5), (3, 5), (1, 6), (2, 6), (3, 6)}
\def\X{(5, 4), (5, 5), (5, 6), (5, 7), (4, 4), (4, 5), (4, 6), (4, 7)}
\def\Y{(1, 1), (2, 1), (3, 1), (1, 2), (2, 2), (3, 2)}
\begin{tikzpicture}
\loz\XY\X\Y
\end{tikzpicture}
\]
and we compute, after cancellations of various powers of $t^{1/2}$:
\[
\tilde m_r^{-1} \mu_* [\sh_r]|_{\text{specialization above}} = (1-t)^{n(n-1)}(t^{1/2}+t^{-1/2})^{bc}
\]
which means $\mu_* [\sh_r]=
(1-t)^{n(n-1)} \tilde m_r\Psi_r$, thus proving Conj.~\ref{conj:ourconj2'}.

\junk{in principle these properties might've been obtained {\em before}
degenerating by using geometric vertex decomposition to write recurrence
relations, but not sure how to apply it to sheaves\dots
or by showing that conj 1 implies conj 2 by localization arguments.
}


\gdef\MRshorten#1 #2MRend{#1}%
\gdef\MRfirsttwo#1#2{\if#1M%
MR\else MR#1#2\fi}
\def\MRfix#1{\MRshorten\MRfirsttwo#1 MRend}
\renewcommand\MR[1]{\relax\ifhmode\unskip\spacefactor3000 \space\fi
\MRhref{\MRfix{#1}}{{\scriptsize \MRfix{#1}}}}
\renewcommand{\MRhref}[2]{%
\href{http://www.ams.org/mathscinet-getitem?mr=#1}{#2}}
\bibliographystyle{amsplainhyper}
\bibliography{biblio}
\end{document}